\documentclass[11pt,reqno]{amsart} 
\usepackage[left=2.5cm,right=2.5cm,bottom=3cm]{geometry}
\usepackage{amsfonts}
\usepackage{amsmath}
\usepackage{amssymb}
\usepackage{mathtools}
\usepackage{amsthm}
\usepackage{pgfplots}
\pgfplotsset{compat=newest}
\usepackage[shortlabels]{enumitem}
\usepackage{tikz}   
\usepackage{tikz-cd}  
\usepackage{changepage}
\usetikzlibrary{arrows}
\usepackage{bm}
\usepackage{graphicx}
\usepackage{cases}
\usepackage{appendix}
\usepackage{esint}
\usepackage{hyperref}
\usepackage{mathrsfs}
\usepackage{mathrsfs}
\numberwithin{equation}{section}

\theoremstyle{definition}
\newtheorem{thm}{Theorem}
\newtheorem*{claim}{Claim}
\newtheorem{defn}[thm]{Definition}
\newtheorem{lem}[thm]{Lemma}
\newtheorem{prop}[thm]{Proposition}
\newtheorem{cor}[thm]{Corollary}
\newtheorem*{rmk}{Remark}

\numberwithin{thm}{section}

\newcommand{\R}{\mathbb{R}}  
\newcommand{\N}{\mathbb{N}}  
\newcommand{\CC}{\mathbb{C}}   
\newcommand{\Sp}{\mathbb{S}}
\newcommand{\p}{\partial}  

\newcommand{\dif}{\,\textup{d}} 

\newcommand{\Hau}{\mathcal{H}} 
\newcommand{\dist}{\textup{dist}}

\newcommand{\sing}{\operatorname{sing}}
\newcommand{\reg}{\operatorname{reg}}

\newcommand{\loc}{\textup{loc}}
\newcommand{\ind}{\operatorname{index}}

\usepackage{xcolor}
\usepackage{listings}
\begin{document}

\title[Generic regularity of harmonic maps]{Isolated singularities of harmonic maps with generic boundary data}
\author{Xuanyu Li}
\address{Department of Mathematics, Cornell University, Ithaca, NY 14853, USA}
\email{xl896@cornell.edu}
\begin{abstract}
    We study the behavior of isolated singularities of stationary harmonic maps with generic boundary data. For round sphere targets, we prove that, for $4\leqslant n\leqslant7$, every stable stationary harmonic map from a bounded smooth $n$-dimensional domain to round $(n-1)$-sphere with generic smooth boundary data has only radial projections composed with orthogonal transformations as tangent maps at its singularities; for 7-dimensional domains and round $k$-sphere targets with $k\geqslant7$, every stable stationary harmonic map with generic smooth boundary data is smooth.

    These results follow from a general minimum-index principle for closed real-analytic target manifolds. Under suitable target hypotheses ensuring strong compactness and excluding lower-dimensional singularity models, we show that, in the first domain dimension in which singularities can occur, generic boundary data force the link of every singular tangent map to attain the smallest possible Morse index. This principle applies both to stable stationary harmonic maps and, under the corresponding stronger target hypotheses, to all stationary harmonic maps. If no non-constant tangent map in the relevant class has a link attaining this minimum, generic smoothness follows.

    As an application, we establish a rigidity at infinity theorem. For $3\leqslant n\leqslant7$, we show that any energy-minimizing map from $n$-Euclidean space to round $(n-1)$-sphere admitting the radial projection as a blow-down limit must itself be a translate of the radial projection.
\end{abstract}
\maketitle

\section{Introduction}
In this paper, we are interested in the generic regularity theory of harmonic maps. Let $M^m,m\geqslant 2$ be a closed analytic Riemannian manifold and let $\Omega\subset \R^n$ be a bounded smooth domain. For simplicity we assume $M$ is isometrically embedded in a sufficiently high dimensional Euclidean space $\R^N$. \textbf{Stationary harmonic maps} are critical points of the energy functional
$$E(u)=\frac{1}{2}\int_{\Omega}\vert\nabla u\vert^2\dif x,u\in H^1(\Omega,M).$$
We refer readers to Section \ref{s: prelim} for the precise definition of stationary harmonic maps.

As solutions of a non-linear variational problem, the stationary harmonic maps may have singularities, at which they fail to be continuous. This is still the case even if stronger assumptions such as stability or energy minimality are assumed. A model example is the radial projection map $$x\mapsto\frac{x}{\vert x\vert},\R^n\rightarrow\Sp^{n-1},$$
where $\Sp^{n-1}$ is the unit sphere in $\R^n$. This map minimizes the energy functional among maps with the same boundary trace for all $n\geqslant3$ as shown by F. Lin \cite{LinMinimizing}; see also \cite{JagerKaulMinimizing,BrezisCoronLieb,CoronGulliver}. Every sufficiently small $C^0$ perturbation of its boundary value in $B_1$ has topological degree one and therefore admits no continuous extension to the closed ball. This singularity is indeed stable under boundary perturbation as studied by R. Hardt and F. Lin \cite{HardtLinStability}; see also \cite{Mazowieckaetc}. Upon finishing this manuscript, the author learnt the work of D. Gutwein and T. Langlais \cite{gutwein2026deformations}, who found a criterion of the tangent maps for the isolated singularities of a stationary harmonic map to be persistent under metric perturbation.

By contrast, the equator map 
$$x\mapsto\left(\frac{x}{\vert x\vert},0\right),\R^n\rightarrow\Sp^n,$$
is energy minimizing for $n\geqslant7$ by R. Schoen and K. Uhlenbeck \cite{SchoenUhlenbeck} but its singularity can be removed by suitable perturbations of the boundary data into an open hemisphere. We refer readers to R. McIntosh and L. Simon \cite{McIntoshSimon} and C. Wang \cite{Wangp-harmonicmaps} for further analysis of this phenomenon. Thus, energy minimality alone does not determine whether a singularity survives perturbation.

Motivated by these examples and results, there is a longstanding open question: rather than following a particular singular solution, can one restrict the singularities of all solutions with a generic boundary value?

Our principal conclusion is that, under the hypotheses below, generic boundary data force every singular tangent link to have the smallest possible Morse index.

\subsection{Main results for sphere targets}

 For sphere targets, our general principle yields a classification of the remaining tangent maps in the first singular dimensions and generic smoothness in dimension seven. In the sequel, a subset of $C^\infty(\partial\Omega,M)$ is called generic if it contains a countable intersection of open dense subsets in the $C^\infty$ topology.

\begin{thm}\label{thm: generic regularity for sphere valued maps}
    Let $\Omega\subset\R^n$ be a bounded smooth domain. 
    \begin{enumerate}
        \item Suppose $4\leqslant n\leqslant7$. There exists a generic subset $\Phi\subset C^{\infty}(\p\Omega,\Sp^{n-1})$ with the following property. At each singularity $x_0$ of a stable stationary harmonic maps $u\in H^1(\Omega,\Sp^{n-1})$ with boundary $u|_{\p\Omega}\in\Phi$, the tangent map can only be 
        $$h_{x_0}(x)=\mathscr{O}_{x_0}\left(\frac{x}{\vert x\vert}\right),\text{ for some }\mathscr{O}_{x_0}\in O(n);$$
        \item Suppose $n=7$ and $k\geqslant7$. There exists a generic subset $\Phi\subset C^{\infty}(\p\Omega,\Sp^{k})$ such that every stable stationary harmonic map $u\in H^1(\Omega,\Sp^{k})$ with $u|_{\p\Omega}\in\Phi$ is smooth.
    \end{enumerate}
\end{thm}

The regularity results of the author \cite{li2026optimalregularitystableharmonic} building on the regularity theory of R. Schoen and K. Uhlenbeck \cite{SchoenUhlenbeckRegularity} place these cases in the first dimensions in which singular stable stationary maps can occur; see also \cite{LinWangSphere,SchoenUhlenbeck}. The theorem therefore strengthens the corresponding partial regularity results in two different ways: when domain dimension is greater than target dimension where topologically nontrivial boundaries are allowed, it classifies every persistent singularity; in the remaining case, it eliminates singularities generically in dimension seven.

\subsection{A minimum index principle for tangent links}
We now state the general result underlying Theorem \ref{thm: generic regularity for sphere valued maps}.

A \textbf{regular $0$-homogeneous map} is a map $h:\R^n\to M$ that is scaling invariant, namely $h(\lambda x)=h(x)$ for all $\lambda>0$, and smooth away from the origin. It is determined by its \textbf{link}
\[
 u=h|_{\p B_1},h(x)=u\left(\frac{x}{|x|}\right).
\]
$h$ is harmonic if and only if $u$ is. The \textbf{Morse index} of a harmonic link is the maximal dimension of a linear space on which the second variation of the energy is negative definite; see Section \ref{s: prelim} for the precise definition. For $n\geq4$, every nonconstant harmonic map $u\in C^{\infty}(\Sp^{n-1},M)$ has Morse index at least $n$ by Y. L. Xin \cite{Xin} and A. El Soufi \cite{ElSoufi}. The following theorem shows that generic boundary data force equality.

\begin{thm}\label{thm: generic regularity of stable harmonic maps}
    Let $M$ be a closed analytic manifold and $\Omega\subset\R^n$ be a bounded smooth domain, where $n\geqslant4$. Assume that $M$ admits no non-constant stable harmonic map from $\Sp^2$, and that, for $3\leqslant k\leqslant n-1$, there is no non-constant stable stationary harmonic regular 0-homogeneous map from $\R^k$ to $M$. 
    
    Then there exists a generic subset $\Phi$ of $C^{\infty}(\p\Omega,M)$ with the following property. At each singularity $x_0$ of a stable stationary harmonic map $u\in H^1(\Omega,M)$ with boundary $u|_{\p\Omega}\in\Phi$, the link of the tangent map $h_{x_0}$ has Morse index $n$.
\end{thm}

The two target hypotheses have distinct purposes. The absence of non-constant stable harmonic two-spheres gives strong $H^1_{\loc}$ compactness for stable stationary harmonic maps with uniform energy bounds; see F. Lin \cite{LinGradientEstimate} and D. Hsu \cite{Hsu}. Excluding lower-dimensional stable homogeneous maps then allows dimension reduction to restrict the analysis to strongly isolated singularities; see R. Schoen and K. Uhlenbeck \cite{SchoenUhlenbeckRegularity}. Analyticity enters the analysis of tangent maps and the local finiteness of the possible densities in compact families of links. 

Under the hypotheses of Theorem \ref{thm: generic regularity of stable harmonic maps}, if every nonconstant stable stationary regular 0-homogeneous map from \(\mathbb R^n\) to \(M\) has a link of Morse index at least \(n+1\), then every stable stationary harmonic map with generic boundary data is smooth. The sphere-target conclusions of Theorem \ref{thm: generic regularity for sphere valued maps} follow by combining this principle with the link-index results recalled in Proposition \ref{prop: index of sphere targets}. Besides applying to the sphere targets, Theorem \ref{thm: generic regularity of stable harmonic maps} is broad enough to study the generic behavior of stable stationary harmonic maps to other targets. For example, recently J. Krantz \cite{Krantz} studied the regularity of stable harmonic maps to homogeneous spaces. One could calculate the index of the link of tangent maps in that case, to get the corresponding generic regularity result.

In dimension three, the corresponding minimum index condition is stability of the harmonic two-sphere forming the link.

\begin{thm}\label{thm: generic regularity in dimension 3}
    Let $M$ be a closed analytic manifold and $\Omega\subset\R^3$ be a bounded smooth domain. There exists a generic subset $\Phi$ of $C^{\infty}(\p\Omega,M)$ with the following property. At each singularity $x_0$ of an energy minimizing map $u\in H^1(\Omega,M)$ with boundary $u|_{\p\Omega}=\varphi\in\Phi$, the link of the tangent map is stable.

    Moreover, assume that there is no non-constant stable harmonic map from $\Sp^2$ to $M$. Then, any stable stationary harmonic map $u\in H^1(\Omega,M)$ with boundary $\varphi\in\Phi$ is smooth.
\end{thm}

Theorem \ref{thm: generic regularity in dimension 3} identifies a geometric restriction on singularities occurring for generic boundary data. This should be distinguished from a purely topological obstruction to extending a boundary map. Namely, though the persistence of the singularity of the map $x\mapsto x/\vert x\vert$ can be seen from the topology of its boundary map, persistence itself comes from the geometry fact that the target manifold supports a stable harmonic 2-sphere. For the $\Sp^2$ target, it is possible to construct a degree zero boundary map all of whose sufficiently small perturbations admit only singular energy minimizing maps. We refer readers to R. Hardt and F. Lin \cite{HardtLinRemark} for the construction of such a boundary map.

There is also a version for stationary harmonic maps without a stability assumption. In this case, the target hypotheses exclude harmonic spheres rather than only the stable objects appearing above.

\begin{thm}\label{thm: generic regularity of stationary harmonic maps}
    Let $M$ be a closed analytic manifold and $\Omega\subset\R^n$ be a bounded smooth domain where $n\geqslant4$. Assume that $M$ does not admit any non-constant harmonic map from $\Sp^k$ for $2\leqslant k\leqslant n-2$. 
    
    Then there exists a generic subset $\Phi$ of $C^{\infty}(\p\Omega,M)$ with the following property. At each singularity $x_0$ of a stationary harmonic map $u\in H^1(\Omega,M)$ with boundary $u|_{\p\Omega}=\varphi\in\Phi$, the link of the tangent map $h_{x_0}$ has Morse index $n$.
\end{thm}

\begin{rmk}
    Although we only state our theorems for maps from smooth bounded Euclidean domains, the results listed above indeed hold for maps from smooth compact manifolds with non-empty boundary. This is because our analysis mainly depends on the blow up analysis of tangent maps, whose domains are always Euclidean ones. The curvature of the domain manifold also becomes negligible in the blow-up procedure. However, in order to make the idea behind the presentation of the paper more transparent, we focus on Euclidean domains.
\end{rmk}

\subsection{Application: rigidity at infinity}
As a byproduct of our generic regularity result, we also obtain a global rigidity theorem from the existence of a single radial blow-down. The conclusion requires neither a rate of convergence nor an a priori assumption on the number of singular points.
\begin{cor}\label{cor: rigidity at infinity}
    Let $3\leqslant n\leqslant7$ and $u\in H^1_{\loc}(\R^n,\Sp^{n-1})$ be an energy minimizing map. Assume that for some sequence $R_j\rightarrow\infty$, $u(R_j\cdot)$ converges in $H_{\loc}^1$ to the radial projection map $x\mapsto x/\vert x\vert$ as $j\rightarrow\infty$. 
    
    Then there exists $x_0\in\R^n$, such that $$u(x)=\frac{x-x_0}{\vert x-x_0\vert},\text{ for all }x\ne x_0.$$
\end{cor}

Thus, a subsequential asymptotic model determines the entire map up to translation. Analogous rigidity questions for minimal hypersurfaces asymptotic to quadratic cones have been studied by L. Simon and B. Solomon \cite{SimonSolomon}, where the hypersurfaces are shown to be a translate of the corresponding cone or a leaf of Hardt-Simon foliation; see also N. Edelen and L. Spolaor \cite{EdelenSpolaorQuadraticcone} and L. Mazet \cite{MazetSimonsCones}. Related results for other variational problems appear in \cite{MillotPisante,EngelsteinRestrepoZhao}. We conjecture that Corollary \ref{cor: rigidity at infinity} holds for every $n\geqslant3$ and energy minimizing assumption is not necessary. 

\subsection{Proof strategy}

The deformation argument follows the framework developed by Y. Li and Z. Wang \cite{LiWangGenericRegularity} and A. Carlotto, Y. Li and Z. Wang \cite{CLWNon-persistence}, with the following major novelties:
\begin{itemize}
    \item To study deformations of stationary harmonic maps, we introduce a notion of stationarity for Jacobi fields and show that this condition is satisfied by Jacobi fields arising as limits of suitably normalized differences of stationary harmonic maps. To the best of our knowledge, the compatibility of Jacobi fields with the stationarity condition has not been explicitly addressed in previous work on singular harmonic maps.

    \item We compute the Fredholm index of the Jacobi operator through a direct decomposition that separates compact contributions and index-zero operators from the components that determine the index. Our argument relies only on standard tools from elliptic PDE theory and functional analysis, without invoking the general index formula of R. Lockhart and R. McOwen \cite{LockhartMcOwen} used in \cite{CLWNon-persistence}. This approach may also be useful in other geometric problems.

    \item We develop a refined local analysis of stationary harmonic maps and their Jacobi fields near isolated singularities, providing tools for the study of singular harmonic maps and their deformations.
\end{itemize}

We study harmonic maps with \emph{strongly isolated singularities}: at each singular point, at least one tangent map is $0$-homogeneous and smooth away from the origin. We also require regularity in a neighborhood of the boundary and refer to maps satisfying these conditions as \textbf{HSI}s. When the target is analytic, the tangent map at each strongly isolated singularity is unique. Our local analysis compares an \textbf{HSI} with its tangent maps and controls the differences between nearby \textbf{HSI}s after aligning their singular points.

Jacobi fields describe the infinitesimal behavior of nearby harmonic maps. In the presence of singularities, this relationship requires suitable control of the growth of the fields near the singular set. Lemma \ref{lem: induced Jacobi fields for pairs} makes this relationship precise by showing how Jacobi fields arise from suitably normalized differences between nearby \textbf{HSI}s with corresponding singular sets.

We denote by $\widehat{\ker}_{\tau}L_u$ and $\widehat{\operatorname{coker}}_{\tau}L_u$, respectively, the kernel and cokernel of the Jacobi operator $L_u$ on the relevant weighted spaces. The parameter $\tau$ specifies a lower bound on the growth rate near each singular point of $u$. The hat indicates that the domain of $L_u$ is augmented by translation-type fields, which account for the motion of singular points and may arise in the normalized convergence of \textbf{HSI}s. The kernel describes infinitesimal deformations, whereas the cokernel records obstructions to solving the linearized equation. We use these obstructions to select boundary perturbations that prevent singularities from persisting. More precisely, Theorem \ref{thm: local sard smale theorem} shows that, if
\[
\dim \widehat{\operatorname{coker}}_{\tau}L_u
>
\dim \widehat{\ker}_{\tau}L_u,
\]
then the boundary data can be perturbed in a direction determined by a suitable cokernel element so as to reduce this dimensional excess.

To apply this perturbation argument, we compute the augmented Fredholm index
\[
\widehat{\ind}_{\tau}L_u
:=
\dim \widehat{\ker}_{\tau}L_u
-
\dim \widehat{\operatorname{coker}}_{\tau}L_u.
\]
Theorem \ref{thm: count of index} localizes this index at the singular points, expressing it as a sum of contributions determined by the effective Morse indices of the tangent maps $h_{x_0}$, $x_0\in\sing u$. The corresponding index computation for smooth harmonic maps was carried out by R.~Hardt and L.~Mou \cite{HardMou}.

The key observation is that the regular part of the domain contributes no index. Indeed, on a smooth subdomain whose closure avoids $\sing u$, the Jacobi operator is elliptic with smooth coefficients. Standard elliptic theory shows that its Dirichlet realization is a compact perturbation of an isomorphism and therefore has index zero. Guided by this observation, we decompose the domain and range of $L_u$ as
\[
\bigoplus_{i=1}^{3}\mathscr{D}_i
\qquad\text{and}\qquad
\bigoplus_{i=1}^{3}\mathscr{R}_i.
\]
The first two pairs of summands are supported away from the singular set, and the corresponding diagonal blocks have index zero. We also show that all off-diagonal blocks are compact. Consequently, the entire index is carried by the second diagonal block, from $\mathscr{D}_2$ to $\mathscr{R}_2$, which is localized to small balls centered at the singular points of $u$.

On each such ball, quantitative convergence to the unique tangent map $h_{x_0}$ allows us to compare $L_u$ with the model operator $L_{h_{x_0}}$. After identifying the corresponding spaces, their difference is compact, so the local index agrees with that of the model operator. The latter can be computed explicitly from the Jacobi fields of $h_{x_0}$, yielding the contribution determined by its effective Morse index.

Combining the boundary perturbation argument with this index formula, we show that, for generic boundary data $\p\Omega\to M$, every \textbf{HSI} has vanishing total effective Morse index of its tangent maps. By the definition of the effective Morse index, the link of every tangent map at a singular point must therefore have Morse index $n$ when $n\geqslant4$, and have Morse index 0 when $n=3$.
\subsection{History and related results}
The study of regularity theory of harmonic maps begins with the seminal work of J. Eells and J. H. Sampson \cite{EellsSampson}, who prove that the homotopy classes between compact manifolds $M$ and manifold with nonpositive sectional curvature $N$ can be represented by smooth harmonic maps using harmonic map flow. S. Hildebrandt, H. Kaul and K. Widman \cite{HildebrandtKaulWidman} further show that weak harmonic maps are all smooth when the target manifold is negatively curved. This is a byproduct of the smoothness of harmonic maps whose images lie in a small geodesic ball. On the other hand, the work of C. B. Morrey \cite{Morrey}, R. Schoen \cite{SchoenAnalyticAspects} and F. H\'elein \cite{Helein} shows that weak harmonic maps from 2-dimensional domains are always smooth.

R. Schoen and K. Uhlenbeck \cite{SchoenUhlenbeckRegularity} initiate the study of partial regularity theory for harmonic maps from higher dimensional domains to general targets. They establish the monotonicity formula, $\epsilon$-regularity theorem as well as the strong compactness of energy minimizing maps. Combining these with the dimension reduction principle of H. Federer \cite{Federer}, they prove that energy minimizing maps are smooth away from a set of codimension at least 3.

However, the regularity for stationary harmonic maps is considerably more delicate. There are two parts of the difficulty. The first one is to derive the $\epsilon$-regularity theorem, which is obtained by L. C. Evans \cite{EvansSphere}, F. Bethuel \cite{Bethuel} generalized by T. Rivi\`ere and M. Struwe \cite{RiviereStruwe}. A further difficulty is that convergence of stationary harmonic maps need not be strong in $H^1$, even subsequentially. The resulting loss of compactness gives rise to the defect measures and complicates the dimension reduction argument. Consequently, the optimal estimate is that the singular set of a general stationary harmonic map has zero codimension 2 Hausdorff measure; see \cite{JiangLiYu} for recent development. See J. Jost \cite{Jost}, T. H. Parker \cite{ParkerBubbleTree}, F. Lin \cite{LinGradientEstimate}, F. Lin and T. Rivi\`ere \cite{LinRiviere}, A. Naber and D. Valtorta \cite{naber2024energyidentitystationaryharmonic} and the author's work \cite{li2025energy} on fine characterization of the loss of energy in this convergence.

In some cases the strong convergence of stationary harmonic maps can be guaranteed. The strong compactness of all stationary harmonic maps is ensured by the no harmonic two spheres assumption of the target which is shown by F. Lin \cite{LinGradientEstimate}, while D. Hsu \cite{Hsu} observes that the weaker assumption that the target admits no stable harmonic two sphere is enough for the strong compactness of stable stationary harmonic maps; see also D. Hsu and J. Li \cite{HsuLiregularity} and M. Karpukhin and D. Stern \cite{MikhailStern}. In either case, the singular sets of corresponding objects have codimension at least 3. The bound can be improved when we know the target has no higher dimensional corresponding harmonic tangent maps. For improvement of dimension bounds of singular sets of stable stationary harmonic maps, see for \cite{SchoenUhlenbeck,LinWangSphere,li2026optimalregularitystableharmonic} for improvement for sphere targets. 

Finally, our work fits into the broader program of generic regularity in geometric variational problems, which has seen substantial progress in recent years. A foundational contribution is the Hardt--Simon foliation \cite{HardtSimon}, which provides smoothings of regular area-minimizing hypercones by families of minimal hypersurfaces. For regular mean-convex cones that are not area-minimizing on the relevant side, Q. Ding \cite{DingExpander} shows that the corresponding self-expanding solutions to mean curvature flow are smooth. Z. Wang \cite{WangSmoothing} subsequently extends these smoothing constructions to cones that need not be regular, treating both the one-sided minimizing case and the non-minimizing, viscosity mean-convex case. A complementary result is B. White's bumpy metrics theorem \cite{WhiteBumpyMetric}, which establishes the generic nondegeneracy of closed smooth minimal submanifolds.

Beyond the conical setting, N. Smale \cite{SmaleGenericHomological} establishes generic regularity for homologically area-minimizing hypersurfaces in ambient dimension eight. This result is extended to dimension nine and ten by O. Chodosh, C. Mantoulidis, and F. Schulze \cite{ChodoshMantoulidisSchulze}. O. Chodosh, Y. Liokumovich, and L. Spolaor \cite{ChodoshLiokumovichSpolaor} develop a perturbation theory for one-parameter min-max hypersurfaces, obtaining generic smoothness in dimension eight under a positive Ricci curvature assumption. Y. Li and Z. Wang \cite{LiWangGenericRegularity} subsequently prove that, for generic metrics on closed eight-dimensional manifolds, every embedded locally stable minimal hypersurface is smooth and nondegenerate.

For the Plateau problem, O. Chodosh, C. Mantoulidis, and F. Schulze \cite{ChodoshMantoulidisSchulze} establish generic regularity in ambient dimensions nine and ten, and their joint work with Z. Wang \cite{chodosh2025generic} extends this result to ambient dimension eleven. In higher dimensions, sharper Hausdorff-dimension bounds for the singular sets of generic area-minimizing hypersurfaces are obtained in \cite{ChodoshMantoulidisSchulzeImprovedBound,chodosh2025generic}, while the author \cite{LiMinkowskiBound} establishes Minkowski-dimension bounds. Related generic regularity results for other variational problems can be found in \cite{FigalliRosotonSerra,XavierTorres,fernandez2023generic,fernandez2025linearized}.

\subsection{Organization of the paper}
In Section \ref{s: prelim} we introduce the necessary definitions and conventions of the paper. We also establish preliminary criteria regarding stationary harmonic maps and stationary Jacobi fields. In Section \ref{s: fredholm index} we calculate the index of Jacobi operator on weighted Sobolev spaces. We also construct a finite-dimensional subspace of boundary sections, of dimension equal to that of the cokernel, such that no nonzero element is the boundary trace of an admissible stationary Jacobi field. In Section \ref{s: meagerness} we establish the perturbation results. Finally, in Section \ref{s: proof of main results} we summarize how the perturbation results in Section \ref{s: meagerness} imply our main results.

Appendix \ref{a: asym rate} studies the control of asymptotic rates of Jacobi fields while Appendix \ref{a: quantitaive estimates} studies the quantitative uniqueness of tangent maps and quantitative estimates of differences of harmonic maps when they have the same density at a singular point.

\subsection*{Acknowledgment}
The author is very grateful to Prof. Xin Zhou and Prof. Daniel Stern for their continuous support and valuable guidance. The author also wants to thank Zhihan Wang for suggesting and interest in this problem, numerous fruitful discussions and pointing out the possibility to derive Corollary \ref{cor: rigidity at infinity}. The author is partly supported by NSF grant DMS-2404992 and Simons Dissertation Fellowship in Mathematics.

\subsection*{AI usage statement}
LLMs are only used for the reference search and to improve the exposition of the manuscript. 

\section{Preliminaries}\label{s: prelim}
In this section we recall several basic facts about harmonic maps and the analysis of their singular sets. We refer readers to F. Lin and C. Wang \cite{LinWangbook} for a comprehensive understanding of the subject, to L. C. Evans \cite{EvansPDE} and Q. Han and F. Lin \cite{LinHanBook} for the standard elliptic PDE estimates and to T. B\"uhler and D. A. Salamon \cite{TheoSalamon} for the basics of functional analysis. 

Given a map $u\in W^{1,2}(\Omega,M)$, recall that the energy of $u$ is defined as
$$E(u)=\frac{1}{2}\int_{\Omega}\vert\nabla u\vert^2\dif x.$$
Here and in the sequel, we denote by $\nabla$ the standard Euclidean connection. We say $u$ is a \textbf{stationary harmonic map}, if $u$ is a critical point of the energy functional. When $u$ is smooth, $u$ satisfies the following equation
$$\Delta u+A_u(\nabla u,\nabla u)=0,$$
where $A_u$ is the second fundamental form of $M$ at $u$. When $u$ is merely in $H^1(\Omega,M)$, it is stationary harmonic if and only if it weakly solves the equation above and satisfies the following stationarity identity
$$\int_{\Omega}\vert\nabla u\vert^2\operatorname{div}\xi\dif x=2\sum_{\alpha,\beta=1}^n\int_{\Omega}\langle\partial_{\alpha}u,\partial_{\beta}u\rangle\partial_{\alpha}\xi^{\beta}\dif x,\text{ for each }\xi\in C_c^{\infty}(\Omega,\R^n).$$

\subsection{Strong convergence assumption}
We make the following assumption for the convenience of the paper.
\begin{equation}
    \begin{aligned}
    &\text{Within the class of harmonic maps under consideration, every sequence with}\\
    &\text{uniformly bounded energy has a subsequence converging strongly in }H^1_{\loc}
    \end{aligned}
    \tag{S}\label{eq: strong convergence}
\end{equation}
As a result, given the compactness result of R. Schoen and K. Uhlenbeck and blow up analysis of F. Lin \cite{LinGradientEstimate} and D. Hsu \cite{Hsu}, we assume that, when we are studying
\begin{itemize}
    \item energy minimizing harmonic maps, there is no extra assumption on $M$;
    \item stable stationary harmonic maps, $M$ does not admit non-constant stable harmonic 2-sphere;
    \item stationary harmonic maps, $M$ does not admit non-constant harmonic 2-sphere.
\end{itemize}
Note that in the second and third settings, we automatically have that smooth maps are dense in $H^1(\Omega,M)$. This is because the existence result of J. Sacks and K. Uhlenbeck \cite{SacksUhlenbeck} implies that $\pi_2(M)=0$ under such assumptions, then the approximation result of F. Bethuel \cite{BethuelApproximation} applies.

\subsection{Singular sets of stationary harmonic maps}
Let us recall the definition of regular and singular sets.
\begin{defn}
    Let $u\in H^1(\Omega,M)$. We say $x_0\in\Omega$ is a regular point of $u$, if there exists an $r>0$ such that $u$ is smooth on $B_r(x_0)$. The \textbf{regular set} of $u$, denoted by $\reg u$, is the set of regular points of $u$. We set its complement $\sing u=\Omega\setminus \reg u$ to be the \textbf{singular set} of $u$.
\end{defn}
Under the assumption that all the weak $H^1$ convergence in the chosen class is always strong, by the regularity theory of R. Schoen and K. Uhlenbeck \cite{SchoenUhlenbeckRegularity}, we know that the singularity behavior of harmonic map is modeled on tangent maps.
\begin{defn}
    Let $u:\Omega\rightarrow M$ be a stationary harmonic map.
    \begin{enumerate}
        \item We say $h:\R^n\rightarrow M$ is a \textbf{0-homogeneous map}, if $h(\lambda x)=h(x)$ for each $\lambda>0$. We say $h$ is \textbf{regular} 0-homogeneous, if $h$ is also smooth away from $0$; In either case, we call $u=h|_{\p B_1}$ the \textbf{link} of $h$.
        \item Let $x_0\in\sing u$. We say $h:\R^n\rightarrow M$ is a tangent map of $u$ at $x_0$, if for some $r_j\rightarrow0$, $u(r_j\cdot+x_0)\rightarrow h$ strongly in $H^1_{\loc}(\R^n)$ as $j\rightarrow\infty$. Tangent maps are automatically 0-homogeneous by monotonicity formula.
        \item We say $x_0\in\sing u$ is \textbf{strongly isolated}, if some tangent map of $u$ at $x_0$ is regular 0-homogeneous. If $\sing u$ is empty or consists only of strongly isolated singular points and $u$ is smooth near $\p\Omega$, we call $u$ an \textbf{HSI} (\textbf{H}armonic map with only \textbf{S}trongly \textbf{I}solated singular points).
    \end{enumerate}
\end{defn}
Note that by a celebrated result of L. Simon \cite{SimonAsymptotics}, the tangent map of a stationary harmonic map $u$ at a strongly isolated singular point $p$ is unique, which allows us to denote this map by $h_p$. Also, \cite{SchoenUhlenbeckRegularity} shows strongly isolated singular points are isolated.

The following quantity is useful in the study of singular sets.
\begin{defn}
    The regularity scale $r_u:\Omega\rightarrow[0,1]$ of an \textbf{HSI} is defined as the largest number $0\leqslant r\leqslant1$ such that 
    $$r\vert\nabla u\vert+r^2\vert\nabla^2u\vert\leqslant1\text{ on }B_r(x)\cap\Omega.$$
    And set $r_u(x)=0$ for $x\in\sing u$.
\end{defn}

\subsection{Sobolev spaces with controlled growth rate near singular points}

When considering the variational theory of the energy functional at a map $u$, the variation vectors to be used are those tangent to $u$, i.e. lie in the tangent space of $M$ at $u$, as specified below. 

Given a map $X\in L^2_{\text{loc}}(\reg u,\R^N)$, we say $X$ is an $L^2_{\text{loc}}$ section of $u^*(TM)$, provided for a.e. $x\in\reg u$, $X(x)$ belongs to $T_{u(x)}M$, and we denote $X\in L^2_{\text{loc}}(\reg u,T_uM)$. Similarly, for $k\in\N$, we define the spaces $C^{k,\alpha}_c(\reg u,T_uM)$,  $C^{k,\alpha}_{\text{loc}}(\reg u,T_uM)$, $H^k_{\text{loc}}(\reg u,T_uM)$ and  $H^k(\Omega,T_uM)$ etc. Moreover, the following weighted Sobolev spaces will be frequently used in this paper.

Assume the singular set of an \textbf{HSI} $u$ is $\sing u=\lbrace x_i\rbrace_{i=1}^Q$. Let $r_0=r_0(u)>0$ be a number less than $\dist(\sing u,\p\Omega)/10$ and $\min_{i\ne i'}\dist (x_i,x_{i'})/10$. Given an index $\tau\in\R$, we denote by $\rho^{\tau}_u$ a positive continuous function on $\reg u$ which agrees with $\vert\cdot-x_i\vert^{\tau}$ on $B_{r_0}(x_i)$ for $1\leqslant i\leqslant Q$.
\begin{defn}
    Let $k\in\N$ and $\tau\in\R$. The weighted Sobolev space $H^{k;\tau}(\reg u,T_uM)$ consists of sections $X\in H^k_{\text{loc}}(\reg u,T_uM)$ satisfying
    $$\Vert X\Vert_{H^{k;\tau}(\reg u)}=\left(\sum_{j=0}^k\int_{\Omega}\vert\rho^{-\tau+j}_u\nabla^jX\vert^2\rho_u^{-n}\dif x\right)^{\frac{1}{2}}<\infty.$$
    Moreover, we set $H_0^{k;\tau}(\reg u,T_uM)$ to be the sections in $H^{k;\tau}(\reg u,T_uM)$ that vanish on $\p\Omega$ to order $k-1$. More specifically, given any cut-off function $\zeta$ with $\zeta=0$ near each $x_i$ and $\zeta=1$ near $\partial\Omega$, we require that $\zeta X\in H^k_0(\Omega,\R^N)$.

    Finally, we define $H^{-k;\tau}(\reg u,T_uM)$ to be the dual space of $H^{1;-\tau-n}_0\cap H^{k;-\tau-n}(\reg u,T_uM)$. 
\end{defn}
We need to augment these weighted Sobolev spaces by the effect of translations.
\begin{defn}
    Let $u:\Omega\rightarrow M$ be an \textbf{HSI}. We say that a section $\phi\in C^{\infty}_{\loc}(\bar\Omega\setminus\sing u,T_uM)$ is \textbf{translation-like}, if for every $x_0\in\sing u$, there exists a neighborhood $U$ of $x_0$, and a vector $v\in\R^n$ such that $$\phi|_U=\p_vu.$$
\end{defn}
We then set the space
$$\hat{H}^{k;\tau}(\reg u, T_uM)=\lbrace X\in H^k_{\text{loc}}(\reg u,T_uM):\Vert X-\phi\Vert_{H^{k;\tau}(\reg u)}<\infty\text{ for a translation-like }\phi\rbrace.$$
As in the previous paragraph, we set $\hat{H}^{k;\tau}_0(\reg u, T_uM)$ to be the sections in $\hat{H}^{k;\tau}(\reg u, T_uM)$ that vanish on $\p\Omega$ to order $k-1$.

By definition, when $\tau>-1$ we can write 
$$\hat{H}^{k;\tau}(\reg u, T_uM)=H^{k;\tau}(\reg u, T_uM)\oplus H_{\text{Tr}},$$ where $H_{\text{Tr}}$ is a finite dimensional linear space with basis
$$\lbrace\zeta_j\p_{\alpha}u:1\leqslant\alpha\leqslant n,1\leqslant j\leqslant Q\rbrace,$$
where $\zeta_j=1$ in $B_{r_0}(x_j)$ and $\zeta_j=0$ outside $B_{2r_0}(x_j)$. Hence $H_{\text{Tr}}$ has dimension $nQ$.

\subsection{Second variation of energy and the index of harmonic spheres}

The second variation of a stationary harmonic map $u$ is given by (see e.g. \cite[P. 37]{LinWangbook})
$$\delta^2E(u)(X)=-\int_{\Omega}\langle X,L_uX\rangle\dif x,$$
where $X\in C_c^{\infty}(\Omega,T_uM)$ and the Jacobi operator associated to $u$ is given by
$$L_u=\Delta+DA_u(\cdot,\nabla u,\nabla u)+2A_u(\nabla \cdot,\nabla u).$$
If $\delta^2E(u)(X)\geqslant0$ for all $X\in C_c^{\infty}(\Omega,T_uM)$, we say $u$ is a \textbf{stable} (stationary) harmonic map.

In the following we discuss a special case, where $h:\R^n\rightarrow M$ is a regular 0-homogeneous map. In this case, we denote the link of $h$ by $u=h|_{\p B_1}$, the Jacobi operator of $u$ takes the following form
$$L_u^{\Sp^{n-1}}=\Delta_{\Sp^{n-1}}+DA_u(\cdot,\nabla^{\Sp^{n-1}}u,\nabla^{\Sp^{n-1}}u)+2A_u(\nabla^{\Sp^{n-1}}\cdot,\nabla^{\Sp^{n-1}}u).$$
The spectrum of $L^{\Sp^{n-1}}_u$ is a discrete sequence $$\lambda_1\leqslant\lambda_2\leqslant\cdots\leqslant\lambda_k\leqslant\cdots\rightarrow\infty.$$
\begin{defn}\label{defn: effective Morse index}
    Let $h:\R^n\rightarrow M$ be a regular 0-homogeneous map and $u: \Sp^{n-1}\rightarrow M$ its link. We define the \textbf{effective Morse index} of $h$ as
    \begin{align*}
        I(h)=\begin{cases}
            (\text{Morse index of } L^{\Sp^{n-1}}_u)-n,&\text{ if }n-1\geqslant3;\\
            \text{Morse index of } L^{\Sp^{n-1}}_u,&\text{ if }n-1=2.
        \end{cases}
    \end{align*}
    Here the Morse index of $L^{\Sp^{n-1}}_u$ refers to the number of its negative eigenvalues, counted with multiplicity.
\end{defn}
By the calculation of Y. L. Xin \cite{Xin} and A. El Soufi \cite{ElSoufi}, we always have that $I(h)\geqslant0$ for non-constant $h$. Moreover, when the target is a sphere, we have the following improvement on this index estimate.

\begin{prop}[T. Nakajima\cite{NakajimaIndex} for $n=k$ and X. Li \cite{li2026optimalregularitystableharmonic} for $k>n$]\label{prop: index of sphere targets}
    Let $k\geqslant n\geqslant3$ and $u\in C^{\infty}(\Sp^n,\Sp^k)$ be a non-trivial harmonic map.
    \begin{enumerate}
        \item If $k=n$, and if $\ind u=n+1$, then $u(x)=\mathscr{O}(x/\vert x\vert)$ for some orthogonal matrix $\mathscr{O}\in O(k+1)$;
        \item If $k>n$, then $\ind u\geqslant n+2$.
    \end{enumerate}
\end{prop}

\subsection{Stationarity condition for Jacobi fields}
Given a stationary harmonic map $u$, we want to describe the neighborhood of $u$ in the space of stationary harmonic maps. The major tool used for this description is the \textbf{Jacobi field} of $u$, which models the normalized convergence of a sequence of harmonic maps to $u$; see Lemma \ref{lem: induced Jacobi fields for pairs} for more details. 

In order to depict the normalized convergence of \emph{stationary} harmonic maps, we introduce the following condition.
\begin{defn}\label{defn: stationary jaocbi fields}
    Let $u:\Omega\rightarrow M$ be an \textbf{HSI} and $X\in H^1_{\loc}(\reg u,T_uM)$.
    \begin{enumerate}
        \item We say $X$ is a \textbf{Jacobi field} of $u$, if $X$ solves $L_uX=0$ in the distributional sense on $\reg u$;
        \item We say $X$ is a \textbf{stationary Jacobi field} of $u$ if for all $\xi\in C_c^{\infty}(\Omega,\R^n)$ that is a constant near each singularity of $u$,
        $$\int_{\Omega}\langle\nabla u,\nabla X\rangle\operatorname{div}\xi\dif x=\sum_{\alpha,\beta=1}^n\int_{\Omega}(\langle\p_{\alpha}u,\p_{\beta}X\rangle+\langle\p_{\beta} u,\p_{\alpha}X\rangle)\p_{\alpha}\xi^{\beta}\dif x.$$
    \end{enumerate}
\end{defn}

We note that if the singular set of a harmonic map is small in the sense of dimension, then the stationary condition is automatic. For simplicity we only state the following property for harmonic maps with isolated singularities, but they can be easily generalized to harmonic maps with $(n-4)$-rectifiable singular sets.

\begin{prop}\label{prop: stationarity for n greater than 3}
    Let $u:\Omega\rightarrow M$ be a weak harmonic map with only isolated singularities and $X\in H^{1;\tau}(\reg u,T_uM)$ be a Jacobi field of $u$, where $\Omega$ is a smooth bounded domain in $\R^n$ and $\tau\geqslant-1$.
    
    Assume $n\geqslant4$, then both $u$ and $X$ are stationary.
\end{prop}
\begin{proof}
    Stationarity of $u$ is well-known, we recall the proof here for completeness. Set the stress energy tensor $S_{\alpha\beta}=\vert\nabla u\vert^2\delta_{\alpha\beta}-2\langle\p_{\alpha} u,\p_{\beta}u\rangle$, we calculate on $\reg u$,
    \begin{equation}
        \sum_{\alpha=1}^n\p_{\alpha}S_{\alpha\beta}=\p_{\beta}\vert\nabla u\vert^2-2\langle\Delta u,\p_{\beta}u\rangle-2\sum_{\alpha=1}^n\langle\p_{\alpha}u,\p^2_{\alpha\beta}u\rangle=-2\langle\Delta u,\p_{\beta}u\rangle\label{eq: divergence of stress energy tensor}.
    \end{equation}
    Note that this calculation is valid for arbitrary smooth maps, and in particular when $u$ is harmonic, we have that $\sum_{\alpha}\p_{\alpha}S_{\alpha\beta}=0$. Take $\xi\in C_c^{\infty}(\reg u,\R^n)$. By multiplying $\xi^{\beta}$, summing for $\beta\in\lbrace1,\dots,n\rbrace$ and integrating by parts, we deduce the stationary equation for such $\xi$ supported in $\reg u$.
    For a general $\xi\in C_{c}^{\infty}(\Omega,\R^n)$, for small $r$ take a cutoff function $\zeta_r\in C_c^{\infty}(B_r(\sing u))$ such that $\zeta_r\in[0,1],\zeta_r=1$ on $B_{r/2}(\sing u)$ and $\vert\nabla\zeta_r\vert\leqslant C/r$ on $B_r(\sing u)\setminus B_{r/2}(\sing u)$. Since $\xi(1-\zeta_r)$ is supported in $\reg u$, we have
    \begin{equation}
        \begin{aligned}
            0=&\int_{\Omega}\left(\vert\nabla u\vert^2\operatorname{div}(\xi(1-\zeta_r))-2\sum_{\alpha,\beta=1}^n\langle\p_{\alpha}u,\p_{\beta}u\rangle\p_{\alpha}(\xi^{\beta}(1-\zeta_r))\right)\dif x\\=&\int_{\Omega}\left(\vert\nabla u\vert^2\operatorname{div}\xi-2\sum_{\alpha,\beta=1}^n\langle\p_{\alpha}u,\p_{\beta}u\rangle\p_{\alpha}\xi^{\beta}\right)(1-\zeta_r)\dif x+R(r).
        \end{aligned}\label{eq: stationary equation of harmonic on regular part}
    \end{equation}
    \begin{claim}
        For each $x_0\in\sing u$ and $r$ sufficiently small, $r^{2-n}\int_{B_r(x_0)}\vert\nabla u\vert^2\dif x$ is monotone non-decreasing.
    \end{claim}
    \begin{proof}
    In order to prove the claim, choose $\xi(x)=(x-x_0)\eta(\vert x-x_0\vert)$ in \eqref{eq: stationary equation of harmonic on regular part}, where $\eta$ is a cutoff function such that $\eta=1$ on $B_{r_0}(x_0)$ and $\eta=0$ on $B_{2r_0}(x_0)$ and $r_0$ is a small number such that $\sing u\cap B_{r_0}(x_0)=\lbrace x_0\rbrace$. In this case, since $\vert\xi\vert\leqslant r$ on $\operatorname{spt}\nabla\zeta\cap B_{r_0}(x_0)= A_{r/2,r}(x_0)$, the remaining term in \eqref{eq: stationary equation of harmonic on regular part} is controlled by
    $$\vert R(r)\vert\leqslant C\int_{\Omega}\vert\nabla u\vert^2\vert\nabla\zeta_r\vert\vert\xi\vert\dif x\leqslant C\int_{A_{r/2,r}(x_0)}\vert\nabla u\vert^2\dif x\rightarrow0\text{ as }r\rightarrow0.$$
    We see there holds
    $$\int_{\Omega}((n-2)\eta(\vert x-x_0\vert)\vert\nabla u\vert^2+\vert x-x_0\vert\eta'(\vert x-x_0\vert)(\vert\nabla u\vert^2-2\vert\nabla_{(x-x_0)/\vert x-x_0\vert}u\vert^2))\dif x=0.$$
    We can then derive the monotonicity formula of $u$ on $B_{r_0}(x_0)$ as usual, see for example \cite{LinWangbook}, to get the claim.   
    \end{proof}
    In particular, $r^{2-n}\int_{B_r(x_0)}\vert\nabla u\vert^2\dif x$ is bounded by a constant that may depend on $u,\Omega$ and $n$. For a general $\xi\in C_c^{\infty}(\Omega,\R^n)$, the remaining term in \eqref{eq: stationary equation of harmonic on regular part} is controlled by
    $$\vert R(r)\vert\leqslant C\int_{\Omega}\vert\nabla u\vert^2\vert\nabla\zeta_r\vert\vert\xi\vert\dif x\leqslant C\sup\vert\xi\vert r^{-1}\int_{B_{r}(\sing u)}\vert\nabla u\vert^2\dif x\leqslant Cr^{n-3}\rightarrow0\text{ as }r\rightarrow0.$$
    Hence, taking $r\rightarrow0$ in \eqref{eq: stationary equation of harmonic on regular part}, we see that $u$ is stationary.

    The stationarity of $X$ follows a similar route. Standard elliptic regularity shows that $X$ is smooth on $\reg u$. Consider a family of maps $(u_t)|_{t\in(-\epsilon,\epsilon)}$ with $u_0=u$ and $\p_t u_t|_{t=0}=X$ on $\reg u$. Differentiating of \eqref{eq: divergence of stress energy tensor} for $u_t$ at $t=0$, we see that the linearization of the stress energy tensor $$T_{\alpha\beta}=\langle\nabla u,\nabla X\rangle\delta_{\alpha,\beta}-\langle\p_{\alpha}u,\p_{\beta}X\rangle-\langle\p_{\beta}u,\p_{\alpha}X\rangle,$$
    satisfies
    \begin{equation}
        \sum_{\alpha=1}^n\p_{\alpha}T_{\alpha\beta}=-\frac{\textup{d}}{\textup{d}t}\langle\tau(u_t),\p_{\beta}u_t\rangle|_{t=0}=-\langle L_uX,\p_{\beta} u\rangle-\langle\tau(u),\p_{\beta}X\rangle,\label{eq: divergence of stress-energy tensor of section}
    \end{equation}
    where $\tau(u)=\Delta u+A_u(\nabla u,\nabla u)=\pi_{T_uM}\Delta u$ is the tension field of $u$, whose linearization is $L_u$. In particular, since $u$ is harmonic and $X$ is a Jacobi field of $u$, we have that $\sum_{\alpha=1}^n\p_{\alpha}T_{\alpha\beta}=0$.
    Hence, for arbitrary $\xi\in C_c^{\infty}(\Omega,\R^n)$ that is constant near each singularity of $u$, take $\zeta_r$ as before. Since $\xi^{\beta}(1-\zeta_r)$ is constant near each singularity of $u$, we can multiply \eqref{eq: divergence of stress-energy tensor of section} by it, sum for $\beta$ and integrate by parts
    \begin{equation}
        \int_{\Omega}\left(\langle\nabla u,\nabla X\rangle\operatorname{div}\xi-\sum_{\alpha,\beta=1}^n(\langle\p_{\alpha} u,\p_{\beta}X\rangle+\langle\p_{\beta}u,\p_{\alpha}X\rangle)\p_{\alpha}\xi^{\beta}\right)(1-\zeta_r)\dif x=R(r),\label{eq: stationary equation of field on regular part}
    \end{equation}
    where the remaining term is bounded by
    \begin{align*}
        \vert R(r)\vert^2\leqslant &C(\xi)\int_{B_r(\sing u)}\vert\nabla u\vert^2\dif x\int_{B_r(\sing u)\setminus B_{r/2}(\sing u)}\vert\nabla X\vert^2\vert\nabla\zeta_r\vert^2\dif x\\\leqslant& C r^{n-4}\int_{B_r(\sing u)\setminus B_{r/2}(\sing u)}\vert\nabla X\vert^2\dif x\\\leqslant& Cr^{2n-6+2\tau}\int_{B_r(\sing u)\setminus B_{r/2}(\sing u)}\vert\nabla X\vert^2\rho_u^{-n+2-2\tau}\dif x\rightarrow0\text{ as }r\rightarrow0,
    \end{align*}
    where to get the limit we used the integrability of $\int_{\Omega}\vert\nabla X\vert^2\rho_u^{-n+2-2\tau}\dif x<\infty$ and the fact that $2n-6+2\tau\geqslant0$. Hence, taking $r\rightarrow0$ in \eqref{eq: stationary equation of field on regular part}, we see that $X$ is stationary. 
\end{proof}

When $n=3$, we have the following characterization of the stationarity of homogeneous harmonic maps and their Jacobi fields.

\begin{prop}\label{prop: stationarity when n=3}
    Let $h:\R^3\rightarrow M$ be a regular 0-homogeneous map that is weakly harmonic, $u=h|_{\p B_1}$ and  $X(r\omega)=r^{\gamma}Y(\omega)$ be a Jacobi field of $h$. Here $r=\vert x\vert,\omega=x/r$ are the polar coordinates on $\R^3$. Then,
    \begin{enumerate}
        \item $h$ is stationary if and only if
        $$\int_{\Sp^2}\vert\nabla u\vert^2\omega\dif\omega=0;$$
        \item in the case where $h$ is stationary, if $\gamma\ne0$, $X$ is stationary. If $\gamma=0$, $X$ is stationary if and only if $$\int_{\Sp^2}\langle \nabla Y,\nabla u\rangle\omega\dif\omega=0.$$ 
    \end{enumerate}
\end{prop}
\begin{proof}
    (1) is well known; see for example W. Ding, J. Li and W. Li\cite{DingLiLi}. For completeness, we include the proof of it here. For any $\xi\in C_c^{\infty}(\R^3,\R^3)$, multiply \eqref{eq: divergence of stress energy tensor} by $\xi^{\beta}$, sum for $\beta$ and integrate on $\R^3\setminus B_{\epsilon}$ gives
    $$\int_{\R^3\setminus B_{\epsilon}}\left(\vert\nabla h\vert^2\operatorname{div}\xi-2\sum_{\alpha,\beta=1}^3\langle\p_{\alpha}h,\p_{\beta}h\rangle\p_{\alpha}\xi^{\beta}\right)\dif x=-\frac{1}{\epsilon}\int_{\p B_{\epsilon}}\vert\nabla h\vert^2\langle\xi,x\rangle\dif\sigma_{\p B_{\epsilon}},$$
    where $x/r$ is the outer normal at $x\in\p B_r$ and we use the fact that $\sum_{\alpha=1}^3x_{\alpha}\p_{\alpha}h=r\p_rh=0$. Let us change the variable $x=\epsilon\omega,\omega\in\Sp^2$. Since $\nabla h(r\omega)=\nabla u(\omega)/r$ and $\dif\sigma_{\p B_{\epsilon}}=\epsilon^2\dif\omega$, we have that 
    $$\frac{1}{\epsilon}\int_{\p B_{\epsilon}}\vert\nabla h\vert^2\langle\xi,x\rangle\dif\sigma_{\p B_{\epsilon}}=\int_{\Sp^2}\vert\nabla^{\Sp^2} u\vert^2\langle\xi(\epsilon\omega),\omega\rangle\dif\omega\rightarrow\int_{\Sp^2}\vert\nabla^{\Sp^2} u\vert^2\langle\xi(0),\omega\rangle\dif\omega\text{ as }\epsilon\rightarrow0.$$
    $h$ being stationary is equivalent to above equation vanishing for any $\xi(0)$, which is equivalent to the conclusion of (1).

    The proof of (2) is similar. Assume $\xi$ is constant near 0. We multiply \eqref{eq: divergence of stress-energy tensor of section} by $\xi^{\beta}$, sum for $\beta$ and integrate on $\R^3\setminus B_{\epsilon}$. We deduce
    \begin{align*}
        &\int_{\R^3\setminus B_{\epsilon}}\left(\langle\nabla h,\nabla  X\rangle\operatorname{div}\xi-\sum_{\alpha,\beta=1}^3(\langle\p_{\alpha} h,\p_{\beta}X\rangle+\langle\p_{\beta}h,\p_{\alpha}X\rangle)\p_{\alpha}\xi^{\beta}\right)\dif x\\=&-\frac{1}{\epsilon}\int_{\p B_{\epsilon}}\left(\langle\nabla h,\nabla X\rangle\langle \xi,x\rangle-\sum_{\alpha,\beta=1}^3\langle\p_{\beta}h,\p_{\alpha}X\rangle\xi^{\beta}x_{\alpha}\right)\dif\sigma_{\p B_{\epsilon}}\\=&-\epsilon^{\gamma}\int_{\Sp^2}\left(\langle\nabla^{\Sp^2}u,\nabla^{\Sp^2}Y\rangle\langle\xi(0),\omega\rangle-\gamma\langle\nabla^{\Sp^2}_{\pi_{T_{\omega}\Sp^2}\xi(0)}u,Y(\omega)\rangle\right)\dif\omega.
    \end{align*}
    where we used the fact that on $\p B_{\epsilon}$, $\sum_{\alpha=1}^3x_{\alpha}\p_{\alpha}X=\epsilon\p_rX=\gamma\epsilon^{\gamma}Y$ and $\xi(\epsilon\cdot)=\xi(0)$ for small $\epsilon$.
    Note that if we integrate by parts on $\Sp^2$
    \begin{align*}
        \int_{\Sp^2}\langle\nabla^{\Sp^2}u,\nabla^{\Sp^2}Y\rangle\langle\xi(0),\omega\rangle\dif\omega=-\int_{\Sp^2}\langle\nabla^{\Sp^2}_{\pi_{T_{\omega}\Sp^2}\xi(0)}u,Y\rangle\dif\omega.
    \end{align*}
    Hence 
    \begin{align*}
        &\int_{\R^3\setminus B_{\epsilon}}\left(\langle\nabla h,\nabla  X\rangle\operatorname{div}\xi-\sum_{\alpha,\beta=1}^3(\langle\p_{\alpha} h,\p_{\beta}X\rangle+\langle\p_{\beta}h,\p_{\alpha}X\rangle)\p_{\alpha}\xi^{\beta}\right)\dif x\\&=-\epsilon^{\gamma}(1+\gamma)\int_{\Sp^2}\langle\nabla^{\Sp^2}u,\nabla^{\Sp^2}Y\rangle\langle\xi(0),\omega\rangle\dif\omega.
    \end{align*}
    Let $\epsilon\rightarrow0$. If $\gamma>0$, RHS converges to 0 which means $X$ is stationary. If $\gamma<0$, since LHS always converges to a finite value for any choice of $\xi(0)$, we know that RHS is forced to be 0 identically, otherwise it blows up when $\epsilon\rightarrow0$. In this case $X$ must also be stationary. Finally, if $\gamma=0$, the limit of RHS is $-\int_{\Sp^2}\langle\nabla^{\Sp^2}u,\nabla^{\Sp^2}Y\rangle\langle\xi(0),\omega\rangle\dif\omega$, and $X$ being stationary is equivalent to this integral vanishing for any choice of $\xi(0)$.
\end{proof}

\subsection{Conventions}
Here we list some conventions that are frequently used in this paper. We denote balls by $B_r(p)$, and when the center $p$ is the origin 0, we simply denote by $B_r$. For two positive numbers $s<r$, we denote by $A_{s,r}(p)=B_r(p)\setminus\overline{B_s(p)}$ the open annulus centered at $p$.

Given an \textbf{HSI} u with singular set $\lbrace x_1,\cdots x_Q\rbrace$, recall that we take $\rho_u^{\tau}$ to be a smooth function on $\reg u$ which equals $\vert\cdot-x_i\vert^{\tau}$ on each $B_{r_0}(x_i)$ and equals 1 outside $\cup_i B_{r_0}(x_i)$, where $r_0$ is a small number to make these balls dijoint from each other and $\p\Omega$. The energy density of $u$ is denoted by $\Theta_u(x_0,r)=r^{2-n}\int_{B_r(x_0)}\vert\nabla u\vert^2\dif x$ and $\Theta_u(x_0)=\lim_{r\rightarrow0}\Theta_u(x_0,r)$, where the latter limit exists by monotonicity formula.

\section{The Fredholm index of the Jacobi operator}\label{s: fredholm index}
In this section, we study the Fredholm index of the Jacobi operator of an \textbf{HSI}.

First let us recall the analysis of Jacobi fields of a tangent map. Let $h:\R^n\rightarrow M$ be such a map and $u=h|_{\Sp^{n-1}}$. Let $x=r\omega$ be polar coordinates of $\R^n$. The Jacobi operator of $h$ can be written as
\begin{align*}
    L_h&=\Delta+DA_h(\cdot,\nabla h,\nabla h)+2A_h(\nabla \cdot,\nabla h)\\&=\p_r^2+\frac{n-1}{r}\p_r+\frac{1}{r^2}(\Delta_{\Sp^{n-1}}+DA_u(\cdot,\nabla^{\Sp^{n-1}}u,\nabla^{\Sp^{n-1}}u)+2A_u(\nabla^{\Sp^{n-1}}\cdot,\nabla^{\Sp^{n-1}}u))\\&=\p_r^2+\frac{n-1}{r}\p_r+\frac{1}{r^2}L^{\Sp^{n-1}}_u.
\end{align*}
The homogeneous Jacobi fields for $h$ can be found by separation of variables as done by R. Hardt and L. Mou \cite{HardMou}, which we recall here; see also L. Caffarelli, R. Hardt and L. Simon \cite{CaffarelliHardSimon}. Let again $\lambda_1\leqslant\lambda_2\leqslant\cdots\leqslant\lambda_j\rightarrow\infty$ be the eigenvalues of $L^{\Sp^{n-1}}_u$ and let $X_i$ be the corresponding orthonormal eigensections. The growth rate of homogeneous Jacobi fields for $h$ can be solved from the following equation
$$\gamma^2+(n-2)\gamma-\lambda_j=0.$$
which has two roots 
$$\gamma_j^{\pm}=-\frac{n-2}{2}\pm\sqrt{\left(\frac{n-2}{2}\right)^2+\lambda_j}\in\CC.$$
We collect the real parts of these roots $\Gamma(h)=\lbrace\text{Re}(\gamma_j^{\pm})\rbrace\subset\R$ and call them the \textbf{asymptotic spectrum of} $h$. Moreover, set $$\gamma_-(h)=\sup(\Gamma(h)\cap\R_{<0})<0.$$
Note that Re$(\gamma_j^+)<0$ if and only if $\lambda_j<0$. Hence $\gamma_-(h)$ records the largest negative indicial growth rate. Also, since each derivative $\p_{\alpha}h,1\leqslant\alpha\leqslant n,$ has growth rate $-1$, we see $-1\in\Gamma(h)$. As a result, $\gamma_-(h)\geqslant-1$.

Write the polar coordinates on $\R^n$ as $x=r\omega$. A general (real) Jacobi field is given by
\begin{align}\label{eq: expansion of Jacobi field}
    X(r\omega)=\sum_{j=1}^{\infty}(\eta_j^+(r)+\eta_j^-(r))X_j(\omega),
\end{align}
where the coefficients are
\begin{equation}\label{eq: coefficients in the expansion of Jacobi fields}
    \begin{gathered}
        \eta_j^+(r)=c_j^+r^{\gamma_j^+}\\
            \eta_j^-=\begin{cases}
        c_j^-r^{\gamma_j^-}, &\text{ if }\lambda_j\ne-\frac{(n-2)^2}{4};\\
        c_j^-r^{\gamma_j^-}\log r,&\text{ if }\lambda_j=-\frac{(n-2)^2}{4}.
    \end{cases}
    \end{gathered}
\end{equation}
Moreover, the coefficients $c_j^{\pm}\in\R$ if $\lambda_j\geqslant-\frac{(n-2)^2}{4}$, and $c_j^{\pm}\in\CC$, $c_j^-=\overline{c_j^+}$ if $\lambda_j<-\frac{(n-2)^2}{4}$.

\begin{defn}
    Let $u:\Omega\rightarrow M$ be an \textbf{HSI} and $\tau>-1$. Consider the Jacobi operator $L_u:\hat{H}^{1;\tau}(\reg u,T_uM)\rightarrow H^{-1;\tau-2}(\reg u,T_uM)$. We define
    \begin{enumerate}
        \item $\widehat{\ker}_{\tau}L_u$ to be the space of \emph{stationary} Jacobi fields in $\hat{H}^{1;\tau}_0(\reg u,T_uM)$;
        \item $\widehat{\operatorname{coker}}_{\tau}L_u$ to be $H^{-1;\tau-2}(\reg u,T_uM)/L_u(\hat{H}_0^{1;\tau}(\reg u,T_uM))$;
        \item the index of $L_u$ to be $$\widehat{\ind}_{\tau}L_u=\dim\widehat{\ker}_{\tau}L_u-\dim \widehat{\operatorname{coker}}_{\tau}L_u,$$
        whenever it is finite.
    \end{enumerate}
    Finally, the kernel, cokernel and index of $L_u$ on $H^{1;\tau}_0(\reg u,T_uM)$ will be denoted by $\ker_{\tau}L_u$, $\operatorname{coker}_{\tau}L_u$, $\ind_{\tau}L_u$ respectively.
\end{defn}

\begin{rmk}
    So defined, when $n=3$, $\widehat{\ind}_{\tau}L_u$ is not the real Fredholm index of the operator $$L_u:\hat{H}_0^{1;\tau}(\reg u,T_uM)\rightarrow H^{-1;\tau-2}(\reg u,T_uM),$$ as the kernel of $L_u$ does not naturally enforce the stationarity condition. To realize $\widehat{\ind}_{\tau}L_u$ as the Fredholm index of a linear operator, we need to introduce extra coordinates as follows. For an \textbf{HSI} $u$, write $\sing u=\lbrace x_i\rbrace_{i=1}^Q$. For each $i$, choose a cutoff function $\zeta_i\in C_c^{\infty}(B_{2r_0}(x_i),[0,1])$, so that $\zeta_i|_{B_{r_0}(x_i)}=1$, where $r_0$ is a small number such that $B_{10r_0}(x_i)$ does not touch $\p\Omega$ and does not contain other singular points. We then define
    $$\tilde L_u:X\mapsto(L_uX,\mathcal{F}^{\alpha}_i(X)),\hat{H}_0^{1;\tau}(\reg u,T_uM)\rightarrow H^{-1;\tau-2}(\reg u,T_uM)\oplus\R^{3Q},$$
    where we fix $\lbrace e_{\lambda}\rbrace_{\lambda=1}^3$ be an orthonormal frame of $\R^3$ and set $$\mathcal{F}^{\lambda}_i(X)=\int_{\Omega}\left(\langle\nabla u,\nabla  X\rangle\operatorname{div}(\zeta_ie_{\lambda})-\sum_{\alpha}^3(\langle\p_{\alpha} u,\p_{\lambda}X\rangle+\langle\p_{\lambda}u,\p_{\alpha}X\rangle)\p_{\alpha}\zeta_i\right)\dif x.$$
    It is easy to see, using the fact that $\tau>-1$, that the kernel and cokernel of $\tilde L_u$ are exactly $\widehat{\ker}_{\tau}L_u$ and $\widehat{\operatorname{coker}}_{\tau}L_u$. However, in order to make the theory consistent, and as in Section \ref{s: meagerness}, what we really care about is the difference $\dim\widehat{\ker}_{\tau}L_u-\dim \widehat{\operatorname{coker}}_{\tau}L_u$, we only work with the operator $L_u$ and do not introduce a new operator $\tilde L_u$ in the rest of this paper.
\end{rmk}

We first calculate the index of the Jacobi operator of a regular 0-homogeneous map.

\begin{prop}\label{prop: index of tangent map}
    Let $h:\R^n\rightarrow M$ be a non-constant regular 0-homogeneous stationary harmonic map and $\tau\in(\gamma_-(h),0)$. Then the Jacobi operator $L_h:\hat{H}^{1;\tau}_{0}(B_1\setminus\lbrace0\rbrace,T_hM)\rightarrow H^{-1;\tau-2}(B_1\setminus\lbrace0\rbrace, T_hM)$ has index $-I(h)$.
\end{prop}
\begin{proof}
    We first compute the dimension of $\widehat\ker_{\tau}L_h$ for $n\geqslant 4$. In this case the stationarity condition is automatic by Proposition \ref{prop: stationarity for n greater than 3}.  Let $X\in\widehat\ker_{\tau}L_h$ and write the polar coordinates on $\R^n$ as $x=r\omega$. Hence, we can write $(X-\p_vh)\in H^{1;\tau}(\reg h,T_hM)$ for some $v\in\R^n$. In order to keep the zero trace of the vector field, we note that $r^{4-n}\p_vh,n\geqslant5$ and $\log r\p_vh,n=4$ is the only other choice of Jacobi field with same boundary value as $\p_vh$. But as $3-n\leqslant-1$ while $\tau>\gamma_-(h)\geqslant-1$, this term does not belong to $H^{1;\tau}(B_1\setminus\lbrace0\rbrace, T_hM)$ if $v\ne0$. 
    
    As a result, $v=0$ and $X\in H^{1;\tau}_0(\reg h,T_hM)$. Recall the representation formulas \eqref{eq: expansion of Jacobi field} and \eqref{eq: coefficients in the expansion of Jacobi fields}. In order that $X|_{\p B_1}=0$, we must have $\eta_j^+(1)+\eta_j^-(1)=0$ for each $j$ in the expansion. As such, any such $X$ can be written as $X(r\omega)=\sum_{j=1}^{\infty}\eta_j(r)X_j(\omega)$, where 
    \begin{align*}
        \eta_j(r)=\begin{cases}
            c_j(r^{\gamma_j^+}-r^{\gamma_j^-}),&\text{ if }\lambda_j>-\frac{(n-2)^2}{4};\\
            c_jr^{-(n-2)/2}\log r,&\text{ if }\lambda_j=-\frac{(n-2)^2}{4};\\
            c_jr^{-(n-2)/2}\sin\left(\sqrt{-(n-2)^2/4-\lambda_j}\log r\right),&\text{ if }\lambda_j<-\frac{(n-2)^2}{4}.
        \end{cases}
    \end{align*}
    If there is a non-zero $c_j(r^{\gamma_j^+}-r^{\gamma_j^-})\in H^{1;\tau}(B_1\setminus\lbrace0\rbrace,T_hM)$, then $\gamma_j^->\tau$. But $\gamma_j^-<-(n-2)/2\leqslant-1$ while $\tau>\gamma_-(h)\geqslant-1$, this is impossible. Similarly, either of other two types in the expansion of $X$ are not allowed. We conclude that $\widehat{\ker}_{\tau}L_h=\lbrace 0\rbrace$.

    Next let us calculate $\dim\widehat{\operatorname{coker}}_{\tau}L_h$. To this end, we compute the kernel of adjoint operator $L_h^*$ whose domain is $H^{1;2-\tau-n}_0(B_1\setminus\lbrace0\rbrace,T_hM)$. A vector field $X\in\ker L_h^*$ is characterized by $0=\langle L_h^*X,Y\rangle_{L^2}=\langle X,L_hY\rangle_{L^2}$, for all $Y\in \hat{H}^{1;\tau}_{0}(B_1\setminus\lbrace0\rbrace,T_hM)$. Fix a cutoff function $\zeta\in C_c^{\infty}(B_1,[0,1])$ which equals to 1 on $B_{1/2}$. Recall that $\hat{H}^{1;\tau}_{0}(B_1\setminus\lbrace0\rbrace,T_hM)=H^{1;\tau}_{0}(B_1\setminus\lbrace0\rbrace,T_hM)\oplus H_{\text{Tr}}$, where $H_{\text{Tr}}=\lbrace\zeta\nabla_vh:v\in\R^n\rbrace$. By taking $Y\in H^{1;\tau}_{0}(B_1\setminus\lbrace0\rbrace,T_hM)$, we see that $X\in\ker_{2-\tau-n}L_h$. By taking  $Y=\zeta\nabla_vh\in H_{\text{Tr}}$, we see that $$0=\langle X,L_h\zeta\nabla_vh\rangle_{L^2}=\langle X,L_h((\zeta-1)\nabla_vh)\rangle_{L^2}.$$
    Note that $(1-\zeta)\nabla_vh$ vanishes on $B_{1/2}$ where $X$ may fail to be in $H^1(B_1)$. Hence, we can apply Green's formula to get
    $$0=\int_{B_1}(\langle L_hX,(1-\zeta)\nabla_vh\rangle-\langle X,L_h((1-\zeta)\nabla_vh)\rangle\dif x=\int_{\p B_1}\langle\p_r X,\nabla_vh\rangle\dif\omega.$$
    Equivalently, 
    $$(\widehat{\operatorname{coker}}_{\tau}L_h)^*=\ker_{2-\tau-n}L_h\cap\lbrace X\in H^{1;2-\tau-n}_0(B_1\setminus\lbrace0\rbrace,T_hM):\langle \p_rX,\nabla_v h\rangle_{L^2(\p B_1)}=0\text{ for all }v\in\R^n\rbrace.$$
    As in the calculation in the $\widehat{\ker}_{\tau}L_h$, we can write $X(r\omega)=\sum_{j=1}^{\infty}\eta_j(r)X_j(\omega)$. For the coefficient $\eta_j(r)$ to be non-zero, we must have $\operatorname{Re}(\gamma_j^{-})>2-n-\tau$, or equivalently $\operatorname{Re}(\gamma^+_j)=2-n-\operatorname{Re}(\gamma_j^{-})<\tau$. By the choice of $\gamma_-(h)<\tau<0$, these are exactly the ones with $\lambda_j<0$. 
    
    On the other hand, by the orthogonality of $X_j$ for different $j$, what we can see in the product $\langle\p_r X,\nabla_vh\rangle_{L^2(\p B_1)}$ are of the form $\nabla_vh(\omega)$, which corresponds to $\gamma_j^+=-1<\tau$. The dimension of $\lbrace\nabla_vh(\omega):v\in\R^n\rbrace$ is $n$, which is the dimension of space excluded by $\langle\p_r X,\nabla_vh\rangle_{L^2(\p B_1)}=0$. In summary 
    \begin{align*}
        (\widehat{\operatorname{coker}}_{\tau}L_h)^*=&\ker_{2-\tau-n}L_h\cap\lbrace X\in H^{1;2-\tau-n}_0(B_1\setminus\lbrace0\rbrace,T_hM):\langle \p_rX,\nabla_v h\rangle_{L^2(\p B_1)}=0\text{ for all }v\in\R^n\rbrace\\=&\left\lbrace\sum_{\lambda_j<0}\eta_j(r)X_j(\omega):X_j(\omega)\perp\lbrace\nabla_vh(\omega):v\in\R^n\rbrace\right\rbrace,
    \end{align*}
    the dimension of which is $\dim \widehat {\operatorname{coker}}_{\tau}L_h=\dim(\widehat{\operatorname{coker}}_{\tau}L_h)^*=\#(\lbrace\lambda_j\rbrace\cap(-\infty,0))-n=I(h)$. As a consequence,
    $$\widehat{\ind}_{\tau}L_h=\dim\widehat{\ker}_{\tau}L_h-\dim \widehat{\operatorname{coker}}_{\tau}L_h=0-I(h)=-I(h).$$

    Finally let us look at the case $n=3$. This time, $\gamma_j^-=-1$ and $\gamma_j^+=0$ are the growth rates for $\nabla_vh$ and $r\nabla_vh$, which makes $(1-r)\nabla_vh\in\hat H^{1;\tau}_0(B_1\setminus\lbrace0\rbrace,T_hM)$. But $\nabla_vh$ is a stationary Jacobi field corresponding to the translation of $h$ while $r\nabla_vh$ is not by Proposition \ref{prop: stationarity when n=3}. Hence, again $\widehat{\ker}_{\tau}L_h=\lbrace0\rbrace$. On the other hand, $(1-r)\nabla_vh$ corresponds to $\lambda_j=0$ which makes them no longer in $(\widehat{\operatorname{coker}}_{\tau}L_h)^*$. As such, no vector field is excluded by condition $\langle\p_r X,\nabla_vh\rangle_{L^2(\p B_1)}=0$. We see $\dim\widehat{\operatorname{coker}}_{\tau}L_h=\#(\lbrace\lambda_j\rbrace\cap(-\infty,0))=I(h)$. As a result, once more we have
    $$\widehat{\ind}_{\tau}L_h=\dim\widehat{\ker}_{\tau}L_h-\dim \widehat{\operatorname{coker}}_{\tau}L_h=0-I(h)=-I(h).$$
\end{proof}

We can sum up the effect of each singular point to compute the total index of the Jacobi operator of an \textbf{HSI}.

\begin{thm}\label{thm: count of index}
    Let $u$ be an \textbf{HSI} and 
    \begin{equation}\label{eq: range of rate}
        \sup_{x_0\in\sing u}\gamma_-(h_{x_0})<\tau<0.\tag{*}
    \end{equation}
    Then the Jacobi operator $L_u:\hat{H}^{1;\tau}_{0}(\reg u,T_uM)\rightarrow \hat{H}^{-1;\tau-2}(\reg u, T_uM)$ is a Fredholm operator with index $$\widehat{\operatorname{index}}_{\tau}L_u=-\sum_{x_0\in\sing u}I(h_{x_0}).$$
\end{thm}
We start with the following localization lemma.
\begin{lem}\label{lem: decomposition of index}
    Let $u,\tau$ be as in Theorem \ref{thm: count of index} and write $\sing u=\lbrace x_i\rbrace_{i=1}^Q$. For any $r_0$ small such that $B_{10r_0}(x_i)\Subset \Omega$ mutually disjoint for different $i$, there holds $$\widehat{\operatorname{index}}_{\tau}L_u=\sum_{i=1}^Q\widehat{\operatorname{index}}_{\tau}(L_u|_{B_{r_0}(x_i)}), $$where $L_u|_{B_{r_0}(x_i)}$ is the operator $L_u:\hat{H}^{1;\tau}_{0}(B_{r_0}(x_i)\setminus\lbrace x_i\rbrace,T_uM)\rightarrow H^{-1;\tau-2}(B_{r_0}(x_i)\setminus\lbrace x_i\rbrace, T_uM)$.
\end{lem}
\begin{proof}
    Consider the bounded linear extension operator $$\mathcal{E}:H^{1/2}(\cup_{i=1}^Q\p B_{r_0}(x_i),T_uM)\rightarrow H^1_0(\cup_{i=1}^QA_{r_0/2,2r_0}(x_i),T_uM).$$
    Then, we can write $\hat{H}^{1;\tau}_0(\reg u,T_uM)$ as the direct sum
    \[
    \begin{array}{@{}c@{}c@{\;\;}c@{\;\;}c@{\;\;}c@{\;\;}c}
        &H_0^1(\Omega\setminus\cup_{i=1}^Q B_{r_0}(x_i),T_uM)
        &\oplus
        &\displaystyle\bigoplus_{i=1}^Q\hat H_0^{1;\tau}(B_{r_0}(x_i),T_uM)
        &\oplus
        &\displaystyle\mathcal E H^{1/2}(\cup_{i=1}^Q\partial B_{r_0}(x_i),T_uM)\\\llap{$=:\;$}
        &\mathscr{D}_1&\oplus&\mathscr{D}_2&\oplus&\mathscr{D}_3.
    \end{array}
    \]  
    Take the dual space of this direct sum at $\tau'=2-\tau-n$, and the fact that  $H^{-1;\tau-2}(\reg u,T_uM)=(H^{1;2-\tau-n}_0(\reg u,T_uM))^*$ gives
    \[
    \begin{array}{@{}c@{}c@{\;\;}c@{\;\;}c@{\;\;}c@{\;\;}c}
        &H^{-1}(\Omega\setminus\cup_{i=1}^QB_{r_0}(x_i),T_uM)
        &\oplus
        &\displaystyle\bigoplus_{i=1}^QH^{-1;\tau-2}(B_{r_0}(x_i),T_uM)
        &\oplus
        &\displaystyle\mathcal (\mathcal{E}H^{1/2}(\cup_{i=1}^Q\p B_{r_0}(x_i),T_uM))^*\\\llap{$=:\;$}
        &\mathscr{R}_1&\oplus&\mathscr{R}_2&\oplus&\mathscr{R}_3.
    \end{array}
    \]
     We have the decomposition of $L_u$ accordingly, $$L_u=L_1+L_2+L_3+R+R',$$
     where $$L_i=\pi_{\mathscr{R}_i}L_u|_{\mathscr{D}_i}:\mathscr{D}_i\rightarrow\mathscr{R}_i\text{ for }i=1,2,3\text{ and }R=\pi_{\mathscr{R}_1\oplus\mathscr{R}_2}L_u|_{\mathscr{D}_3},R'=\pi_{\mathscr{R}_3}L_u|_{\mathscr{D}_1\oplus\mathscr{D}_2}.$$
    \begin{claim}
        We can choose the extension operator $\mathcal{E}$ to make $R$ and $R'$ compact operators.
    \end{claim}
    \begin{proof}[Proof of Claim]
        We build the extension operator as follows. We pick a large number $\lambda$ which makes $L_u-\lambda$ invertible on $A_{r_0/2,2r_0}(x_i)$ for each $i$. For each $\psi\in H^{1/2}(\cup_{i=1}^Q\p B_{r_0}(x_i),T_uM)$ consider the solution 
        \begin{align*}
            \begin{cases}
                L_uf^+=\lambda f^+,&\text{ in }\cup_{i=1}^QA_{r_0,2r_0}(x_i);\\
                f^+=\psi,&\text{ on }\cup_{i=1}^Q\p B_{r_0}(x_i);\\
                f^+=0,&\text{ on }\cup_{i=1}^Q\p B_{2r_0}(x_i);
            \end{cases}\quad\text{and}\quad
                    \begin{cases}
                L_uf^-=\lambda f^-,&\text{ in }\cup_{i=1}^QA_{r_0/2,r_0}(x_i);\\
                f^-=\psi,&\text{ on }\cup_{i=1}^Q\p B_{r_0}(x_i);\\
                f^-=0,&\text{ on }\cup_{i=1}^Q\p B_{r_0/2}(x_i).
            \end{cases}
        \end{align*}
        Then let us take a cutoff function $\zeta\in C_c^{\infty}(\cup_{i=1}^QA_{r_0/2,2r_0}(x_i),[0,1])$ which equals 1 on $A_{3r_0/4,3r_0/2}(x_i)$ for each $i$. We then define $\mathcal{E}\psi$ to equal $\zeta f^+,\zeta f^-$ on $\cup_{i=1}^Q A_{r_0,2r_0}(x_i),\cup_{i=1}^QA_{r_0/2,r_0}(x_i)$ respectively. So defined, we have $\mathcal{E}:H^{1/2}(\cup_{i=1}^Q\p B_{r_0}(x_i),T_uM)\rightarrow H^1_0(\cup_{i=1}^QA_{r_0/2,2r_0}(x_i),T_uM)$ is a linear operator and bounded by elliptic regularity theory.

        We next show that this choice of $\mathcal{E}$ makes $R$ a compact operator. Note that $R$ is characterized by the following: for each $X\in H_0^1(\Omega\setminus\cup_{i=1}^Q B_{r_0}(x_i),T_uM)\oplus\bigoplus_{i=1}^Q H_0^{1;2-\tau-n}(B_{r_0}(x_i),T_uM)$. We calculate
        \begin{align*}
            &R\mathcal{E}\psi(X)\\=&\int_{\Omega}(\langle\nabla\mathcal{E}\psi,\nabla X\rangle-\langle DA_u(\mathcal{E}\psi,\nabla u,\nabla u),X\rangle)\dif x\\=&-\sum_{i=1}^Q\left(\int_{A_{r_0,2r_0}(x_i)}\langle L_u(\zeta f^+),X\rangle\dif x+\int_{A_{r_0/2,r_0}(x_i)}\langle L_u(\zeta f^-),X\rangle\dif x\right)\\=&-\sum_{i=1}^Q\left(\int_{A_{r_0,2r_0}(x_i)}\langle (\Delta\zeta+\lambda\zeta)f^++2\nabla_{\nabla\zeta}f^+,X\rangle\dif x+\int_{A_{r_0/2,r_0}(x_i)}\langle (\Delta\zeta+\lambda\zeta)f^-+2\nabla_{\nabla\zeta}f^-,X\rangle\dif x\right).
        \end{align*}
        As a result, $\vert R\mathcal{E}\psi(X)\vert\leqslant C\Vert\mathcal{E}\psi\Vert_{L^2(\Omega)}\Vert X\Vert_{H^1(\cup_{i=1}^QA_{r_0/2,2r_0}(x_i))}$. By Sobolev embedding, $$\mathcal{E}:H^{1/2}(\cup_{i=1}^Q\p B_{r_0}(x_i),T_uM)\rightarrow L^2(\cup_{i=1}^QA_{r_0/2,2r_0}(x_i),T_uM),$$ is compact. This implies $R$ is compact. A similar duality argument shows that $R'$ is compact.
    \end{proof}
    
    Once we know $R$ is compact, we know that the $\widehat{\ind}_{\tau}L_u$ equals to the sum of indices of $L_i,1\leqslant i\leqslant3$. Next let us show that $L_1$ and $L_3$ have index 0.

    $L_1=L_u:H^1_0(\Omega\setminus\cup_{i=1}^QB_{r_0}(x_i),T_uM)\rightarrow H^{-1}(\Omega\setminus\cup_{i=1}^QB_{r_0}(x_i),T_uM)$ is a uniformly elliptic operator with smooth coefficients. For large $\lambda$, $L_1-\lambda$ is invertible. Hence, $L_1$ differs from an invertible operator by a compact operator $X\mapsto\lambda X$, which makes the index of $L_1$ equals to 0. Since $\mathscr{R}_3=\mathscr{D}_3^*$, $L_3$ is characterized by $$-L_3\mathcal{E}\psi(\mathcal{E}\bar\psi)=\int_{\Omega}(\langle\nabla\mathcal{E}\psi,\nabla\mathcal{E}\bar\psi\rangle-\langle DA_u(\mathcal{E}\psi,\nabla u,\nabla u),\mathcal{E}\bar\psi\rangle)\dif x.$$
    Again, for a large $\lambda$ there holds 
    $$c\Vert\mathcal{E}\psi\Vert_{H^1(\Omega)}^2\leqslant -L_3\mathcal{E}\psi(\mathcal{E}\psi)+\lambda\Vert\mathcal{E}\psi\Vert_{L^2(\Omega)}^2\leqslant C\Vert\mathcal{E}\psi\Vert_{H^1(\Omega)}^2.$$
    By Lax-Milgram theorem, $L_3-\lambda I:\mathcal{D}_3\rightarrow\mathcal{R}_3=\mathscr{D}_3^*$ is an isomorphism, where $I: D_3\rightarrow D_3^*$ is defined by $IX(Y)=\int_{\Omega}\langle X,Y\rangle\dif x$. $I$ is compact, which makes $L_3$ has index zero.

    Finally, since $\lbrace B_{r_0}(x_i)\rbrace_{i=1}^Q$ are mutually disjoint balls, the index of $L_2$ equals $\sum_{i=1}^Q\widehat{\operatorname{index}}_{\tau}(L_u|_{B_{r_0}(x_i)})$. We have proved the lemma.
\end{proof}

We also need the following fact from functional analysis.

\begin{lem}\label{lem: difference of index of restriction}
    Let $V,W$ be Banach spaces and $L:V\rightarrow W$ be a bounded linear operator. Let $V_0\subset V$ be a closed subspace of finite codimension $d\in\N$ and denote by $L_0$ the restriction of $L$ to $V_0$. Then $L$ being a Fredholm operator is equivalent to $L_0$ being a Fredholm operator, and $$\ind L=\ind L_0+d.$$
\end{lem}

\begin{proof}[Proof of Theorem \ref{thm: count of index}]
    By Proposition \ref{prop: index of tangent map} and Lemma \ref{lem: decomposition of index}, we only need to show that if $r_0$ is chosen small enough, then 
    $$\widehat{\ind}_{\tau}(L_u|_{B_{r_0}(x_i)})=\widehat{\ind}_{\tau}(L_h|_{B_1}).$$ By Lemma \ref{lem: difference of index of restriction}, this is equivalent to 
    $$\ind_{\tau}(L_u|_{B_{r_0}(x_i)})=\ind_{\tau}(L_h|_{B_1}).$$ Fix $x_0\in\sing u$. Translating, we assume $x_0=0$ and write $h_{x_0}=h$. 
    
    To start with, we note that there exists $d_0>0$ which may depend on the geometry of $M$, such that if $p,q\in M$ satisfy $\vert p-q\vert<d_0$, then the orthogonal projection $\pi_{T_qM}|_{T_pM}:T_pM\rightarrow T_qM$ is a bi-Lipschitz isomorphism, with bi-Lipschitz constants bounded from above and below by 2, $1/2$ respectively. By Corollary \ref{Cor: uniform graph on tangent maps}, there holds
    $$\lim_{r\rightarrow0}\sup_{x\in B_r}\sum_{i=0}^2\vert x\vert^i\vert\nabla^i(u-h)\vert=0$$
    Hence, for $r_0$ sufficiently small, $\Vert u-h\Vert_{L^{\infty}(B_{r_0})}<d_0$ which ensures that the orthogonal projection $\pi:H^{-1;\tau'}(B_{r_0}\setminus\lbrace 0\rbrace,T_uM)\rightarrow H^{-1;\tau'}(B_{r_0}\setminus\lbrace 0\rbrace,T_hM),(\pi X)(x)=\pi_{T_{h(x)}M}X(x)$ is a bi-Lipschitz isomorphism, for any $\tau'\in\R$. Now, we calculate the difference
    $$L_u-\pi ^{-1}L_h\pi=A\nabla+B,\text{ where }\lim_{r\rightarrow0}\sup_{x\in B_r}(\vert x\vert\vert A\vert+\vert x\vert^2\vert B\vert)=0.$$
    We now show that $K=A\nabla+B:H^{1;\tau}_0(B_{r_0}\setminus\lbrace 0\rbrace,T_uM)\rightarrow H^{-1;\tau-2}(B_{r_0}\setminus\lbrace 0\rbrace,T_uM)$ is a compact operator. As an element in $(H^{1;2-\tau-n}_0(B_{r_0}\setminus\lbrace 0\rbrace,T_uM))^*$, this operator is characterized by
    \begin{equation}
        \begin{aligned}
            &\vert(KX)(Y)\vert=\left\vert\int_{B_{r_0}}\langle A\nabla X+BX,Y\rangle\dif x\right\vert\\=&\left\vert\int_{B_{r_0}}\langle (\vert x\vert A)\vert x\vert^{1-\tau-n/2}\nabla X+(\vert x\vert^2B)\vert x\vert^{-\tau-n/2}X,\vert x\vert^{-(2-\tau-n)-n/2}Y\rangle\dif x\right\vert\\\leqslant&\sup_{B_{r_0}}(\vert x\vert \vert A\vert+\vert x\vert^2\vert B\vert)\Vert X\Vert_{H^{1;\tau}(B_{r_0})}\Vert Y\Vert_{L^{2;2-\tau-n}(B_{r_0})}.
    \end{aligned}\label{eq: compact operator in section 3}
    \end{equation}
    Let us choose a cutoff function $\zeta_{r}\in C_{c}^{\infty}(B_{2r},[0,1])$ which equals to 1 on $B_r$. Consider $K_{\delta}=(1-\zeta_{\delta})K$. We decompose $K_{\delta}=\iota \bar K_{\delta}$, where 
    $$\bar K_{\delta}=(1-\zeta_{\delta})(A\nabla+B):H^{1;\tau}_0(B_{r_0}\setminus\lbrace 0\rbrace,T_uM)\rightarrow V:=\lbrace X\in L^2(B_{r_0},T_uM): X|_{B_{\delta}}=0\rbrace,$$
    and $\iota: V\rightarrow H^{-1;\tau-2}(B_{r_0}\setminus\lbrace0\rbrace,T_uM)$ is the inclusion. $\iota$ is compact by Sobolev embedding while $\bar K_{\delta}$ is bounded, which makes $K_{\delta}$ compact. By \eqref{eq: compact operator in section 3}, $K$ is the operator norm limit of compact operators $K_{\delta}$, which makes $K$ itself compact.

    In other words, $L_u-\pi^{-1}L_h\pi$ is compact, which makes $\ind_{\tau}L_u$ equals to the index of $\pi^{-1}L_h\pi$. $\pi$ is the isomorphism from the kernel and cokernel of $\pi^{-1}L_h\pi$ to that of $L_h$, hence index of $\pi^{-1}L_h\pi$ equals to $\ind_{\tau}(L_h|_{B_{r_0}})$. Finally, by scaling invariance, $\ind_{\tau}(L_h|_{B_{r_0}})=\ind_{\tau}(L_h|_{B_1})$, which concludes the proof.
\end{proof}

In the application, instead of $\widehat{\text{coker}}_{\tau}L_u$, what we really care about is the following space. 

\begin{lem}\label{lem: equiv of dimension of coker}
    Let $u$ be an \textbf{HSI}, $\varphi=u|_{\p\Omega}\in C^{k,\alpha}(\p\Omega,M)$ and let $\tau$ satisfy \eqref{eq: range of rate}.
    
    Then there exists a linear subspace $\Psi\subset C^{k,\alpha}(\p\Omega,T_{\varphi}M)$ of dimension $\dim\widehat{\text{coker}}_{\tau}L_u$ with the following property. If $X\in \hat{H}^{1;\tau}(\reg u,T_uM)$ is a stationary Jacobi field with $X|_{\p\Omega}\in\Psi$, we must have $X|_{\p\Omega}=0$. 
\end{lem}
\begin{proof}
    Let us consider the dual space $(\widehat{\text{coker}}_{\tau}L_u)^*\subset H^{1;2-\tau-n}_0(\reg u,T_uM)$, which has dimension equal to $\dim\widehat{\text{coker}}_{\tau}L_u$ and is characterized by 
    $$(\widehat{\text{coker}}_{\tau}L_u)^*=\lbrace X\in H^{1;2-\tau-n}_0(\reg u,T_uM):L_uY(X)=0,\text{ for all }Y\in\hat H^{1;\tau}_0(\reg u,T_uM)\rbrace.$$
    In particular, $X\in(\widehat{\text{coker}}_{\tau}L_u)^*$ is smooth and solves $L_uX=0$ on $\reg u$. The unique continuation together with elliptic regularity gives $X\mapsto \p_{\nu}X,(\widehat{\text{coker}}_{\tau}L_u)^*\rightarrow C^{k,\alpha}(\p\Omega,T_{\varphi}M)$ is injective. Hence, the set
    $$\Psi=\lbrace \p_{\nu}X:X\in(\widehat{\text{coker}}_{\tau}L_u)^*\rbrace,$$
    is a set with same dimension as $\widehat{\text{coker}}_{\tau}L_u$. We claim that $\Psi$ has the desired property. Indeed, suppose $X$ is a stationary Jacobi field with $X|_{\p\Omega}\in\Psi$, then for some $\bar X\in(\widehat{\text{coker}}_{\tau}L_u)^*$, there holds $X|_{\p\Omega}=\p_{\nu}\bar X$. Pick a cutoff function $\zeta$ which equals 1 near $\p\Omega$ and $0$ near $\sing u$. By Green's formula and the definition of $(\widehat{\text{coker}}_{\tau}L_u)^*$, there holds
    \begin{align*}
        0=L_u((\zeta-1)X)(\bar X)=L_u(\zeta X)(\bar X)=\int_{\Omega}\langle L_u(\zeta X),\bar X\rangle\dif x=&\int_{\Omega}(\langle L_u(\zeta X),\bar X\rangle-\langle\zeta X,L_u\bar X\rangle)\dif x\\=&-\int_{\p\Omega}\langle\p_{\nu}\bar X,X\rangle\dif\sigma=-\int_{\p\Omega}\vert \p_{\nu}\bar X\vert^2\dif\sigma.
    \end{align*}
    Therefore, $\p_{\nu}\bar X=X|_{\p\Omega}=0$, as desired.
\end{proof}

\section{Meagerness of singular points with non-zero index}\label{s: meagerness}
In this section, we wish to prove the following theorem. Throughout this section, the $k$ in boundary regularity is assumed to be at least 2.
\begin{thm}\label{thm: genericity of Jacobi operator}
    For each $k\in\N_{\geqslant2},\alpha\in(0,1)$, there exists a generic subset $\mathcal{B}$ of $C^{k,\alpha}(\p\Omega,M)$ or $C^{\infty}(\p\Omega,M)$, such that for each $\varphi\in \mathcal{B}$, and $\tau$ satisfying \eqref{eq: range of rate}, every \textbf{HSI} $u$ with boundary value $u|_{\p\Omega}=\varphi$ satisfies
    $$\widehat{\ind}_{\tau}L_u=0.$$
\end{thm}

Recall that, all harmonic maps below belong to one of three classes: energy minimizing/ stable stationary/ stationary harmonic maps with strong convergence assumption \eqref{eq: strong convergence}.

\subsection{Compact subclass of regular harmonic tangent maps}
\begin{defn}
    Let us set $\mathcal{T}_n$ be the set of all non-constant regular 0-homogeneous stationary harmonic map from $\R^n$ to $M$. Moreover, for each $\Lambda>0$, set $$\mathcal{T}_n(\Lambda)=\left\lbrace h\in\mathcal{T}_n:\inf_{x\in \Sp^{n-1}}r_h(x)\geqslant\Lambda^{-1}\right\rbrace.$$
\end{defn}
By the definition of regularity scale, the density $\Theta_h(0)$ is bounded by a constant depending on $\Lambda$ for maps $h\in\mathcal{T}_n(\Lambda)$. Let us set $\Theta(\Lambda)$ be the supremum of this density for $h\in\mathcal{T}_n(\Lambda)$.

By definition, $\mathcal{T}_n(\Lambda)$ is compact under $H^1(B_1)$ norm and the $C^{\infty}$ topology of links on $\Sp^{n-1}$. Moreover, if $h_k\in\mathcal{T}_n(\Lambda)$ and $h_k\rightarrow h$ as $k\rightarrow\infty$, by the smooth convergence of links $h_k|_{\Sp^{n-1}}$ we have that $\gamma^{\pm}_j(h_k)\rightarrow\gamma_j^{\pm}(h)$ as $k\rightarrow\infty$. Moreover, \L ojasiewicz-Simon inequality implies the set of possible densities for an $h\in\mathcal{T}_n(\Lambda)$ is finite.
\begin{lem}\label{lem: discreteness of density}
    $\lbrace\Theta_h(0):h\in\mathcal{T}_n(\Lambda)\rbrace$ is a finite set for each $\Lambda>0$.
\end{lem}
\begin{proof}
    Suppose the lemma is not true, we can find a sequence of maps $h_k\in\mathcal{T}_n(\Lambda)$, such that $\Theta_{h_k}(0)$ are all distinct numbers. After passing to a subsequence, we can assume $h_k\rightarrow h$ strongly in $H^1(B_2)$ and $u_k\rightarrow u$ smoothly on $\Sp^{n-1}$, where $u_k$ and $u$ are the restrictions of $h_k$ and $h$ on $\Sp^{n-1}$ respectively.

    The smooth convergence allows us to apply the \L ojasiewicz-Simon inequality \cite[Section 3.14]{SimonBook} to get
    $$\left\vert\int_{\Sp^{n-1}}\vert\nabla u_k\vert^2-\int_{\Sp^{n-1}}\vert\nabla u\vert^2\right\vert^{1-\alpha}\leqslant C\Vert\tau_{\Sp^{n-1}}(u_k)\Vert_{L^2(\Sp^{n-1})}=0,\text{ for } k\text{ sufficiently large,}$$
    where $\tau_{S^{n-1}}(u_k)=\Delta_{S^{n-1}}u_k+A_{u_k}(\nabla^{S^{n-1}}u_k,\nabla^{S^{n-1}}u_k)$ is the tension field of $u_k$. As a result, we see
    $$\Theta_{h_k}(0)=\frac{1}{n-2}\int_{\Sp^{n-1}}\vert \nabla u_k\vert^2=\frac{1}{n-2}\int_{\Sp^{n-1}}\vert\nabla u\vert^2=\Theta_h(0),\text{ for large }k.$$
    This contradicts the assumption that $\Theta_{h_k}(0)$ are all distinct.
\end{proof}
We are also able to derive a quantitative lower bound on the $L^2$ norm of partial derivatives. 
\begin{lem}\label{lem: lower bound of derivative of tangent map}
    There exists a constant $C>0$ depending only on $n,\Lambda,M$ such that for each $h\in\mathcal{T}_n(\Lambda)$, we have 
    $$\vert v\vert\leqslant C\Vert\nabla_vh\Vert_{L^2(B_1)}.$$
\end{lem}
\begin{proof}
    Suppose this inequality is not true, there exists a sequence of unit vectors $e_k$, and maps $h_k\in\mathcal{T}_n(\Lambda)$ such that $$\Vert\nabla_{e_k}h_k\Vert_{L^2(B_1)}\rightarrow0.$$ Assume $h_k\rightarrow h\in\mathcal{T}_n(\Lambda)$ and $e_k\rightarrow e$ as $k\rightarrow\infty$, we see from above assumption that $\nabla_eh=0$. However, this implies that the singular set of $h$ contains at least the line spanned by $e$, which contradicts our assumption that the singular set of $h\in\mathcal{T}_n(\Lambda)$ contains only $0$.
\end{proof}

\subsection{Key definitions and tools}
We will work with following spaces when proving our results.
\begin{defn}
    Let us denote by $\Hau^{k,\alpha}$ the set of pairs $(u,\varphi)$ where $u\in H^1(\Omega,M)$ is an \textbf{HSI} and $\varphi\in C^{k,\alpha}(\p\Omega,M)$ is the boundary value of $u$, $u|_{\p\Omega}=\varphi$. We also denote the projection to boundary value by $$\Pi:\mathcal{H}^{k,\alpha}(\Omega,M)\rightarrow C^{k,\alpha}(\p\Omega,M),(u,\varphi)\mapsto\varphi.$$ 
\end{defn}
The convergence in $\Hau^{k,\alpha}$ is understood as the $H^1$ convergence in the first factor and the $C^{k,\alpha}$ convergence in the second factor. Note that under our assumption of strong convergence of harmonic maps, boundary regularity holds automatically for $u\in\Hau^{k,\alpha}$ by the analysis established by F. Lin \cite{LinGradientEstimate} and D. Hsu \cite{Hsu}, and the fact that any harmonic map $u$ from the half space $\lbrace (x_,y)\in\R^2:y\geqslant0\rbrace$ to $M$ with $u(x,0)=$ Constant itself must be a constant; see \cite{EellsLemaire} and the references therein. 

\begin{defn}\label{defn: tame field}
    Let $u$ be an \textbf{HSI}. A section $X$ in $\hat{L}^{2;\tau}$ is called a \textbf{$\tau$-growth} section. Moreover, if $\tau$ satisfies \eqref{eq: range of rate}, we call $X$ a \textbf{tame} section.
\end{defn}

We use the following asymptotic rate to control the local behavior of a section near the singular points. This notion is first introduced in Z. Wang \cite{WangDeformationI}.

\begin{defn}
    Let $u$ be an \textbf{HSI}, $x_0$ be a singular point of $u$, and an $L^2_{\text{loc}}$ section $X$ defined near $x_0$. The asymptotic rate of $X$ at $x_0$ is defined as $$\mathcal{AR}_{x_0}(X)=\sup\left\lbrace\gamma\in\R:\lim_{r\rightarrow0}\int_{A_{r,2r}(x_0)}\vert X(x)\vert^2\vert x-x_0\vert^{-n-2\gamma}\dif x=0\right\rbrace.$$
\end{defn}
Heuristically, $\mathcal{AR}_p(X)=\gamma$ means $X$ grows like $\rho^{\gamma}$ near $p$. 
\begin{lem}\label{lem: asymtpotic rate and growth rate relation}
    Let $\tau$ satisfy \eqref{eq: range of rate} and $X\in W^{2,2}_{\textup{loc}}(\reg u, T_uM)$ be a non-zero section which is $C^k$ on $\p\Omega$, $k\geqslant2$.
    \begin{enumerate}
        \item If $L_uX\in L^{2;\tau-2}(\reg u,T_uM)$, then $\mathcal{AR}_{x_0}(X)\in\lbrace-\infty\rbrace\cup\Gamma(h_{x_0})\cup\lbrace-(n-2)/2\rbrace\cup[\tau,\infty)$ for every $x_0\in\sing u$;
        \item If $X\in L^{2;\tau}(\reg u,T_uM)$ (resp. $\hat{L}^{2;\tau}(\reg u,T_uM)$), then $\mathcal{AR}_{x_0}(X)\geqslant\tau$ (resp. $\mathcal{AR}_{x_0}(X)\geqslant-1)$ for each $x_0\in\sing u$;
        \item Let $k\in\N,k\geqslant2$. If $\mathcal{AR}_{x_0}(X)\geqslant\tau$ for each $x_0\in\sing u$ and $L_uX\in H^{k-2;\tau-2}(\reg u,T_uM)$, then $X\in H^{k;\tau'}(\reg u,T_uM)$ for each $\tau'<\tau$.
    \end{enumerate}
\end{lem}
Note that we have the following inequality for asymptotic rate
$$\mathcal{AR}_{x_0}(X_1+X_2)\geqslant\min\lbrace\mathcal{AR}_{x_0}(X_1),\mathcal{AR}_{x_0}(X_2)\rbrace.$$
Moreover, we have the weighted version of Holder inequality: if $X_i\in L^{2;\tau_i}(\reg u,T_uM)$, then
\begin{align*}
    \begin{cases}
        X_1\cdot X_2\in L^1(\Omega),&\text{ for }\tau_1+\tau_2\geqslant-n;\\
        X_1\cdot X_2\in L^{1;\tau}(\reg u,T_uM),&\text{ for }\tau_1+\tau_2\geqslant\tau.
    \end{cases}
\end{align*}
Indeed, the corresponding norms are bounded by
\begin{align*}
    \Vert X_1\cdot X_2\Vert_{L^1(\Omega)}\leqslant C\Vert X_1\Vert_{L^{2;\tau_1}(\reg u)}\Vert X_2\Vert_{L^{2;\tau_2}(\reg u)},
\end{align*}
$$\Vert X_1\cdot X_2\Vert_{L^{1;\tau}(\Omega)}\leqslant C\Vert X_1\Vert_{L^{2;\tau_1}(\reg u)}\Vert X_2\Vert_{L^{2;\tau_2}(\reg u)}.$$
\begin{proof}[Proof of Lemma \ref{lem: asymtpotic rate and growth rate relation}]
    For part (1), we show that if $\gamma=\mathcal{AR}_{x_0}(X)\in(-\infty,\tau)$, then $\gamma\in\Gamma(h_{x_0})\cup\lbrace-(n-2)/2\rbrace$. If this does not hold, then we can pick $\gamma'<\gamma<\gamma''$ and $\sigma>0$, such that 
    $$[\gamma'-\sigma,\gamma''+\sigma]\cap(\Gamma(h_{x_0})\cup\lbrace-(n-2)/2\rbrace\cup[\tau,\infty))=\emptyset.$$
    By the definition of asymptotic rate
    $$\limsup_{r\rightarrow0}\int_{A_{r,2r}(x_0)}\vert X(x)\vert^2\vert x-x_0\vert^{-n-2\gamma''}\dif x=\infty.$$
    As such, we can pick a decreasing sequence $r_j\rightarrow0$ such that for every $r\in[r_j,r_1]$, there holds
    \begin{equation}\label{eq: lower bound 2 in asymptotic rate}
        \int_{A_{r_j,2r_j}(x_0)}\vert X(x)\vert^2\vert x-x_0\vert^{-n-2\gamma''}\dif x\geqslant\int_{A_{r,2r}(x_0)}\vert X(x)\vert^2\vert x-x_0\vert^{-n-2\gamma''}\dif x,
    \end{equation}
    and the integrals on the left hand side are positive for all $j\geqslant1$. Set $X_j=X(r_j\cdot+x_0)$. Above inequality at $r=r_1$ implies
    \begin{equation}\label{eq: lower bound of X_j in asymptotic rate}
        \begin{aligned}
            \Vert X_j\Vert_{L^2({A_{1,2}})}^2&\geqslant C(\gamma'')r_j^{2\gamma''}\int_{A_{r_j,2r_j}(x_0)}\vert X\vert^2\vert x-x_0\vert^{-n-2\gamma''}\dif x\\&\geqslant C(\gamma'')r_j^{2\gamma''}\int_{A_{r_1,2r_1}(x_0)}\vert X\vert^2\vert x-x_0\vert^{-n-2\gamma''}\dif x.
        \end{aligned}
    \end{equation}
    On the other hand, the equation of $X$ is $L_uX=f$ for $f\in L^{2;\tau-2}$. Rescaling, we see $X_j$ satisfies the equation $L_{u_j}X_j=f_j$, where $u_j=u(r_j\cdot+x_0)$ and $f_j=r_j^2f(r_j\cdot+x_0)$. For $j$ large enough, $u_j$ is close to the $h_{x_0}$ in the sense of Lemma \ref{lem: Jacobi operator growth estimate}. In this case, Lemma \ref{lem: Jacobi operator growth estimate} gives the growth estimate of $X_j$:
    \begin{equation}\label{eq: upper bound 1 in asymptotic rate}
        \Vert X_j\Vert_{L^2(A_{rK^{-1},r})}\leqslant C(\sigma,\Lambda)(F_j+\Vert X_j\Vert_{L^2(A_{K^{-1},1})})r^{n/2+\gamma'},
    \end{equation}
    where $\Lambda=\sup_{\omega\in\Sp^{n-1}}\vert\nabla h_{x_0}\vert$ and 
    $$F_j\leqslant \Vert\vert x\vert^{2-\gamma'-n/2}f_j\Vert_{L^2(B_1)}\leqslant C(\sigma,\Lambda,\tau)\Vert f\Vert_{L^{2;\tau-2}(B_{2r_1}(x_0)}r_j^{\tau}\leqslant C\Vert f\Vert_{L^{2;\tau-2}}\Vert X_j\Vert_{L^2(A_{1,2})}r_j^{\tau-\gamma''},$$
    where in the last step we used the equation \eqref{eq: lower bound of X_j in asymptotic rate}. The bound on $X_j$ allows us to consider the normalization $X_j/\Vert X_j\Vert_{L^2(A_{1,2})}$ which after passing to a subsequence converges to a non-trivial field $\hat X$ strongly in $L^2_{\loc}(B_2\setminus\{0\})$ and weakly in $H^1_{\loc}(B_2\setminus\{0\})$. The bound on $F_j$ implies that $\hat X$ is a Jacobi field of $h_{x_0}$. And \eqref{eq: lower bound 2 in asymptotic rate} and \eqref{eq: upper bound 1 in asymptotic rate} gives respectively
    $$\sup_{s>1}\int_{A_{s,2s}}\vert \hat X(x)\vert^2\vert x\vert^{-n-2\gamma''}\dif x<\infty\text{ and}\sup_{s\in(0,1)}\int_{A_{s,2s}}\vert \hat X(x)\vert^2\vert x\vert^{-n-2\gamma'}\dif x<\infty.$$
    But there is no such non-trivial Jacobi field of $h_{x_0}$ given the assumption $[\gamma',\gamma'']\cap \Gamma(h_{x_0})=\emptyset$. This is the desired contradiction.

    Part (2) of Lemma \ref{lem: asymtpotic rate and growth rate relation} follows directly from the definition and the fact that $\p_vu$ has asymptotic rate $-1$ at each $x_0\in\sing u$ for $v\ne0$. And (3) follows from by an argument similar to the derivation of \eqref{eq: lower bound of X_j in asymptotic rate} and standard elliptic regularity theory. 
\end{proof}

\begin{cor}\label{cor: asy rate of tame field}
    $\widehat{\ker}_{\tau}L_u$ is independent of $\tau$ satisfying \eqref{eq: range of rate}. If $X\in\widehat{\ker}_{\tau}L_u,$ then for every $p\in\sing u$, there exists $v\in\R^n$ such that $\mathcal{AR}_p(X-\p_vu)\geqslant0$.
\end{cor}
\begin{proof}
    By Definition \ref{defn: tame field} of tame field, we can decompose $X=X_0+\phi$ where $\phi$ is a translation-like section and $X_0\in L^{2;\tau}(\reg u)$ for $\tau$ satisfying \eqref{eq: range of rate}. Hence, $L_u\phi$ vanishes near each singular point. This implies $L_uX_0=0$ near each singular point. Hence Lemma \ref{lem: asymtpotic rate and growth rate relation} (1) and (2) applies to give $X_0$ has asymptotic rate at least $\tau'$ for each $\tau'\in(\tau,0)$. By the arbitrariness of $\tau'$, we have that $\mathcal{AR}_{x_0}(X_0)\geqslant0$ for each $x_0\in\sing u$. The fact $\widehat{\ker}_{\tau}L_u$ is independent of $\tau$ satisfying \eqref{eq: range of rate} is a corollary of Lemma \ref{lem: asymtpotic rate and growth rate relation} (3), which allows us to improve the growth rate of $X_0$ near each $x_0\in\sing u$.
\end{proof}

\subsection{Canonical neighborhoods}
Let us set
$$\mathcal{R}^{k,\alpha}(u,\varphi)=\lbrace\Lambda>0:\Vert \varphi\Vert_{C^{k,\alpha}(\p\Omega)}\leqslant\Lambda,r_u\geqslant\Lambda^{-1}\rho_u\rbrace,$$
and also the possible triples to form $\mathcal{R}^{k,\alpha}(u,\varphi)$,
$$\mathcal{T}^{k,\alpha}=\lbrace(u,\varphi,\Lambda):(u,\varphi)\in\mathcal{H}^{k,\alpha}(\Omega,M),\Lambda\in\mathcal{R}^{k,\alpha}(u,\varphi)\rbrace.$$
\begin{defn}\label{defn: definition of canonical neighborhood}
    The \textbf{canonical (pseudo-)neighborhood}, denoted by $\mathcal{L}^{k,\alpha}(u,\varphi;\Lambda,\delta)$ of $(u,\varphi)$ in $\mathcal{H}^{k,\alpha}(\Omega,M)$ is defined to be pairs $(u',\varphi')\in\Hau^{k,\alpha}(\Omega,M)$ that satisfy the following conditions.
    \begin{itemize}
        \item $\Vert\varphi'\Vert_{C^{k,\alpha}}\leqslant\Lambda$ and $\Vert\varphi-\varphi'\Vert_{C^k}\leqslant\delta$;
        \item $u'$ is an \textbf{HSI} with $r_{u'}\geqslant\Lambda^{-1}\rho_{u'}$ and $\Vert u-u'\Vert_{L^2(\Omega)}\leqslant\delta$;
        \item There is a correspondence between the singular sets $\sing u$ and $\sing u'$, such that for each $x_i\in\sing u,i=1,\dots,I$, the corresponding $x_i'\in\sing u'$ satisfies $$\Theta_{u'}(x_i')=\Theta_{u}(x_i)\text{, and }\vert x_i-x_i'\vert\leqslant\frac{1}{3}\min_{i\ne j}\vert x_i-x_j\vert.$$
    \end{itemize}
\end{defn}
Note that in general, $\mathcal{L}^{k,\alpha}(u,\varphi;\Lambda,\delta)$ is not a topological neighborhood of $(u,\varphi)$ in $\Hau^{k,\alpha}(\Omega,M)$, or even may not contain an open set containing $(u,\varphi)$. For example, there could be $(u',\varphi')$ arbitrarily close to $(u,\varphi)$ but not contained in $\mathcal{L}^{k,\alpha}(u,\varphi;\Lambda,\delta)$. This happens when multiple singular points collapse to a singular one. 

One important feature of the canonical neighborhood is that if a sequence $\{u_j\}$ in $\mathcal{L}^{k,\alpha}(u,\varphi)$ converges strongly in $L^2(\Omega)$ to $\hat u$, then this convergence is strong in $H^1(\Omega)$. There are three ingredients needed to see this upgrade of convergence: set $\varphi_j=u_j|_{\p\Omega}$, any subsequence of $\{(u_j,\varphi_j)\}$ satisfies the following up to a further subsequence
\begin{enumerate}
    \item $\varphi_j\rightarrow\hat\varphi=\hat u|_{\p\Omega}$ in $C^{k}(\p\Omega)$, which is because of the bound $\Vert \varphi_j\Vert_{C^{k,\alpha}(\p\Omega)}\leqslant\Lambda$ in (1) of Definition \ref{defn: definition of canonical neighborhood};
    \item The $C^{k}$ convergence of boundary data implies the smooth convergence of $u_j$ near a neighborhood of $\p\Omega$, which is due to the boundary regularity of F. Lin \cite[Section 4]{LinGradientEstimate};
    \item In the interior of $\Omega$, $u_j$ converges to $\hat u$ strongly in $H^1_{\loc}(\Omega)$, the reason of which is that (2) of Definition \ref{defn: definition of canonical neighborhood} gives the uniform bound of the energy of $u_j$ and our assumption \eqref{eq: strong convergence} gives the strong convergence. 
\end{enumerate}
The strong convergence of $\lbrace u_j\rbrace$ itself in $H^1(\Omega)$ then follows.

The canonical neighborhood enjoys the following compactness.

\begin{prop}\label{prop: compactness of canonical neighborhood}
    Given $(u,\varphi)\in\Hau^{k,\alpha}(\Omega,M)$ and $\Lambda\in\mathcal{R}^{k,\alpha}(u,\varphi)$, there exists $\delta_0=\delta_0(u,\varphi;\Lambda)\in(0,1)$ such that $\mathcal{L}^{k,\alpha}(u,\varphi;\Lambda,\delta_0)$ is compact under the $L^2$ norm in the first factor and $C^{k}$ norm in the second factor. 
\end{prop}
\begin{proof}
    The compactness of any sequence in $\mathcal{L}(u,\varphi;\Lambda,\delta)$ is direct from their energy bounds and $C^{k,\alpha}$ bounds. We only need to show that the limit $(\hat u,\hat\varphi)$ of any such sequence still belongs to $\mathcal{L}^{k,\alpha}(u,\varphi;\Lambda,\delta)$, provided $\delta$ is chosen small enough. The only missing ingredient is the correspondence of the singular set, which is (3) of Definition \ref{defn: definition of canonical neighborhood}.

    To get such a correspondence, let us write $\sing u=\lbrace x_i\rbrace_{i=1}^I$. Recall that Lemma \ref{lem: discreteness of density} shows that the possible densities
    $$\lbrace\Theta_{h}(0):h\in\mathcal{T}_n(\Lambda)\rbrace,$$
    is a finite set. Let $\eta_{\Lambda}$ be the minimum of the difference of different densities in this set. For each $i$, we can choose $r_i$ such that $\Theta_u(x_i,r_i)\leqslant\Theta_u(x_i)+\eta_{\Lambda}/4$. Take any $u'\in\mathcal{L}(u,\varphi;\Lambda,\delta)$ and label $\sing u'=\lbrace x_i'\rbrace_{i=1}^I$ where $x_i'$ is the point corresponding to $x_i$. We claim that if $\delta$ is chosen small enough, then $\Theta_{u'}(x_i',r_i)\leqslant\Theta_u(x_i)+\eta_{\Lambda}/3$. If this is not true, then for $\delta_j=1/j\rightarrow0$, there exists a sequence $(u_j,\varphi_j)\in\mathcal{L}(u,\varphi;\Lambda,\delta_j)$ such that 
    \begin{equation}\label{eq: lower bound of density in compactness lemma}
        \Theta_{u_j}(x_{i}^j,r_i)\geqslant\Theta_{u_j}(x_{i}^j)+\eta_{\Lambda}/3=\Theta_u(x_i)+\eta_{\Lambda}/3,
    \end{equation}
    where $\lbrace x_i^j\rbrace$ is the arrangement of $\sing u_j$ corresponding to the arrangement of $\sing u$. By the definition of $\mathcal{L}(u,\varphi;\Lambda,\delta_j)$, when $\delta_j\rightarrow0$, we must have $(u_j,\varphi_j)$ converges strongly to $(u,\varphi)$. The upper semicontinuity of density forces $x_i^j$ to converge to a singular point of $u$, and the distance assumption $\vert x_i^j-x_i\vert\leqslant\min_{i\ne i'}\vert x_i-x_{i'}\vert/3$ guarantees that the limit of $x_i^j$ can only be $x_i$. But \eqref{eq: lower bound of density in compactness lemma} together with the strong $H^1$ convergence leads to 
    $$\Theta_u(x_i,r_i)\geqslant\Theta_u(x_i)+\eta_{\Lambda}/3,$$
    contradicting to our choice of $r_i$.

    Hence the claim is proved, which gives us a small $\delta_0=\delta_0(u,\varphi;\Lambda)$, such that any $u'\in\mathcal{L}(u,\varphi;\Lambda,\delta_0)$ satisfies $\Theta_{u'}(x_i',r_i)\leqslant\Theta_u(x_i)+\eta_{\Lambda}/3$. Now, suppose we have a sequence $(u_j,\varphi_j)\in\mathcal{L}(u,\varphi;\Lambda,\delta_0)$ and this sequence converges to $(\hat u,\hat\varphi)$ with singular sets labeled as before $\sing u_j=\lbrace x_i^j\rbrace_{i=1}^I$. Firstly $\hat u$ must have exactly $I$ singular points $\lbrace\hat x_i\rbrace_{i=1}^I$ which are limits of $x_i^j$. There are no new singular points because of the bound of the regularity scales $r_{u_j}\geqslant\Lambda^{-1}\rho_{u_j}$ and no two singular points can collapse to a single singular points because of the distance bound $\vert x_i^j-x_i\vert\leqslant\min_{i\ne i'}\vert x_i-x_{i'}\vert/3$. Finally, our claim ensures that the limit $\hat x_i$ satisfies the density bound
    $$\Theta_{u}(x_i)\leqslant\Theta_{\hat u}(\hat x_i)\leqslant\Theta_{\hat{u}}(\hat x_i,r_i)\leqslant\Theta_{u}(x_i)+\eta_{\Lambda/3},$$
    where the first inequality is because of the upper semicontinuity of the density. Given our choice of $\eta_{\Lambda}$, the only density in the range $[\Theta_{u}(x_i),\Theta_u(x_i)+\eta_{\Lambda}/3]$ is $\Theta_u(x_i)$. Hence $\Theta_{\hat u}(\hat x_i)=\Theta_u(x_i)$ which is the desired correspondence.
\end{proof}

\subsection{Convergence lemmas}

The following non-concentration of Jacobi fields is an important feature for the convergence in canonical neighborhood.

\begin{lem}\label{lem: non-concentration of Jacobi fields}
    Let $(u_j,\varphi_j)\in\mathcal{L}(u,\varphi;\Lambda,\delta_0)$ such that $(u_j,\varphi_j)\rightarrow (\hat u,\hat \varphi)$, where $\delta_0$ is given by Proposition \ref{prop: compactness of canonical neighborhood}. Given a sequence $X_j$ of tame Jacobi fields of $u_j$ with $\Vert X_j\Vert_{L^2(\Omega)}=1$ with $\Vert X_j|_{\p\Omega}\Vert_{C^{2,\alpha}(\p\Omega)}$ uniformly bounded, after passing to a subsequence, $X_j$ converges strongly in $H^1_{\text{loc}}(\reg \hat u)$ to a tame Jacobi field $\hat X$ and $X_j|_{\p\Omega}\rightarrow X|_{\p\Omega}$ in $C^2(\p\Omega)$. Moreover, for each $\epsilon>0$, there exists $r=r(\epsilon)>0,j_0=j_0(\epsilon)>0$, such that
    $$\sum_{x_0\in\sing u_j}\int_{B_r(x_0)}\vert X_j(x)\vert^2\dif x\leqslant\epsilon,\text{ for all }j\geqslant j_0.$$
    In particular, $\Vert \hat X\Vert_{L^2(\Omega)}=1$.
\end{lem}
\begin{proof}
    The strong convergence of $X_j$ in $H^1_{\loc}(\reg u)$ follows from the local smooth convergence of $X_j$ in $\reg u$ and this local smooth convergence is because of the classical elliptic estimates.

    In order to prove the $L^2$ strong convergence of $X_j$, let us fix a $\sigma\in(0,1)$ such that
    $$\min_{x_0\in\sing u}\dist (-\sigma,\Gamma(h_{x_0}))=\sigma,\text{ where }h_{x_0}\text{ is the tangent map of }u\text{ at }x_0.$$
    Label $\sing u_j$ as $\lbrace x_i^j\rbrace$. By Corollary \ref{Cor: uniform graph on tangent maps}, the condition of Lemma \ref{lem: tame field upper bound} is satisfied on $B_{r_0}(x_i^j)$ for some small $r_0>0$ and large $j$, which gives a uniform $L^2_{-\sigma}$ estimate of the form
    $$\vert v_j\vert+\sup_{l\geqslant0}\Vert X_j-\p_{v_j}u_j\Vert_{L^2_{-\sigma}(2^{-l-1}r_0,2^{-l}r_0(x_i^j))}\leqslant C(\Lambda,\sigma)\Vert X_j\Vert_{L^2(\Omega)}=C(\Lambda,\sigma).$$
    In particular, if $r<r_0$,
    $$\Vert X_j\Vert_{L^2(B_r(x_i^j))}\leqslant \Vert X_j-\p_{v_j}u_j\Vert_{L^2(B_r(x_i^j))}+\Vert\p_{v_j}u_j\Vert_{L^2(B_r(x_i^j))}\leqslant C(\Lambda,\sigma)r^{n/2-1},$$
    where the $\Vert\p_{v_j}u_j\Vert_{L^2(B_r(x_i^j))}$ is controlled by monotonicity formula. This bound together with the strong convergence away from $\sing u$ gives $\Vert \hat X\Vert_{L^2(\Omega)}=1$. That $\hat X$ is a tame Jacobi field follows from Lemma \ref{lem: asymtpotic rate and growth rate relation} (3).
    
\end{proof}

The normalized convergence of pairs of \textbf{HSI} induces a tame Jacobi field of the limit \textbf{HSI}.
\begin{lem}\label{lem: induced Jacobi fields for pairs}
    Let $\delta_0$ be the constant given by Proposition \ref{prop: compactness of canonical neighborhood}. Given $(u,\varphi)\in\Hau^{k,\alpha}(\Omega,M),\Lambda\in\mathcal{R}^{k,\alpha}(u,\varphi)$ and $\delta>0$, suppose that:
    \begin{enumerate}
        \item $\Psi$ is a finite dimensional subspace of $C^{k,\alpha}(\p\Omega,T_{\varphi}M)$
        \item $\lbrace(\bar u_j,\bar\varphi_j)\rbrace_{j\in\N},\lbrace( u_j^1,\varphi_j^1)\rbrace_{j\in\N},\lbrace (u_j^2,\varphi_j^2)\rbrace_{j\in\N}$ be three sequences in $\mathcal{L}^{k,\alpha}(u,\varphi;\Lambda,\delta_0)$ converge to $(u,\varphi)$ with $\lbrace (u_j^1,\varphi_j^1)\rbrace_{j\in\N},\lbrace (u_j^2,\varphi_j^2)\rbrace_{j\in\N}$ distinct;
        \item For all $j\geqslant1$ and $i=1,2$, either $\varphi^1_j=\varphi_j^2$ for all large $j$, in which case we set $\psi=0$, or $\varphi^i_j=\exp_{\bar\varphi_j}\pi_{\bar \varphi_j}\psi_j^i$ for some $\psi_j^i\in\Psi$ and
        $$\frac{\varphi^2_j-\varphi_j^1}{\Vert\varphi_j^2-\varphi_j^1\Vert_{C^{k,\alpha}}}\rightarrow\psi\ne 0\text{ in }C^{k,\alpha}(\p\Omega,T_{\varphi}M).$$
    \end{enumerate}
    Set $d_j=(\Vert u_j^2-u_j^1\Vert_{L^2(\Omega)}+\Vert\varphi_j^2-\varphi_j^1\Vert_{C^{k,\alpha}(\p\Omega)})$. Then, after passing to a subsequence, we have 
    $$d_j^{-1}(u_j^2-u_j^1)\rightarrow \hat X\text{ in } L^2(\Omega)\text{ and }d_j^{-1}(\varphi_j^2-\varphi_j^1)\rightarrow\hat\psi\text{ in } C^{k,\alpha}(\p\Omega),$$
    where $\hat\psi$ is a non-negative multiple of $\psi$, $\hat X\in C^{\infty}(\reg u,T_uM)$ is a non-zero tame section, $\hat X|_{\p\Omega}=\hat\psi$ and is a stationary Jacobi field per Definition \ref{defn: stationary jaocbi fields}.
\end{lem}
\begin{proof}
    The local boundedness of $(u_j^2-u_j^1)/d_j$ in $H^1_{\loc}(\reg u)$ is a direct consequence of classical elliptic PDE estimates and the uniform smoothness of $u,u_j^i,i=1,2$ on $\reg u$ for $j$ large. Passing to a subsequence, we can assume it converges to $\hat X$ weakly in $H^1_{loc}(\reg u)$ and strongly in $L^2_{\loc}(\reg u)$. Using the PDE again, this convergence is indeed strong in $H^1_{\loc}(\reg u)$. On the other hand, the fact that $d_j\geqslant \Vert\varphi_j^2-\varphi_j^1\Vert_{C^{k,\alpha}(\p\Omega)}$ allows us to assume $\Vert \varphi_j^2-\varphi_j^1\Vert_{C^{k,\alpha}(\p\Omega)}/d_j$ converges to $a\in[0,1]$. Then $\hat \psi=a\psi$.
    
    Let us subtract the harmonic map equation of $u_j^1$ and $u_j^2$: in the distributional sense, we have
    $$\Delta(u_j^2-u_j^1)+A_{u_j^2}(\nabla (u_j^1+u_j^2),\nabla (u_j^2-u_j^1))+(A_{u_j^2}-A_{u_{j}^1})(\nabla u_j^1,\nabla u_j^1)=0.$$
    Divide this equation by $d_j$ and take the limit, using the fact that $u_j^1$ and $u_j^2$ both converge strongly to $u$, we see that $\hat X$ is a Jacobi field of $u$. Moreover, we subtract the stationary equation of $u_j^1$, $u_j^2$ to derive: for any $\xi\in C^{\infty}_c(\Omega,\R^n)$ that is constant near each singularity of $u$,
    $$\int_{\Omega}\langle\nabla (u_j^1+u_j^2),\nabla (u_j^2-u_j^1)\rangle\operatorname{div}\xi\dif x=2\sum_{\alpha,\beta=1}^n\int_{\Omega}(\langle\p_{\alpha}(u_j^2-u_j^1),\p_{\beta} u_j^2\rangle+\langle\p_{\alpha}u_j^1,\p_{\beta }(u_j^2-u_j^1)\rangle)\p_{\alpha}\xi^{\beta}\dif x.$$
    Dividing both sides by $d_j$ and taking the limit, by the strong convergence of $u_j^1,\tilde u_j^2$ to $u$ in $H^1(\Omega)$ and of $(u_j^2-u_j^1)/d_j$ to $\hat X$ in $H^1($spt$\nabla\xi)$, we see that $\hat X$ satisfies the stationarity condition in Definition \ref{defn: stationary jaocbi fields}.

    Now, let us show that $X$ is a tame section. Let us fix a $\sigma=\sigma(u,\varphi)>0$ such that
    $$\dist(-2\sigma,\Gamma(h_{x_i})\cup\{-(n-2)/2\})\geqslant2\sigma,\text{ for all }x_i\in\sing u$$
    By the PDE of $\hat X$, it is clear that $\hat X$ is smooth away from the singular points of $u$. Hence, let us fix a singular point $x_0$ of $u$ and look at a sufficiently small ball $B_{r_0}(x_0)$ such that the distance of this ball to other singular points of $u$ and $\p\Omega$ is at least $10 r_0$. Without loss of  generality, assume $x_0=0$.  By our assumption of pseudo-neighborhood of $u$, there exists $\sing u_j^i\cap B_{r_0}(x_0)=\lbrace x_j^i\rbrace$ for a unique $x_j^i\rightarrow 0$. Let us write

    $$u_j^2-u_j^1=(u_j^2-u_j^2(\cdot+x_j^2-x_j^1))+(u_j^2(\cdot+x_j^2-x_j^1)-u_j^1).$$

    Let us show that the first part of difference converges to a translation-like section while the second part satisfies the growth estimate. Indeed, Lemma \ref{lem: HSI with different singularity} shows for $j$ large
    $$\vert x_j^2-x_j^1\vert\leqslant C(\sigma,\Lambda)\Vert u_j^1-u_j^2\Vert_{L^2(\Omega)}\leqslant Cd_j.$$
    Hence subsequentially $(x_j^2-x_j^1)/d_j\rightarrow v\in\R^n$. Then 
    $$(u_j^2(\cdot+x_j^2-x_j^1)-u_j^2)/d_j=\int_0^1\nabla_{(x_j^2-x_j^1)/d_j}u_j^2(\cdot+t(x_j^2-x_j^1))\dif t\rightarrow\nabla_vu,\text{ as }j\rightarrow\infty.$$
    Here we used the fact that $\nabla u_j^2$ converges to $\nabla u$ strongly in $L^2(\Omega)$. On the other hand, since $u_j^2(\cdot+x_j^2-x_j^1)$ and $u_j^1$ have the same singular point within $B_{r_0}$, we can apply Lemma \ref{Lem: hsi with same singularity} to get 
    $$\Vert u_j^2(\cdot+x_j^2-x_j^1)-u_j^1\Vert_{C^{0;-\sigma}(B_{r_0/2})}\leqslant C(\sigma,\Lambda)\Vert u_j^1-u_j^2\Vert_{L^2(\Omega)}\leqslant Cd_j.$$
    In particular,
    $$\int_{A_{2^{-l-1},2^{-l}}}\vert u_j^2(\cdot+x_j^2-x_j^1)-u_j^1\vert^2 \vert x\vert^{-n+4\sigma}\dif x\leqslant C2^{-2l\sigma}d_j^2.$$
    Passing to a subsequence, since $(u_j^2(\cdot+x_j^2-x_j^1)-u_j^1)/d_j\rightarrow \hat X+\p_vu=: X$ strongly in $L^2_{\loc}(B_{r_0}\setminus\{0\})$ as $j\rightarrow\infty$, we have
    $$\Vert X\Vert_{L^{2;-2\sigma}(B_{r_0})}^2\leqslant\sum_{l=0}^{\infty}C2^{-2\sigma l}<\infty.$$
    Hence we successfully decomposed $\hat X$ into a translation-like field $\p_v u$ and a field $X$ which has growth rate control near $0$. This shows that $\hat X$ is a tame field. 

    Finally, in order to show that $\hat X$ is non-zero, we note that the argument above shows that $(u_j^2-u_j^1)/d_j$ converges strongly in $L^2(\Omega)$ to $\hat X$; see also the analogous non-concentration property needed for this strong convergence established in Lemma \ref{lem: non-concentration of Jacobi fields}. Hence, the normalization by $d_j=\Vert u_j^1-u_j^2\Vert_{L^2(\Omega)}+\Vert\varphi^1_j-\varphi_j^2\Vert_{C^{k,\alpha}(\p\Omega)}$ gives
    $$\Vert \hat X\Vert_{L^2(\Omega)}+\Vert\hat\psi\Vert_{C^{k,\alpha}(\p\Omega)}=1,$$
    which gives the nontriviality of $\hat X$.
\end{proof}

\subsection{The local Sard-Smale Theorem}

We are now ready to prove the following local Sard-Smale theorem for the canonical neighborhood.
\begin{thm}\label{thm: local sard smale theorem}
    Let $(u,\varphi)\in\Hau^{k,\alpha}(\Omega,M),\Lambda\in\mathcal{R}^{k,\alpha}(u,\varphi)$ with $\widehat{\ind}L_u<0,\tau$ satisfies \eqref{eq: range of rate}. Then there exists $\delta_2=\delta_2(u,\varphi;\Lambda)>0$ such that for every $0<\delta\leqslant\delta_2$,
    $$\mathcal{B}^{k,\alpha}(u,\varphi;\Lambda,\delta)=C^{k,\alpha
    }(\p\Omega,M)\setminus\Pi(\mathcal{L}^{k,\alpha}(u,\varphi;\Lambda,\delta)),$$
    is an open and dense subset of $C^{k,\alpha
    }(\p\Omega,M)$ under $C^{k,\alpha}$ norm.
\end{thm}

\begin{lem}\label{lem: bound of dimension of the kernel}
    Let $u,\varphi,\tau$ be as in Theorem \ref{thm: local sard smale theorem}. There exists $\delta_1=\delta_1(u,\varphi;\Lambda)>0$ such that for each $(u',\varphi')\in\mathcal{L}^{k,\alpha}(u,\varphi;\Lambda,\delta_1)$, there holds
    $$\dim\widehat{\ker}_{\tau}L_{u'}\leqslant\dim\widehat{\ker}_{\tau}L_u.$$
\end{lem}
\begin{proof}
    Let us argue by contradiction. Suppose the lemma is not true, then we would get a sequence $(u_j,\varphi_j)\in\mathcal{L}^{k,\alpha}(u,\varphi;\Lambda,1/j)$ such that $\dim\widehat{\ker}_{\tau}L_{u_j}>\dim\widehat \ker_{\tau}L_u$. We can then take an $L^2$-orthonormal family $\{X_j^i\}_{i=1}^{d+1}\subset \widehat \ker_{\tau}L_{u_j}$, where $d=\dim\widehat \ker_{\tau}L_u$. Passing to a subsequence $X_j^i\rightarrow X^i$ as $j\rightarrow\infty$, the weak $H^1(\reg u)$ convergence guarantees that $X^i$ is still in $\widehat \ker_{\tau}L_u$ for each $i$. Lemma \ref{lem: non-concentration of Jacobi fields} also guarantees that the convergence of $X_j^i$ to $X^i$ is strong in $L^2(\Omega)$ for each $i$, which makes $\lbrace X^i\rbrace_{i=1}^{d+1}$ still an $L^2$-orthonormal family. As a result, $ \dim\widehat \ker_{\tau}L_u>d$, which contradicts to the assumption that $d=\dim\widehat \ker_{\tau}L_u$.
\end{proof}
Let then $$\mathcal{L}^{k,\alpha}_{\text{top}}(u,\varphi;\Lambda,\delta)=\lbrace (u',\varphi')\in \mathcal{L}^{k,\alpha}(u,\varphi;\Lambda,\delta):\dim\widehat{\ker}_{\tau}L_{u'}=\dim\widehat{\ker}_{\tau}L_u\rbrace$$ 
be \textbf{HSI} with kernel of highest possible dimension. By Lemma \ref{lem: bound of dimension of the kernel}, this is a closed subset of $\mathcal{L}^{k,\alpha}(u,\varphi;\Lambda,\delta)$. We further set
$$I=\dim\widehat{\ker}_{\tau}L_u,J=\dim\widehat{\operatorname{coker}}_{\tau}L_u.$$
For $(u',\varphi')\in\mathcal{L}^{k,\alpha}_{\text{top}}(u,\varphi;\Lambda,\delta_1)$, let 
$$\pi^{L^2}_{u'}:L^2(\Omega,\R^N)\rightarrow\widehat{\ker}_{\tau}L_{u'},$$
be the orthogonal projection and 
$$\pi_{\varphi}X(x)=\pi_{T_{\varphi(x)}M}(X(x)),$$ 
be the projection to the tangent space at $\varphi(x)$ for any vector field $X:\Omega\rightarrow\R^N$.

By Lemma \ref{lem: equiv of dimension of coker}, we can choose a $J=\dim\widehat{\text{coker}}_{\tau}L_u$-dimensional subspace $\Psi$ such that if $0\ne\psi\in\Psi$, the equation
    \begin{align*}
        \begin{cases}
            L_uX=0,&\text{ in }\Omega;\\
            X=\psi,&\text{ on }\p\Omega;
        \end{cases}
    \end{align*}
has no tame field solution $X$ with $\pi_u^{L^2}X=0$.
\begin{lem}\label{lem: Lipshitz of projection}
    There exists $\delta_2=\delta_2(u,\varphi;\Lambda)\in(0,\delta_1(u,\varphi;\Lambda)),$ where $\delta_1$ is given by Lemma \ref{lem: bound of dimension of the kernel}, $r_0(u,\Lambda)>0$, such that for each $(\bar u,\bar \varphi)\in\mathcal{L}^{k,\alpha}_{\text{top}}(u,\varphi;\Lambda,\delta_2)$, the map
    \begin{align*}
        P_{\bar u,\bar \varphi}:\mathcal{L}^{k,\alpha}_{\text{top}}(u,\varphi;\Lambda,\delta_2)\cap\Pi^{-1}(\exp_{\bar\varphi}\Psi)&\rightarrow\widehat{\ker}_{\tau}L_{\bar u},\\
        (u',\varphi')&\mapsto\pi_{\bar u}^{L^2}((u'-\bar u)\zeta_{\bar u,r_0}),
    \end{align*}
    is bi-Lipschitz onto its image, with the bi-Lipschitz constant bounded by a constant depending on $u,\varphi,\Lambda$. Here $\exp_{\bar\varphi}\Psi=\lbrace\exp_{\bar\varphi}\pi_{\bar\varphi}\psi:\psi\in\Psi,\Vert\psi\Vert_{C^0(\p\Omega)}\leqslant\text{inj}(M)\rbrace$ is a slice of $C^{k,\alpha}(\p\Omega,M)$ of finite dimension $J$.
\end{lem}
In the lemma, $\zeta\in C^{\infty}(\R,[0,1])$ is a function such that $\zeta=0$ on $(-\infty,1]$ and $\zeta=1$ on $[2,\infty)$. We then set $\zeta_{u,r_0}=\zeta(\rho_{u}/r_0)$.
\begin{proof}
    Let $\psi_1,\dots\psi_J$ be a basis of $\Psi$ and $$\tau_0=\frac{1}{2}\sup_{x\in\sing u}\gamma^-(h_x),$$
    Since $\widehat\ker_{\tau_0}$ is a finite dimensional set, we can choose $r_0=r_0(u,\Lambda)$ sufficiently small such that the linear map
    $$\widehat\ker_{\tau_0}\rightarrow\widehat\ker_{\tau_0},X\mapsto\pi^{L^2}_u(\zeta_{u,r_0}X),$$
    is an isomorphism. 
    
    We claim that for small enough $\delta_2$, and each $(\bar u,\bar \varphi)\in\mathcal{L}^{k,\alpha}_{\text{top}}(u,\varphi;\Lambda,\delta_2)$, the corresponding subspace $\pi_{\bar \varphi}\Psi\subset C^{k,\alpha}(\p\Omega,T_{\bar\varphi}M)$ also does not support Jacobi fields. Indeed, suppose this is not true, then we can find a sequence $(u_j,\varphi_j)\rightarrow(u,\varphi)$ in $L^2(\Omega)$ and $C^{k}(\p\Omega)$ respectively, and a sequence $X_j,\psi_j$ such that
    \begin{align*}
        L_{u_j}X_j=0\textup{ in }\Omega,X_j=\pi_{\varphi_j}\psi_j\text{ on }\p\Omega;\qquad\pi_{u_j}^{L^2}X_j=0,
        \psi_j\in\Psi,\Vert X_j\Vert_{L^2}+\Vert\psi_j\Vert_{C^{k,\alpha}}=1.
    \end{align*}
    Passing to a subsequence, by Lemma \ref{lem: non-concentration of Jacobi fields} we can assume $X_j\rightarrow X$ strongly in $L^2(\Omega)$ and weakly in $H^1(\reg u)$, and $\psi_j$ all belonging to a finite dimensional space $\Psi$ guarantees that $\psi_j\rightarrow\psi\in\Psi$ in $C^{k,\alpha}(\p\Omega)$. Hence, $X$ is a tame Jacobi field with boundary data $\psi$. Moreover $\pi_u^{L^2}X=0$ and $\Vert X\Vert_{L^2(\Omega)}+\Vert\psi\Vert_{C^{k,\alpha}(\p\Omega)}=1$. $\psi=0$ makes $X\in\widehat{\ker}_{\tau_0}L_u$, contradicting to the assumption that $\pi_u^{L^2}X=0$. But $\psi\ne0$ also contradicts to the definition of $\Psi$. Hence the claim is proved.

    We proceed with the proof of Lemma \ref{lem: Lipshitz of projection}. Suppose Lemma \ref{lem: Lipshitz of projection} is not true, we get three sequences $\lbrace u_j^i\rbrace_{j\in\N},i=0,1,2$, such that $u^i_j$ restrict to $\p\Omega$ equals $\varphi_j^i\in\exp_{\bar\varphi}\Psi$ and $(u_j^i,\varphi_j^i)\rightarrow(u,\varphi)$ but if we set $\zeta_j=\zeta_{u_j^0,r_0}$ and $d_j=\Vert u^1_j-u_j^2\Vert_{L^2(\Omega)}+\Vert\varphi_j^1-\varphi_j^2\Vert_{C^{k,\alpha}(\p\Omega)}$ we have
    \begin{equation}
        \begin{aligned}
            \text{either }\Vert\pi_{u_j^0}^{L^2}((u_j^2-u_j^1)\zeta_j)\Vert_{L^2}\leqslant d_j/j;\qquad
        \text{or }\Vert\pi_{u_j^0}^{L^2}((u_j^2-u_j^1)\zeta_j)\Vert_{L^2}\geqslant jd_j.
        \end{aligned}\label{eq: two possiblities of Lipshitz blow up}
    \end{equation}
    After passing to a subsequence, by Lemma \ref{lem: induced Jacobi fields for pairs}, $(u^2_j-u_j^1)/d_j$ converges to a tame Jacobi field $X$, whose boundary $\psi$ is the limit of $(\varphi_j^2-\varphi_j^1)/d_j$. As such, $\psi\in\Psi$. By the definition of $\Psi$, we must have $\psi=0$. This tells us that $X$ belongs to $\widehat{\ker}_{\tau_0}L_u$. The second possibility in \eqref{eq: two possiblities of Lipshitz blow up} clearly cannot hold for large $j$ by this convergence. The first possibility in \eqref{eq: two possiblities of Lipshitz blow up}, however, tells us $\pi_u^{L^2}(X\zeta_{u,r_0})=0$. By the choice of $r_0$, this contradicts to the fact that $X\ne0$.
\end{proof}
\begin{proof}[Proof of Theorem \ref{thm: local sard smale theorem}]
    Let us fix a $\delta\in(0,\delta_2]$ and a $\tau$ satisfying \eqref{eq: range of rate}.
    
    The openness of the set $\mathcal{B}^{k,\alpha}(u,\varphi;\Lambda,\delta)$ is an immediate corollary of compactness of Proposition \ref{prop: compactness of canonical neighborhood}. We only need the denseness. We will prove the theorem inductively on $I=\dim\widehat{\ker}_{\tau}L_u\geqslant0$. We henceforth assume the Theorem is already proved for maps with $\dim\widehat{\ker}_{\tau}L_u=I-1$. Note that when $I=0$, there is indeed no assumption made.
\begin{claim}
    For $\bar{\varphi}\in C^{k,\alpha}(\p\Omega,M)\setminus\mathcal{B}^{k,\alpha}(u,\varphi;\Lambda,\delta)=\Pi(\mathcal{L}^{k,\alpha}(u,\varphi;\Lambda,\delta))$, there exists a sequence of maps $\varphi_j\in \exp_{\bar\varphi}\Psi\setminus\Pi(\mathcal{L}^{k,\alpha}_{\text{top}}(u,\varphi;\Lambda,\delta))$ such that $\varphi_j\rightarrow\bar{\varphi}$ in $C^{k,\alpha}(\p\Omega)$ as $j\rightarrow\infty$.
\end{claim} 
We first finish the proof of Theorem \ref{thm: local sard smale theorem} assuming the claim above. We show each $\varphi_j$ can be approximated by maps in $\mathcal{B}^{k,\alpha}(u,\varphi;\Lambda,\delta)$. When $I=0$, this is direct from the claim as $ \mathcal{L}^{k,\alpha}_{\text{top}}(u,\varphi;\Lambda,\delta)=\mathcal{L}^{k,\alpha}(u,\varphi;\Lambda,\delta)$ in this case.

We then consider the case $I\geqslant1$. Fix a $j\geqslant1$ and assume $\varphi_j$ is not in $\mathcal{B}^{k,\alpha}(u,\varphi;\Lambda,\delta),$ which implies that $\Pi^{-1}(\varphi_j)\cap\mathcal{L}^{k,\alpha}(u,\varphi;\Lambda,\delta)\ne\emptyset$. Moreover, by Lemma \ref{lem: bound of dimension of the kernel}, for each $(u',\varphi_j)\in\mathcal{L}^{k,\alpha}(u,\varphi;\Lambda,\delta)$, we have that $\dim\widehat{\ker}_{\tau}L_{u'}\leqslant I-1$. Hence, by the induction assumption, there exists $\delta_{u',\varphi_j}>0$ which may also depend on $\Lambda$, such that $\mathcal{B}^{k,\alpha}(u',\varphi_j;\Lambda,\delta_{u',\varphi_j})$ is open and dense in $C^{k,\alpha}(\p\Omega,M)$. Since $\mathcal{L}^{k,\alpha}(u',\varphi_j;\Lambda,\delta_{u',\varphi_j})$ contains an (relative) open neighborhood of $(u',\varphi_j)$ in $\mathcal{L}^{k,\alpha}(u,\varphi;\Lambda,\delta)$, and $\mathcal{L}^{k,\alpha}(u,\varphi;\Lambda,\delta)\cap\Pi^{-1}(\varphi_j)$ is compact by Proposition \ref{prop: compactness of canonical neighborhood}, we see there exists $(u^i,\varphi_j)\in\Pi^{-1}(\varphi_j)\cap\mathcal{L}^{k,\alpha}(u,\varphi;\Lambda,\delta)$ such that the union of $\mathcal{L}^{k,\alpha}(u^{i},\varphi_j;\Lambda,\delta_{u^{i},\varphi_j}),1\leqslant i\leqslant i_0(j)$ covers $\Pi^{-1}(\varphi_j)\cap\mathcal{L}^{k,\alpha}(u,\varphi;\Lambda,\delta)$. 

The fact that $\mathcal{B}_j=\cap_{i=1}^{i_0}\mathcal{B}^{k,\alpha}(u^{i},\varphi_j;\Lambda,\delta_{u^{i},\varphi_j})$ being open and dense allows us to select a $\varphi_l'$ in $\mathcal{B}_j$ which converges to $\varphi_j$ in $C^{k,\alpha}(\p\Omega)$ as $l\rightarrow\infty$. Moreover, by compactness we can guarantee that if $l$ is large $$\mathcal{L}^{k,\alpha}(u,\varphi;\Lambda,\delta)\cap\Pi^{-1}(\varphi_l')\subset\bigcup_{i=1}^{i_0}\mathcal{L}^{k,\alpha}(u^{i},\varphi_j;\Lambda,\delta_{u^{i},\varphi_j}).$$
But since $\varphi_l'\in\mathcal{B}_j$, this is only possible if LHS is empty. Hence we indeed have $\varphi_l'\in\mathcal{B}^{k,\alpha}(u,\varphi;\Lambda,\delta)$ for large $l$, which concludes our proof.

Finally let us prove our claim.
\begin{proof}[Proof of Claim]
    Suppose without loss of generality that $(\bar u,\bar\varphi)\in \mathcal{L}^{k,\alpha}_{\text{top}}(u,\varphi;\Lambda,\delta)$, otherwise we can take $\varphi_j=\bar\varphi$. By Lemma \ref{lem: Lipshitz of projection}, $\mathcal{L}^{k,\alpha}_{\text{top}}(u,\varphi;\Lambda,\delta)\cap\Pi^{-1}(\exp_{\bar\varphi}\Psi)$ is bi-Lipschitz embedded in $\widehat \ker_{\tau}L_u$ and is compact. Denote by $\mathcal{Z}=P_{\bar u,\bar\varphi}(\mathcal{L}^{k,\alpha}_{\text{top}}(u,\varphi;\Lambda,\delta)\cap\Pi^{-1}(\exp_{\bar\varphi}\Psi))$ be its image in $\widehat\ker_{\tau}L_{\bar u}$. Consider the composite map
    $$\Pi:\mathcal{Z}\simeq\mathcal{L}^{k,\alpha}_{\text{top}}(u,\varphi;\Lambda,\delta)\cap\Pi^{-1}(\exp_{\bar\varphi}\Psi)\xrightarrow{\Pi}\exp_{\bar\varphi}\Psi,$$
    which is a Lipschitz map between compact sets.
    Since $\widehat{\ind}_{\tau}L_u<0$, which is equivalent to $I<J$, which implies $\Pi(\mathcal{Z})$ has dense complement in $\exp_{\bar\varphi}\Psi$ as the latter has dimension $J$. Hence there exists $\varphi_j\in \exp_{\bar\varphi}\Psi\setminus\Pi(\mathcal{L}^{k,\alpha}_{\text{top}}(u,\varphi;\Lambda,\delta))$ which converges to $\bar\varphi$. This proves the claim.
\end{proof}

\end{proof}

Although canonical neighborhoods are not genuine open sets in general, we are still able to use them to write down an "atlas" for the space of \textbf{HSI}. The following theorem is another ingredient needed to prove our generic regularity result.
\begin{thm}\label{thm: covering of space of HSI}
    Let $\kappa:\mathcal{T}^{k,\alpha}\rightarrow\R_+$ be a function that is not necessarily continuous. Then there exists $\lbrace(u_j,\varphi_j;\Lambda_j)\rbrace_{j\in\N}\subset\mathcal{T}^{k,\alpha}$ such that
    $$\Hau^{k,\alpha}(\Omega,M)=\bigcup_{j=1}^{\infty}\mathcal{L}^{k,\alpha}(u_j,\varphi_j;\Lambda_j,\kappa_j),\text{ where }\kappa_j=\kappa(u_j,\varphi_j;\Lambda_j).$$
\end{thm}

The proof of Theorem \ref{thm: covering of space of HSI} follows almost identically the proofs in \cite{LiWangGenericRegularity} and \cite{CLWNon-persistence}, which depend mainly on the compactness and monotonicity formula. However, the proof requires the introduction of combinatorial terminology, which does not fit the main theme of this paper and thus is redundant here. So, we refer readers to \cite{LiWangGenericRegularity} and \cite{CLWNon-persistence} for the proof of Theorem \ref{thm: covering of space of HSI} and do not include it here.

\subsection{Proof of Theorem \ref{thm: genericity of Jacobi operator}}
Let us prove Theorem \ref{thm: genericity of Jacobi operator} using Theorem \ref{thm: count of index}, Theorem \ref{thm: local sard smale theorem} and Theorem \ref{thm: covering of space of HSI}. The proof will be first carried out for $C^{k,\alpha}$ boundary data, then propagate to $C^{\infty}$ boundary data using the following general fact about metric spaces.

\begin{lem}[{\cite[Lemma 4.30]{CLWNon-persistence}}]\label{lem: propagate generic to countable intersection}
    Let $\lbrace(S_j,d_j)\rbrace_{j=1}^{\infty}$ be a sequence of complete metric spaces, with associated inclusions $\iota_j:S_{j+1}\rightarrow S_{j}$ such that 
    $$d_{j}(\iota_j(x),\iota_j(y))\leqslant d_{j+1}(x,y),\text{ for all }x,y\in S_{j+1}.$$
    Identify all $S_j$ as a subset of $S_1$ and define
    $$S_{\infty}=\bigcap_{j=1}^{\infty}S_j,\text{ with metric }d_{\infty}(x,y)=\sum_{j=1}^{\infty}2^{-j}\frac{d_j(x,y)}{1+d_j(x,y)},\text{ for all }x,y\in S_{\infty}.$$
    Suppose $S_{\infty}$ is dense in $(S_j,d_j)$ for all $j\geqslant1$. In addition, assume $A\subset (S_1,d_1)$ is an open (resp. $G_{\delta}$) subset such that for all $j$, $A\cap S_j$ is dense in $(S_j,d_j)$. Then $A\cap S_{\infty}$ is an open (resp. $G_{\delta}$) dense subset in $S_{\infty}$.
\end{lem}
See \cite{CLWNon-persistence} for a proof of this lemma.

\begin{proof}[Proof of Theorem \ref{thm: genericity of Jacobi operator}]
    We only prove the existence of such a generic subset in $C^{k,\alpha}(\p\Omega,M)$. The proof for smooth boundary maps follows by applying Lemma \ref{lem: propagate generic to countable intersection}. Let us focus on $C^{k,\alpha}(\p\Omega,M)$.

    Consider a map $(u,\varphi)\in\mathcal{H}^{k,\alpha}(\Omega,M)$ with $\widehat{\ind}_{\tau}L_u<0$. Theorem \ref{thm: count of index} implies there exists an $x_0\in\sing u$, the tangent map $h_{x_0}$ of $u$ at $x_0$ satisfies $I(h_{x_0})>0$. If another tangent map $h\in\mathcal{T}_n(M)$ satisfies $\Vert h_{x_0}|_{\p B_1}-h|_{\p B_1}\Vert_{C^2(\Sp^{n-1})}<d(h_{x_0})$ for a constant $d(h_{x_0})$, we also have $I(h)>0$. Hence, given $\Lambda\in\mathcal{R}^{k,\alpha}(u,\varphi)$ and recall compactness Proposition \ref{prop: compactness of canonical neighborhood}, we can select $\kappa(u,\varphi;\Lambda)\in(0,\delta_2]$ where $\delta_2=\delta_2(u,\varphi;\Lambda)$ is the constant given by Theorem \ref{thm: local sard smale theorem} with the following property. \\
    \emph{For every $(u',\varphi')\in \mathcal{L}^{k,\alpha}(u,\varphi;\Lambda,\kappa(u,\varphi,\Lambda))$ there exists $x_0'\in \sing u'$, such that the tangent map $h_{x_0'}$ of $u'$ at $x_0'$ satisfies
    $$\Vert h_{x_0}|_{\p B_1}-h_{x_0'}|_{\p B_1}\Vert_{C^2(\Sp^{n-1})}<d(h_{x_0})\implies I(h_{x_0'})>0.$$
    In particular, $\widehat{\ind}_{\tau}L_{u'}<0$.
    }
    When $\widehat{\ind}_{\tau}L_u$ already equals 0, we simply set $\kappa(u,\varphi;\Lambda)=\delta_0(u,\varphi;\Lambda)$. By Theorem \ref{thm: covering of space of HSI}, viewing $\kappa$ as a function from $\mathcal{T}^{k,\alpha}$ to $\R_+$, we can correspondingly select $(u_j,\varphi_j;\Lambda_j)$ and write 
    $$\mathcal{H}^{k,\alpha}(\Omega,M)=\bigcup_{j=1}^{\infty}\mathcal{L}^{k,\alpha}(u_j,\varphi_j;\Lambda_j,\kappa_j),\text{ where }\kappa_j=\kappa(u_j,\varphi_j;\Lambda_j).$$
    
    For every $j\geqslant1$ we denote by $\mathcal{L}_-^{k,\alpha}(u_j,\varphi_j;\Lambda_j,\kappa_j)$ pairs $(u',\varphi')\in\mathcal{L}^{k,\alpha}(u_j,\varphi_j;\Lambda_j,\kappa_j)$ with $\widehat{\ind}_{\tau}L_{u'}<0$. We now construct for each $j\geqslant1$ a covering of $\mathcal{L}_-^{k,\alpha}(u_j,\varphi_j;\Lambda_j,\kappa_j)$ by countably many canonical neighborhood of the form $\mathcal{L}^{k,\alpha}(u,\varphi;\Lambda,\kappa)$ where $\widehat{\ind}_{\tau}L_u<0$. By Proposition \ref{prop: compactness of canonical neighborhood} again, it suffices to show that for every $(u,\varphi)\in\mathcal{L}^{k,\alpha}(u_j,\varphi_j;\Lambda_j,\kappa_j)$, the set $\mathcal{L}^{k,\alpha}(u,\varphi;\Lambda_j,\kappa(u,\varphi;\Lambda_j))$ contains an open neighborhood of $(u,\varphi)$ in $\mathcal{L}^{k,\alpha}(u_j,\varphi_j;\Lambda_j,\kappa_j)$. This can be easily shown to be true using definition of canonical neighborhood.

    Therefore, by relabeling the union of previous covering of $\mathcal{L}_-^{k,\alpha}(u_j,\varphi_j;\Lambda_j,\kappa_j)$ for all $j$, we see 
    $$\lbrace(u,\varphi)\in\mathcal{H}^{k,\alpha}(\Omega,M):\widehat{\ind}_{\tau}L_u<0\rbrace=\bigcup_{j=1}^{\infty}\mathcal{L}_-^{k,\alpha}(u_j,\varphi_j;\Lambda_j,\kappa_j),$$
    can be covered by 
    $$\lbrace(u,\varphi)\in\mathcal{H}^{k,\alpha}(\Omega,M):\widehat{\ind}_{\tau}L_u<0\rbrace\subset\bigcup_{j'=1}^{\infty}\mathcal{L}^{k,\alpha}(u_{j'}',\varphi_{j'}';\Lambda_{j'}',\kappa_{j'}'),\text{ with }\widehat{\ind}_{\tau}L_{u'_{j'}}<0.$$
    By Theorem \ref{thm: local sard smale theorem} and the relationship of sets,
    \begin{align*}
        &\lbrace\varphi\in C^{k,\alpha}(\p\Omega,M):\widehat{\ind}_{\tau}L_u=0\text{ for all }(u,\varphi)\in\mathcal{H}^{k,\alpha}(\Omega,M)\rbrace\\
        =& C^{k,\alpha}(\p\Omega,M)\setminus\lbrace\varphi\in C^{k,\alpha}(\p\Omega,M):\widehat{\ind}_{\tau}L_u<0\text{ for some }(u,\varphi)\in\mathcal{H}^{k,\alpha}(\Omega,M)\rbrace\\\supset& C^{k,\alpha}(\p\Omega,M)\setminus\Pi\left(\cup_{j'=1}^{\infty}\mathcal{L}^{k,\alpha}(u_{j'}',\varphi_{j'}';\Lambda_{j'}',\kappa_{j'}')\right)=\cap_{j'=1}^{\infty}\mathcal{B}^{k,\alpha}(u_{j'}',\varphi_{j'}';\Lambda_{j'}',\kappa_{j'}'),
    \end{align*}
    We know that $\lbrace\varphi\in C^{k,\alpha}(\p\Omega,M):\widehat{\ind}_{\tau}L_u=0\text{ for all }(u,\varphi)\in\mathcal{H}^{k,\alpha}(\Omega,M)\rbrace$ is a residual set in the Baire category sense. In other words, it is a generic set in $C^{k,\alpha}(\p\Omega,M)$ as desired. 
\end{proof}

\section{Proof of main results}\label{s: proof of main results}
\begin{proof}[Proof of Theorems \ref{thm: generic regularity for sphere valued maps}-- \ref{thm: generic regularity of stationary harmonic maps}]
For Theorem \ref{thm: generic regularity of stable harmonic maps}, we note that the space of all stable stationary harmonic maps satisfies strong compactness assumption \eqref{eq: strong convergence} provided $M$ does not admit non-constant stable harmonic 2-sphere; see \cite{Hsu}. Moreover, the dimension reduction argument in \cite{SchoenUhlenbeckRegularity} together with the assumption no non-constant stable stationary harmonic regular 0-homogeneous map from $\R^k$ to $M$ implies that any $u:\Omega\rightarrow M$ with smooth boundary data is an \textbf{HSI}. Hence, Theorem \ref{thm: genericity of Jacobi operator} applies to get a generic subset $\Phi\subset C^{\infty}(\p\Omega,M)$, such that every stable stationary harmonic $u\in H^1(\Omega,M)$ with $u|_{\p\Omega}\in\Phi$ satisfies $\widehat{\ind}_{\tau}L_u=0$ for $\tau$ satisfying \eqref{eq: range of rate}. By the definition of effective index Definition \ref{defn: effective Morse index}, together with the counting Theorem \ref{thm: count of index}, a tangent map $h_{x_0}$ of $u$, if any, must have effective Morse index 0 which is equivalent to a link of Morse index $n$.

When we specify the target in Theorem \ref{thm: generic regularity for sphere valued maps} to be round sphere $\Sp^k$, Proposition \ref{prop: index of sphere targets} tells us if we are considering maps from $n$-dimensional domain to $\Sp^{n-1}$, $4\leqslant n\leqslant 7$, the only possible tangent map with a link of Morse index $n$ is the radial projection map composed with an orthogonal matrix; when we consider maps from a 7-dimensional domain to $\Sp^k$, $k\geqslant7$, we know that there is no tangent map that has a link of Morse index 7. Theorem \ref{thm: generic regularity for sphere valued maps} follows.

Theorem \ref{thm: generic regularity of stationary harmonic maps} is proved with a similar reasoning as above. For the first part of Theorem \ref{thm: generic regularity in dimension 3}, we note that energy minimizing maps always satisfy \eqref{eq: strong convergence}. Also, in dimension 3, the effective index of $h_{x_0}$ is 0 is equivalent to $h_{x_0}$ having a stable link. Hence, Theorem \ref{thm: generic regularity in dimension 3} is proved. Second part of Theorem \ref{thm: generic regularity in dimension 3} follows similarly.
\end{proof}

In order to prove Corollary \ref{cor: rigidity at infinity}, we need the following lemma, which says any energy minimizing harmonic map restricting to a generic sub-region is uniquely minimizing.

\begin{lem}[{F. J. Almgren and E. H. Lieb \cite[Theorem 4.1(2)]{AlmgrenLieb}}]\label{lem: generic uniqueness}
    Let $3\leqslant n\leqslant 7$ and $u\in H^1(B_2^n,\Sp^{n-1})$ be an energy minimizing harmonic map. For each $0<r<2$, if $\p B_r\cap \sing u=\emptyset$, $u|_{B_r}\in H^1(B_r,\Sp^{n-1})$ is the unique minimizer of the energy functional among maps $u'\in H^1(B_r,\Sp^{n-1})$ with boundary $u'|_{\p B_r}=u|_{\p B_r}$.
\end{lem}
\cite{AlmgrenLieb} studies only energy minimizing maps between 3-dimensional domain and $\Sp^2$, but the proof only uses the fact that any such energy minimizing map has only isolated singularities and $\Sp^2$ is real analytic. These two properties also hold in the setting of Lemma \ref{lem: generic uniqueness}. Hence, an identical proof of \cite[Theorem 4.1(2)]{AlmgrenLieb} proves Lemma \ref{lem: generic uniqueness}.

\begin{proof}[Proof of Corollary \ref{cor: rigidity at infinity}]
    Denote by $p(x)=x/\vert x\vert$ the radial projection map and by $\theta=\Theta_p(0)$ its density. For a map $u\in H^1_{\loc}(\R^n,\Sp^{n-1})$ that is energy minimizing and $u_{R_j}=u(R_j\cdot)\rightarrow p$ in $H^1_{\loc}(\R^n)$, we claim that $u$ must have a singularity $x_0$ such that $\Theta_u(x_0)\geqslant\theta$. Note that in the range $3\leqslant n\leqslant 7$, such $u$ can only have isolated singularities by \cite{li2026optimalregularitystableharmonic}.
    
    We first consider the case $4\leqslant n\leqslant 7$. Indeed, since $u_{R_j}$ converges strongly to $p$ in $H^1(B_2)$ as $j\rightarrow\infty$, the $\epsilon$-regularity implies that $u_{R_j}$ converges smoothly to $p$ in $C^{\infty}_{\loc}(B_2\setminus\lbrace0\rbrace)$. For $j$ large, we would have $\Vert u_{R_j}|_{\p B_1}-p|_{\p B_1}\Vert_{C^{2,\alpha}(\Sp^{n-1})}$ is less than any prescribed threshold. Fix such a large $j$. In particular, given that $p|_{\p B_1}:\Sp^{n-1}\rightarrow\Sp^{n-1}$ has topological degree 1, $u_{R_j}$ and any small perturbation of it must also have topological degree 1. As such, none of them admits a smooth extension to $B_1$. By Theorem \ref{thm: generic regularity for sphere valued maps}, we can find a sequence of $\varphi_k\in C^{\infty}(\p B_1,M)$ which converges smoothly to $u_{R_j}|_{\p B_1}$, and energy minimizing map $u'_k$ with $u'_k|_{\p B_1}=\varphi_k$ has only singularities of density $\theta$.
    
    The energy of $u_k'$ is uniformly bounded by $\sup_kC(n)\int_{\Sp^{n-1}}\vert\nabla\varphi_k\vert^2<\infty$. After passing to a subsequence, we must have $u'_k$ converges to $u_{R_j}$ as $k\rightarrow\infty$ as $u_{R_j}$ is the only minimizer of the energy functional with boundary map $u_{R_j}|_{\p B_1}$ by Lemma \ref{lem: generic uniqueness}. Pick $x_k'\in \sing u'_k\ne\emptyset$. By boundary regularity in \cite{LinGradientEstimate}, we must have $x_k'\rightarrow x'\in B_1$, and the upper semi-continuity of density yields
    $$\Theta_{u_{R_j}}(x')\geqslant\limsup_{k\rightarrow\infty}\Theta_{u'_k}(x_k')=\theta.$$
    
    Scaling back $x_0=R_jx'$ satisfies the claim. For $n=3$, H. Brezis, J. Coron and E. H. Lieb \cite{BrezisCoronLieb} show that the only tangent maps at a singularity of energy minimizing maps from $\R^3$ to $\Sp^2$ are the radial projection map composed with orthogonal matrices. The same argument as above shows $\sing u\ne\emptyset$. Hence, the existence of such $x_0$ is immediate.
    
    Hence, the monotonicity formula reads
    $$\theta\leqslant\Theta_{u}(x_0)\leqslant\Theta_u(x_0,R)\leqslant\lim_{j\rightarrow\infty}\Theta_u(0,R_j)=\theta.$$
    We must have $\Theta_u(x_0,R)=\theta$ for all $R>0$. The equality case of monotonicity formula gives that $u$ is 0-homogeneous with respect to $x_0$, 
    $$u(x)=h\left(\frac{x-x_0}{\vert x-x_0\vert}\right).$$
    In the above equation, if we take $j\rightarrow\infty$, $u_{R_j}\rightarrow h$ smoothly away from 0. Hence, $h=p$, and the Corollary is proved.
\end{proof}
\appendix

\section{Asymptotic rate of Jacobi fields}\label{a: asym rate}
In this section, we establish the a priori estimate of the growth rate of a Jacobi field when the underlying harmonic map is close to a tangent map.
\begin{lem}\label{lem: three circle lemma for Jacobi fields}
    For any $\sigma>0$, there exists $K(\sigma)>1$, with the following property. For any $\gamma\in(-1,0)$, given a regular 0-homogeneous map $h\in\mathcal{T}_{n}$ with $\dist(\gamma,\Gamma(h)\cup\lbrace-(n-2)/2\rbrace)\geqslant\sigma$, if $X\in L^2\cap C^2_{\text{loc}}(B_1\setminus B_{K^{-3}},T_hM)$ solves $L_hX=0$, then the growth rate
    $$J^{\gamma}_K(X;r)=\int_{A_{K^{-1}r,r}}\vert X(x)\vert^2\vert x\vert^{-n-2\gamma}\dif x,$$
    satisfies the inequality
    $$J^{\gamma}_{K}(X;K^{-2})-2J^{\gamma}_{K}(X;K^{-1})+J^{\gamma}_{K}(X;1)\geqslant0,$$
    where the equality holds if and only if $X=0$.
\end{lem}
The proof of Lemma \ref{lem: three circle lemma for Jacobi fields} works exactly as the proof of \cite[Lemma A.1]{LiWangGenericRegularity} or \cite[Lemma B.1]{CLWNon-persistence}, taking into consideration the expansion of Jacobi field \eqref{eq: expansion of Jacobi field},\eqref{eq: coefficients in the expansion of Jacobi fields}. The estimate of this form can be traced back to \cite{SimonAsymptotics}.

As a corollary, we obtain the following statement for Jacobi operator of a general \textbf{HSI}, which is the perturbed version of Lemma \ref{lem: three circle lemma for Jacobi fields}. Note that a vector field $X$ satisfies the equation in Lemma \ref{lem: perturbed three circle lemma for Jacobi fields} provided $X$ is a Jacobi field of $u$ and $u$ is close to the 0-homogeneous map $h$ as in the conclusion of Corollary \ref{Cor: uniform graph on tangent maps}.

\begin{lem}\label{lem: perturbed three circle lemma for Jacobi fields}
    Given $\sigma,\Lambda>0$, there exists $\epsilon(\sigma,\Lambda)>0$ such that the following holds.

    For $\gamma\in[-\Lambda,\Lambda],h\in\mathcal{T}_n(\Lambda)$ with
    $$\dist(\gamma,\Gamma(h)\cup\lbrace-(n-2)/2\rbrace)\geqslant\sigma.$$
    Assume $0\ne X\in H^1(A_{K^{-3},1},\R^N)$ is a vector field satisfying $\Vert X-\pi_h X\Vert_{L^2(A_{K^{-3},1})}\leqslant\epsilon\Vert X\Vert_{L^2(A_{K^{-3},1})}$, where $\pi_hX(x)$ is the projection of $X(x)$ to the tangent space $T_{h(x)}M$. Assume also that $X$ is a weak solution to the following equation
    $$\Delta X+DA_h(X,\nabla h,\nabla h)+2A_h(\nabla X,\nabla h)=\vert x\vert ^{-1}(R_1(X)+R_2(\nabla X)),$$
    where 
    $$\Vert R_1(X)\Vert_{L^2(A_{K^{-3},1})}\leqslant\epsilon\Vert X\Vert_{L^{2}(A_{K^{-3},1})},$$
    and 
    $$\Vert R_2(\nabla X)\Vert_{L^2(A_{K^{-3},1})}\leqslant\epsilon\Vert \nabla X\Vert_{L^{2}(A_{K^{-3},1})}.$$
    Then
    $$J^{\gamma}_K(X;K^{-2})-2(1+\epsilon)J^{\gamma}_K(X;K^{-1})+J^{\gamma}_K(X;1)\geqslant0.$$
\end{lem}

\begin{proof}
    We prove by contradiction. Assume the statement is not true, there would exist $h_j\in\mathcal{T}_n(\Lambda),X_j\in H^1(A_{K^{-3},1},\R^N)$ which violate the claimed conclusion for $\epsilon=1/j$. In particular,
    $$J_K^{\gamma}(X_j;1)+J_K^{\gamma}(X_j;K^{-2})\leqslant2(1+1/j)J^{\gamma}_K(X_j;K^{-1}).$$
    After passing to a subsequence, $h_j\rightarrow h\in\mathcal{T}_n(\Lambda)$ as $j\rightarrow\infty$, locally smoothly away from $0$. The smooth convergence of $h_j|_{\p B_1}$ implies that $\gamma_k^{\pm}(h_j)\rightarrow\gamma_k^{\pm}(h)$ as $j\rightarrow\infty$. 
    On the other hand, renormalizing, we can assume $J^{\gamma}_K(X_j;K^{-1})=1$. By the classical elliptic PDE estimate, we can assume that $X_j$ converges strongly in $L^2_{\loc}(A_{K^{-3},1},\R^N)$ and weakly in $W^{1,2}_{\loc}(A_{K^{-3},1},\R^N)$ to a vector field $X$, which takes values in $T_{h(x)}M$ for a.e. $x$. By our assumption, $X$ is a Jacobi field for $h$. Moreover,
    $$J^{\gamma}_K(X;1)+J^{\gamma}_K(X;K^{-2})\leqslant2J^{\gamma}_K(X;K^{-1}).$$
    In particular, Lemma \ref{lem: three circle lemma for Jacobi fields} implies that $X=0$, which contradicts the fact that $J^{\gamma}_K(X;K^{-1})=1$. This proves the lemma. 
\end{proof}

Now, we are ready to estimate the growth rate of the $L^2$ norm of a Jacobi-type field of a general \textbf{HSI} that is close to a 0-homogeneous map.
\begin{lem}\label{lem: Jacobi operator growth estimate}
    Let $\sigma\in(0,1),\Lambda>0,K=K(\sigma/2)>0$ be given in Lemma \ref{lem: three circle lemma for Jacobi fields}. Then there exists $\epsilon_0=\epsilon_0(\sigma,\Lambda)$ such that the following holds.

    Let $\gamma\in[-\Lambda,\Lambda]$, $u\in W^{1,2}(B_4,M)$ be an \textbf{HSI} and $h\in\mathcal{T}_n(\Lambda)$ with $\sing u=\lbrace0\rbrace$ and 
    $$\sup_{x\in B_3}\sum_{i=0}^2\vert x\vert^i\vert\nabla^i(u-h)(x)\vert\leqslant\epsilon_0.$$
    Assume further that $$\dist(\gamma,\Gamma(h)\cup\lbrace-(n-2)/2\rbrace)\geqslant\sigma.$$
    Let $X\in W^{2,2}_{\loc}(\reg u,T_uM)$ be a section with asymptotic rate bound and $f=L_uX$ satisfies
    $$\mathcal{AR}_0(X)>\gamma; \qquad F=\sup_{l>0}\Vert \vert x\vert^{2-\gamma-n/2}f\Vert_{L^2(A_{K^{-l-1},K^{-l}})}<\infty.$$
    Then for each $l>0$,
    $$\Vert X\Vert_{L^2(A_{K^{-l-1},K^{-l}})}\leqslant C(\sigma,\Lambda)(F+\Vert X\Vert_{L^2(A_{K^{-1},1})})K^{-l(n/2+\gamma)}.$$
\end{lem}
\begin{proof}
    We first show the following property. Set $X_l=X(K^{-l}\cdot)$. Then, there exists a large constant $C_0=C_0(\sigma,\Lambda)$ such that
    $$l\mapsto J_l:=\max\lbrace K^{2\gamma l}J_K^{\gamma}(X_l;1),C_0^2 F^2\rbrace\text{ is non-increasing for all } l\geqslant0.$$
    Note that $K^{2\gamma l}J^{\gamma}_K(X_l;1)=J^{\gamma}_K(X;K^{-l})$. Once this is proved, then $J_l\leqslant J_0$ gives the desired estimate. Suppose it is not true, namely for some $l_0\geqslant0$, there holds $J_{l_0}<J_{l_0+1}$. Since $J_{l_0}\geqslant C^2_0F^2$, we see that $J_{l_0+1}$ must equal $K^{2\gamma l}J^{\gamma}_K(X_l;1)$, which gives (note that $K$ is a constant depending on $\sigma$)
    $$F\leqslant C_0^{-1}K^{\gamma(l_0+1)}J^{\gamma}_K(X_{l_0+1};1)^{1/2}\leqslant C_0^{-1}C(\sigma,\Lambda)\Vert X_{l_0}\Vert_{L^2(A_{K^{-2},K^{-1})})}K^{\gamma l_0}.$$
    From this, we then calculate that $X_{l_0}$ satisfies the equation in Lemma \ref{lem: perturbed three circle lemma for Jacobi fields} provided $C_0$ is chosen large. We then apply Lemma \ref{lem: perturbed three circle lemma for Jacobi fields} to get 
    $$J^{\gamma}_K(X_{l_0};K^{-2})-J^{\gamma}_K(X_{l_0};K^{-1})\geqslant J^{\gamma}_K(X_{l_0};K^{-1})-J^{\gamma}_K(X_{l_0};1).$$
    Also,
    $$J^{\gamma}_K(X_{l_0};K^{-1})=K^{-2\gamma l_0}J_{l_0+1}>K^{-2\gamma l_0}J_{l_0}\geqslant J^{\gamma}_K(X_{l_0};1).$$
    Hence,
    $$J_{l_0+2}\geqslant K^{2\gamma (l_0+2)}J^{\gamma}_K(X_{l_0+2};1)=K^{2\gamma l_0}J^{\gamma}_K(X_{l_0};K^{-2})>K^{2\gamma l_0}J^{\gamma}_K(X_{l_0};K^{-1})=J_{l_0+1}.$$
    From there, we can start the induction, to get $J_{l}$ is strictly increasing for $l\geqslant l_0$. However, by the assumption $\mathcal{AR}_0(X)>\gamma$, we have that $\lim_{l\rightarrow\infty}J_{l}=C_0^2F^2$, this is a contradiction.
\end{proof}

\begin{lem}\label{lem: tame field upper bound}
    Let $\sigma,\Lambda,K,\epsilon_0,u,h$ be the same as Lemma \ref{lem: Jacobi operator growth estimate}. By choosing $\gamma=-\sigma$ and making $\epsilon_0$ a little smaller, we have that for each tame Jacobi field $X$ of $u$, there exists a vector $v$ such that
    $$\vert v\vert+\sup_{l>0}\Vert X-\p_vu\Vert_{L^2_{-\sigma}(A_{K^{-l-1},K^{-l}})}\leqslant C(\Lambda,\sigma)\Vert X\Vert_{L^2(A_{\epsilon_0,2})}.$$
\end{lem}
\begin{proof}
    By definition of tame Jacobi field and Corollary \ref{cor: asy rate of tame field}, there exists a $v\in\R^n$ such that $\mathcal{AR}_0(X-\p_vu)\geqslant0$. Hence, we can apply Lemma \ref{lem: Jacobi operator growth estimate} to derive
    $$\Vert X-\p _vu\Vert_{L^2_{-\sigma}(A_{K^{-l-1},K^{-l}})}\leqslant C(\Lambda,\sigma)\Vert X-\p_v u\Vert_{L^2(A_{K^{-1},1})}.$$
    On the other hand, using Lemma \ref{lem: lower bound of derivative of tangent map}, and the smallness assumption of $u-h$, we have that for small $s>0$,
    \begin{align*}
        \vert v\vert&\leqslant C(\Lambda)(\Vert\p_vu\Vert_{L^2(A_{s,2s})}s^{1-\frac{n}{2}}+\epsilon_0\vert v\vert)
        \\&\leqslant Cs^{1-\frac{n}{2}}(\Vert\p_v u-X\Vert_{L^2(A_{s,2s})}+\Vert X\Vert_{L^2(A_{s,2s})})+C\epsilon_0\vert v\vert
        \\&\leqslant Cs^{1-\sigma}\Vert X-\p_vu\Vert_{L^2(A_{K^{-1},1})}+Cs^{1-\frac{n}{2}}\Vert X\Vert_{L^2(A_{s,2s})}+C\epsilon_0\vert v\vert
        \\&\leqslant Cs^{1-\sigma}(\Vert X\Vert_{L^2(A_{K^{-1},1})}+\Lambda \vert v\vert)+Cs^{1-\frac{n}{2}}\Vert X\Vert_{L^2(A_{s,2s})}+C\epsilon_0\vert v\vert
    \end{align*}
    Fix $s=s(\sigma,\Lambda)$ small so that $C\Lambda s^{1-\sigma}=1/3$, and $\epsilon_0=\min\lbrace (3C)^{-1},K^{-1},s\rbrace$, we see
    $$\vert v\vert\leqslant C\Vert X\Vert_{L^2(A_{\epsilon_0,1})}.$$
    This completes the proof.
\end{proof}

\section{Quantitative estimate for the distance between harmonic maps}\label{a: quantitaive estimates}
In this section, we first obtain the following statement on the uniqueness of tangent maps.
\begin{prop}\label{Prop: uniqueness of tangent map}
    Given $0<\epsilon<1$ and $\Lambda>0$, there exists $\delta=\delta(\epsilon,\Lambda)$ with the following property.

    Let $h\in\mathcal{T}_n(\Lambda),u\in W^{1,2}(B_4,M)$ be a stationary harmonic map and $r\in(0,1/16)$. Suppose
    \begin{enumerate}
        \item $\Theta_u(0,2)\leqslant\Theta_u(0,r)+\delta$;
        \item There exists $s\in(2r,1/2)$ such that $$\sup_{A_{s,2s}}\sum_{i=0}^2\vert x\vert^i\vert\nabla^i(u-h)(x)\vert\leqslant\delta.$$
    \end{enumerate}
    Then,
    $$\sup_{A_{2r,1}}\sum_{i=0}^2\vert x\vert^i\vert\nabla^i(u-h)(x)\vert\leqslant\epsilon.$$
\end{prop}
\begin{proof}
    We prove by contradiction. Assume the proposition is false, we would get a sequence of $h_j,u_j,1/j,s_j,r_j$ instead of $h,u,\delta,s,r$ in the assumption but for some $\epsilon>0$, there holds
    \begin{align}\label{eq: uniqueness of tangent maps, contradiction assumption}
        \sup_{x\in A_{2r_j,1}}\sum_{i=0}^2\vert x\vert^i\vert\nabla^i(u_j-h_j)(x)\vert\geqslant\epsilon.
    \end{align}
    Passing to a subsequence, assume $h_j\rightarrow h$ in $C^{\infty}_{\text{loc}}(B_4\setminus\lbrace0\rbrace)$. By the scaling invariance of both $h_j$ and $h$, as well as $h_j|_{\p B_1}\rightarrow h|_{\p B_1}$ smoothly, we have that $\sup_{x\ne0}\vert x\vert^i\vert\nabla^i(h_j-h)\vert\rightarrow0.$ Hence, by slightly enlarging $\delta_j$, we have that
    \begin{align}\label{eq: uniqueness of tangent maps, limit to tangent map}
        \sup_{x\in A_{s_j,2s_j}}\sum_{i=0}^2\vert x\vert^i\vert\nabla^i(u_j-h)(x)\vert\leqslant\delta_j\rightarrow0.
    \end{align}
    Let $\bar{s}_j^-\leqslant s_j\leqslant\bar{s}_j^+$ define the maximal annulus such that 
    $$\sup_{x\in A_{\bar s_j^-,\bar s_j^+}}\sum_{i=0}^2\vert x\vert^i\vert\nabla^i(u_j-h)(x)\vert\leqslant\epsilon'.$$
    where $\epsilon'=\epsilon'(h)<\epsilon$ is a positive number to be determined. We wish to show 
    \begin{align}\label{eq: uniqueness of tangent maps, refined assumption}
        \bar{s}_j^-\leqslant\frac{3}{2}r_j\text{ and }\bar{s}_j^+\geqslant\frac{3}{2},
    \end{align}
    which will lead to a contradiction and hence proves the proposition.

    To this end, let us consider $u_j'=u_j(s_j\cdot)$. Then $u_j'$ converges to $h$ by our assumption \eqref{eq: uniqueness of tangent maps, limit to tangent map}. To proceed with the estimates, consider $X_j=u_j'-h$ and $Y_j(t,\omega)=X_j(r\omega)$, where $r=\vert x\vert,\omega=x/r$ are polar coordinate on $\R^n$ and $t=-\log r$, we get 
    $$\vert Y_j\vert_2^*(t):=\sum_{k+l\leq2,k,l\geqslant0}\vert\p_t^k\nabla^l_{S^{n-1}}Y_j(t,\cdot)\vert_{C^0(S^{n-1})}\leqslant\delta_j,\textup{ for all }t\in[-\log2,0].$$
    Set also the rescaled radii $s_j^{\pm}=\bar s_j^{\pm}/s_j$ and $T_j^{\pm}=-\log s^{\pm}_j$. If $\epsilon'$ is small enough, then we may apply the estimate in \cite[Theorem 11.1]{EdelenDegeneration} to yield a constant $\alpha=\alpha(\Lambda)>0$ such that 
    $$\vert Y_j\vert_2^*(t)\leqslant\delta_j^{\alpha},\text{ for all }t\in(T_j^+,0).$$
    By similar consideration as above and the reasoning in \cite{CLWNon-persistence},
    the vector field $\check{Y}_j=Y_j(\ln 2-t,\omega)$ also satisfies
    $$\vert\check{Y}_j\vert_2^*(t)\leqslant \delta_j^{\alpha},\text{ for all }t\in(0,\ln 2-T_j^-).$$
    Since the condition of \cite[Theorem 11.1]{EdelenDegeneration} is fulfilled whenever $\bar s_j^-\geqslant r_j$ and $\bar s_j^+\leqslant2$, by the assumption (1) of Proposition \ref{Prop: uniqueness of tangent map} and the definition of $s_j^{\pm}$. But when $j$ is large, the control of $Y_j$ and $\check{Y}_j$ would contradict the maximality of $\bar s_j^{\pm}$ by the continuity of $u_j$ and $h$ away from 0. Hence the proposition is proved.
\end{proof}
\begin{cor}\label{Cor: uniform graph on tangent maps}
    Let $u_j\in W^{1,2}(B_4,M)$ be a sequence of \textbf{HSI} with $\sing u_j=\lbrace 0\rbrace$ and $r_{u_j}(x)\geqslant\Lambda^{-1}\vert x\vert$ in $B_3$. Moreover, suppose
    $$\Theta_{u_j}(0,2)-\Theta_{u_j}(0)\rightarrow0,\text{ as }j\rightarrow\infty.$$
    Then, after passing to a subsequence, we have that both $u_j$ and its tangent map $h_j$ at $0$ satisfy
    $$\sup_{x\in B_1\setminus\lbrace0\rbrace}\sum_{i=0}^2\vert x\vert^i\vert\nabla^i(u_j-h)(x)\vert\leqslant\epsilon_j,$$
    and 
    $$\sup_{x\in B_1\setminus\lbrace0\rbrace}\sum_{i=0}^2\vert x\vert^i\vert\nabla^i(h_j-h)(x)\vert\leqslant\epsilon_j,$$
    for an $h\in\mathcal{T}_n(\Lambda)$ and a sequence $\epsilon_j\rightarrow0$.
\end{cor}
\begin{proof}
    First the regularity scale assumption of $u_j$ implies that $h_j\in\mathcal{T}_n(\Lambda)$. By passing to a subsequence, we can assume that $h_j$ converges locally smoothly in $B_4\setminus\lbrace0\rbrace$ to $h\in\mathcal{T}_n(\Lambda)$. By the scaling invariance of $h_j$ and $h$, and the fact that $h_j|_{\p B_1}\rightarrow h|_{\p B_1}$ smoothly, it is direct that (passing to a further subsequence if necessary)
    $$\sup_{x\in B_1\setminus\lbrace0\rbrace}\sum_{i=0}^2\vert x\vert^i\vert\nabla^i(h_j-h)(x)\vert\leqslant1/j\rightarrow0.$$

    On the other hand, for each fixed $j$, by the uniqueness of tangent map of \textbf{HSI}, we have that $u_j(s_k\cdot)\rightarrow h_j$ as $k\rightarrow\infty$ smoothly away from $0$, for any sequence $s_k\rightarrow0$. Hence we can choose a sequence $s_j=s_{k_j}\rightarrow0$ as $j\rightarrow\infty$, such that $$\sup_{x\in A_{s_j,2s_j}}\sum_{i=0}^2\vert x\vert^i\vert\nabla^i(u_j-h)\vert\leqslant 1/j\rightarrow0$$
    As a result, for any $\epsilon>0$, for each large $j$, the conditions of Proposition \ref{Prop: uniqueness of tangent map} are fulfilled for any $r<s_j/2$, which gives
    $$\sup_{x\in B_1\setminus\lbrace0\rbrace}\sum_{i=0}^2\vert x\vert^i\vert\nabla^i(u_j-h)(x)\vert\leqslant\epsilon.$$
    By varying $\epsilon$ and passing to a further subsequence, we get the desired estimate.
\end{proof}

Once we establish the quantitative uniqueness of the tangent maps, we are able to obtain the finer estimates of the distance of two harmonic maps that are very close. 

\begin{lem}\label{Lem: hsi with same singularity}
    Let $\gamma\in(-1,0),\sigma\in(0,1)$ and $\Lambda>0$. There exist positive constants $\delta=\delta(\gamma,\sigma,\Lambda),K=K(\sigma)$ such that the following holds. Suppose $u^1,u^2$ are \textbf{HSI} on $B_3$ with
    \begin{enumerate}
        \item $\sing u^i\cap B_3=\lbrace0\rbrace$ and
        $$\Theta_{u^i}(0,2)-\Theta_{u^i}(0)\leqslant\delta.$$
        Moreover, the tangent map $h^i$ of $u^i$ at $0$ satisfies
        $$\dist(\gamma,\Gamma(h^i)\cup\lbrace-(n-2)/2\rbrace)\geqslant\sigma.$$
        \item The regularity scales of $u^i$ satisfy
        $$r_{u^i}(x)\geqslant\Lambda^{-1}\vert x\vert.$$
        \item $$\sup_{B_1\setminus B_{1/2}}\sum_{k=0}^2\vert x\vert^k\vert\nabla^k(u^1-u^2)(x)\vert\leqslant\delta.$$
    \end{enumerate}
    Then we have a global estimate of $u_1-u_2$:
    $$\sup_{B_1\setminus\lbrace0\rbrace}\sum_{k=0}^2\vert x\vert^k\vert\nabla^k(u^1-u^2)\vert\leqslant1,$$
    and
    $$\Vert u_1-u_2\Vert_{C^{0;\gamma}(B_{1/2})}\leqslant C(\gamma,\sigma,\Lambda)\Vert u_1-u_2\Vert_{L^{2}(A_{K^{-3},1})}.$$
\end{lem}
\begin{proof}
    Let us prove by contradiction. If the lemma is not true, then there would exist a sequence $\gamma_j\in(-1,0)$ in place of $\gamma$ and $u_j^i,1/j$ in place of $u_i,\delta$ in the assumption of the lemma, but at least one of the conclusions of the lemma does not hold. By the regularity scale assumption (2), we have the uniform energy bound of the sequences $u_j^i,i=1,2$. Hence, after passing to a subsequence, we can assume $\gamma_j\rightarrow\gamma_{\infty},u_j^i\rightarrow u^i_{\infty}$ strongly in $W^{1,2}$ and smoothly away from 0. By assumption (3) above, we see that $u_{\infty}^1=u_{\infty}^2$. By assumption (1), $h_{\infty}:=u^1_{\infty}=u^2_{\infty}$ is a regular 0-homogeneous map.

    Given the convergence of $u_j^i$, by Corollary \ref{Cor: uniform graph on tangent maps}, for large $j$, we have that
    $$\sup_{B_1\setminus\lbrace0\rbrace}\sum_{k=0}^2\vert x\vert^k\vert\nabla^k(u_j^i-h_{\infty})\vert\leqslant\epsilon_j\rightarrow0.$$
    As such,
    $$\sup_{B_1\setminus\lbrace0\rbrace}\sum_{k=0}^2\vert x\vert^k\vert\nabla^k(u_j^1-u_j^2)\vert\leqslant\sup_{B_1\setminus\lbrace0\rbrace}\sum_{k=0}^2\vert x\vert^k\vert\nabla^k(u_j^1-h_{\infty})\vert+\sup_{B_1\setminus\lbrace0\rbrace}\sum_{k=0}^2\vert x\vert^k\vert\nabla^k(u_j^2-h_{\infty})\vert\leqslant2\epsilon_j,$$
    which tells us the first conclusion of the lemma must hold for sufficiently large $j$.

    In order to show another estimate of the lemma, we note that $u^2_j-u^1_j$ is a bounded map to $\R^N$. As a result,
    $$\mathcal{AR}_0(u_j^2-u_j^1)\geqslant0>\gamma.$$
    By the same argument of Lemma \ref{lem: Jacobi operator growth estimate}, we have that for the same constant $K(\sigma)$ as in Lemma \ref{lem: Jacobi operator growth estimate}, the estimate holds
    $$\Vert u_j^2-u_j^1\Vert_{L^2(A_{K^{-l-1},K^{-l}})}\leqslant C(\sigma,\Lambda)\Vert u_j^2-u_j^1\Vert_{L^2(A_{K^{-1}},1)}K^{-l(n/2+\gamma)}.$$
    Note that $X=u_j^2-u_j^1$ solves a PDE of the form $L_{u_j^1}X=a\nabla X+b$, where $\vert x\vert\vert a(x)\vert+\vert x\vert^2\vert b(x)\vert\leqslant C\delta$. The $C^0$ estimate follows by classical PDE estimates.
\end{proof}
In general, we can also estimate the distance of singularity of two different \textbf{HSI}.
\begin{lem}\label{lem: HSI with different singularity}
    Let $\gamma\in(-1,0),\sigma\in(0,1)$ and $\Lambda>0$. There exist positive constants $\delta=\delta(\gamma,\sigma,\Lambda),K=K(\sigma)$ such that the following holds. Suppose $u_1,u_2$ are \textbf{HSI} on $B_3$ with
    \begin{enumerate}
        \item $\sing u_i\cap B_3=\lbrace x_i\rbrace$ for $x_1=0$ and $\vert x_2\vert\leqslant\delta$ and
        $$\Theta_{u_i}(x_i,2)-\Theta_{u_i}(x_i)\leqslant\delta.$$
        Moreover, the tangent map $h_i$ of $u_i$ at $x_i$ satisfies
        $$\dist(\gamma,\Gamma(h_i)\cup\lbrace-(n-2)/2\rbrace)\geqslant\sigma.$$
        \item The regularity scales of $u_i$ satisfies
        $$r_{u_i}(x)\geqslant\Lambda^{-1}\vert x-x_i\vert.$$
        \item $$\sup_{B_1\setminus B_{1/2}}\sum_{k=0}^2\vert x\vert^k\vert\nabla^k(u_1-u_2)(x)\vert\leqslant\delta.$$
    \end{enumerate}
    Then $$\vert x_2\vert\leqslant C(\gamma,\sigma,\Lambda)\Vert u_1-u_2\Vert_{L^2(A_{\delta,2})}$$
\end{lem}
\begin{proof}
    First note that $u_1$ and $\tilde u_2=u_2(\cdot+x_2)$ satisfy the hypotheses of Lemma \ref{Lem: hsi with same singularity}. Fix $r>0$ to be chosen later. We first claim the following.
    \begin{claim}
        If $\delta$ and $\vert x_2\vert/r$ are sufficiently small, then
        $$\Vert u_2-\tilde u_2\Vert_{L^2(A_{r,2r})}\geqslant c\vert x_2\vert r^{n/2-1}.$$ 
    \end{claim}
    \begin{proof}[Proof of Claim]
        First note that by the scaling invariance of $h_2$, we have that 
        $$\Vert h_2-h_2(\cdot-x_2)-\nabla_{x_2}h_2\Vert_{L^2(A_{r,2r})}\leqslant C(n,\Lambda)\vert x_2\vert^2 r^{n/2-2}.$$
        On the other hand, by choosing $\delta\leqslant\delta(\epsilon)$, there holds
        $$\Vert u_2-\tilde u_2-h_2+h_2(\cdot-x_2)\Vert_{L^2(A_{r,2r})}\leqslant C\vert x_2\vert r^{n/2}\sup_{A_{r/2,3r}}\vert\nabla (u_2-h_2)\vert\leqslant C\epsilon\vert x_2\vert r^{n/2-1}.$$
        Combining above two inequalities, as well as Lemma \ref{lem: lower bound of derivative of tangent map}, we have that 
        $$\Vert u_2-\tilde u_2\Vert_{L^2(A_{r,2r})}\geqslant(c-C\vert x_2\vert/r-C\epsilon)\vert x_2\vert r^{n/2-1}.$$
        Choosing $\epsilon$ small first, hence fixing $\delta=\delta(\epsilon)$ as given by Proposition \ref{Prop: uniqueness of tangent map}, then choosing $\vert x_2\vert/r$ small gives the claim.
    \end{proof}
    On the other hand, we have another side of inequality
    $$\Vert u_2-\tilde u_2\Vert_{L^2(A_{r,1})}\leqslant C\vert x_2\vert\Vert\nabla u_2\Vert_{L^2(A_{r/2,2}(x_2))}\leqslant C(n,\Lambda)\vert x_2\vert.$$
    Hence, we can apply Lemma \ref{Lem: hsi with same singularity} to deduce that, if $r=r(\sigma,\Lambda,\gamma)$ is small enough, 
    \begin{align*}
        \vert x_2\vert r^{n/2-1}&\leqslant C\Vert u_2-\tilde u_2\Vert_{L^2(A_{r,2r})}\\&\leqslant C(\Vert u_1-\tilde u_2\Vert_{L^2(A_{r,2r})}+\Vert u_1-u_2\Vert_{L^2(A_{r,2r})})\\&\leqslant C(\Vert u_1-u_2\Vert_{L^2(A_{r,1})}+r^{n/2+\gamma}\Vert u_1-\tilde u_2\Vert_{L^2(A_{r,1})})\\&\leqslant C(1+r^{n/2+\gamma})\Vert u_1-u_2\Vert_{L^2(A_{r,1})}+Cr^{n/2+\gamma}\vert x_2\vert.
    \end{align*}
    Fixing $r=\delta\leqslant\delta(\sigma,\Lambda,\gamma)$ according to above claim, we can absorb the $\vert x_2\vert$ on the RHS to the LHS by the fact that $\gamma\in(-1,0)$ and get the desired result.
\end{proof}

\bibliography{reference}

@book {LinWangbook,
    AUTHOR = {Lin, Fanghua and Wang, Changyou},
     TITLE = {The analysis of harmonic maps and their heat flows},
 PUBLISHER = {World Scientific Publishing Co. Pte. Ltd., Hackensack, NJ},
      YEAR = {2008},
     PAGES = {xii+267},
      ISBN = {978-981-277-952-6; 981-277-952-3},
   MRCLASS = {58E20 (53C44 58J35)},
  MRNUMBER = {2431658},
       DOI = {10.1142/9789812779533},
       URL = {https://doi-org.proxy.library.cornell.edu/10.1142/9789812779533},
}

@article {SimonAsymptotics,
    AUTHOR = {Simon, Leon},
     TITLE = {Asymptotics for a class of nonlinear evolution equations, with
              applications to geometric problems},
   JOURNAL = {Ann. of Math. (2)},
  FJOURNAL = {Annals of Mathematics. Second Series},
    VOLUME = {118},
      YEAR = {1983},
    NUMBER = {3},
     PAGES = {525--571},
      ISSN = {0003-486X,1939-8980},
   MRCLASS = {58G11 (35B40 49F99 58E20)},
  MRNUMBER = {727703},
MRREVIEWER = {Helmut\ Kaul},
       DOI = {10.2307/2006981},
       URL = {https://doi-org.proxy.library.cornell.edu/10.2307/2006981},
}

@article {Xin,
    AUTHOR = {Xin, Y. L.},
     TITLE = {Some results on stable harmonic maps},
   JOURNAL = {Duke Math. J.},
  FJOURNAL = {Duke Mathematical Journal},
    VOLUME = {47},
      YEAR = {1980},
    NUMBER = {3},
     PAGES = {609--613},
      ISSN = {0012-7094,1547-7398},
   MRCLASS = {58E20},
  MRNUMBER = {587168},
MRREVIEWER = {Samuel\ I.\ Goldberg},
       URL = {http://projecteuclid.org.proxy.library.cornell.edu/euclid.dmj/1077314183},
}

@article {ElSoufi,
    AUTHOR = {El Soufi, Ahmad},
     TITLE = {Indice de {M}orse des applications harmoniques de la sph\`ere},
   JOURNAL = {Compositio Math.},
  FJOURNAL = {Compositio Mathematica},
    VOLUME = {95},
      YEAR = {1995},
    NUMBER = {3},
     PAGES = {343--362},
      ISSN = {0010-437X,1570-5846},
   MRCLASS = {58E20},
  MRNUMBER = {1318092},
       URL = {http://www.numdam.org.proxy.library.cornell.edu/item?id=CM_1995__95_3_343_0},
}

@article {SchoenUhlenbeck,
    AUTHOR = {Schoen, Richard and Uhlenbeck, Karen},
     TITLE = {Regularity of minimizing harmonic maps into the sphere},
   JOURNAL = {Invent. Math.},
  FJOURNAL = {Inventiones Mathematicae},
    VOLUME = {78},
      YEAR = {1984},
    NUMBER = {1},
     PAGES = {89--100},
      ISSN = {0020-9910,1432-1297},
   MRCLASS = {58E20},
  MRNUMBER = {762354},
MRREVIEWER = {Luc\ Lemaire},
       DOI = {10.1007/BF01388715},
       URL = {https://doi-org.proxy.library.cornell.edu/10.1007/BF01388715},
}

@article {CaffarelliHardSimon,
    AUTHOR = {Caffarelli, Luis and Hardt, Robert and Simon, Leon},
     TITLE = {Minimal surfaces with isolated singularities},
   JOURNAL = {Manuscripta Math.},
  FJOURNAL = {Manuscripta Mathematica},
    VOLUME = {48},
      YEAR = {1984},
    NUMBER = {1-3},
     PAGES = {1--18},
      ISSN = {0025-2611,1432-1785},
   MRCLASS = {53A10 (49F10)},
  MRNUMBER = {753722},
       DOI = {10.1007/BF01168999},
       URL = {https://doi-org.proxy.library.cornell.edu/10.1007/BF01168999},
}

@article {HardMou,
    AUTHOR = {Hardt, Robert and Mou, Libin},
     TITLE = {Harmonic maps with fixed singular sets},
   JOURNAL = {J. Geom. Anal.},
  FJOURNAL = {The Journal of Geometric Analysis},
    VOLUME = {2},
      YEAR = {1992},
    NUMBER = {5},
     PAGES = {445--488},
      ISSN = {1050-6926,1559-002X},
   MRCLASS = {58E20},
  MRNUMBER = {1184709},
MRREVIEWER = {Martin\ Fuchs},
       DOI = {10.1007/BF02921301},
       URL = {https://doi-org.proxy.library.cornell.edu/10.1007/BF02921301},
}

@article {LockhartMcOwen,
    AUTHOR = {Lockhart, Robert B. and McOwen, Robert C.},
     TITLE = {Elliptic differential operators on noncompact manifolds},
   JOURNAL = {Ann. Scuola Norm. Sup. Pisa Cl. Sci. (4)},
  FJOURNAL = {Annali della Scuola Normale Superiore di Pisa. Classe di
              Scienze. Serie IV},
    VOLUME = {12},
      YEAR = {1985},
    NUMBER = {3},
     PAGES = {409--447},
      ISSN = {0391-173X,2036-2145},
   MRCLASS = {58G15 (47A53 47F05 58G10)},
  MRNUMBER = {837256},
MRREVIEWER = {William\ Margulies},
       URL = {http://www.numdam.org.proxy.library.cornell.edu/item?id=ASNSP_1985_4_12_3_409_0},
}

@book {SimonBook,
    AUTHOR = {Simon, Leon},
     TITLE = {Theorems on regularity and singularity of energy minimizing
              maps},
    SERIES = {Lectures in Mathematics ETH Z\"urich},
      NOTE = {Based on lecture notes by Norbert Hungerb\"uhler},
 PUBLISHER = {Birkh\"auser Verlag, Basel},
      YEAR = {1996},
     PAGES = {viii+152},
      ISBN = {3-7643-5397-X},
   MRCLASS = {58E20 (35J60 49N60 58G03)},
  MRNUMBER = {1399562},
MRREVIEWER = {Nathan\ Smale},
       DOI = {10.1007/978-3-0348-9193-6},
       URL = {https://doi-org.proxy.library.cornell.edu/10.1007/978-3-0348-9193-6},
}

@article {Hsu,
    AUTHOR = {Hsu, Deliang},
     TITLE = {An approach to the regularity for stable-stationary harmonic
              maps},
   JOURNAL = {Proc. Amer. Math. Soc.},
  FJOURNAL = {Proceedings of the American Mathematical Society},
    VOLUME = {133},
      YEAR = {2005},
    NUMBER = {9},
     PAGES = {2805--2812},
      ISSN = {0002-9939,1088-6826},
   MRCLASS = {58E20},
  MRNUMBER = {2146230},
MRREVIEWER = {Fr\'ed\'eric\ Robert},
       DOI = {10.1090/S0002-9939-05-07818-4},
       URL = {https://doi-org.proxy.library.cornell.edu/10.1090/S0002-9939-05-07818-4},
}

@article {EellsLemaire,
    AUTHOR = {Eells, J. and Lemaire, L.},
     TITLE = {Another report on harmonic maps},
   JOURNAL = {Bull. London Math. Soc.},
  FJOURNAL = {The Bulletin of the London Mathematical Society},
    VOLUME = {20},
      YEAR = {1988},
    NUMBER = {5},
     PAGES = {385--524},
      ISSN = {0024-6093,1469-2120},
   MRCLASS = {58E20},
  MRNUMBER = {956352},
MRREVIEWER = {James\ F.\ Glazebrook},
       DOI = {10.1112/blms/20.5.385},
       URL = {https://doi-org.proxy.library.cornell.edu/10.1112/blms/20.5.385},
}

@article {EdelenDegeneration,
    AUTHOR = {Edelen, Nick},
     TITLE = {Degeneration of 7-dimensional minimal hypersurfaces which are
              stable or have a bounded index},
   JOURNAL = {Arch. Ration. Mech. Anal.},
  FJOURNAL = {Archive for Rational Mechanics and Analysis},
    VOLUME = {248},
      YEAR = {2024},
    NUMBER = {4},
     PAGES = {Paper No. 65, 73},
      ISSN = {0003-9527,1432-0673},
   MRCLASS = {58E12 (49Q05 53A10)},
  MRNUMBER = {4768490},
MRREVIEWER = {Panayotis\ Vyridis},
       DOI = {10.1007/s00205-024-02003-w},
       URL = {https://doi-org.proxy.library.cornell.edu/10.1007/s00205-024-02003-w},
}

@article{CLWNon-persistence,
  title={Non-persistence of strongly isolated singularities, and geometric applications},
  author={Carlotto, Alessandro and Li, Yangyang and Wang, Zhihan},
  journal={arXiv preprint arXiv:2411.12677},
  year={2024}
}

@article {LiWangGenericRegularity,
    AUTHOR = {Li, Yangyang and Wang, Zhihan},
     TITLE = {Minimal hypersurfaces for generic metrics in dimension 8},
   JOURNAL = {Invent. Math.},
  FJOURNAL = {Inventiones Mathematicae},
    VOLUME = {240},
      YEAR = {2025},
    NUMBER = {3},
     PAGES = {1193--1303},
      ISSN = {0020-9910,1432-1297},
   MRCLASS = {53A10 (49Q05 58E12)},
  MRNUMBER = {4902162},
MRREVIEWER = {Neilha\ M.\ Pinheiro},
       DOI = {10.1007/s00222-025-01333-0},
       URL = {https://doi-org.proxy.library.cornell.edu/10.1007/s00222-025-01333-0},
}

@misc{WangDeformationI,
      title={Deformations of Singular Minimal Hypersurfaces I, Isolated Singularities}, 
      author={Wang, Zhihan},
      year={2020},
      eprint={2011.00548},
      archivePrefix={arXiv},
      primaryClass={math.DG},
      url={https://arxiv.org/abs/2011.00548}, 
}

@article {NakajimaIndex,
    AUTHOR = {Nakajima, T\^oru},
     TITLE = {A remark on instability of harmonic maps between spheres},
   JOURNAL = {Pacific J. Math.},
  FJOURNAL = {Pacific Journal of Mathematics},
    VOLUME = {240},
      YEAR = {2009},
    NUMBER = {2},
     PAGES = {363--369},
      ISSN = {0030-8730,1945-5844},
   MRCLASS = {58E20 (35J50 53C43)},
  MRNUMBER = {2485470},
MRREVIEWER = {John\ C.\ Wood},
       DOI = {10.2140/pjm.2009.240.363},
       URL = {https://doi-org.proxy.library.cornell.edu/10.2140/pjm.2009.240.363},
}

@article{li2026optimalregularitystableharmonic,
  title={Optimal regularity of stable harmonic maps to spheres},
  author={Li, Xuanyu},
  journal={arXiv preprint arXiv:2608.20272},
  year={2026}
}

@article {BethuelApproximation,
    AUTHOR = {Bethuel, Fabrice},
     TITLE = {The approximation problem for {S}obolev maps between two
              manifolds},
   JOURNAL = {Acta Math.},
  FJOURNAL = {Acta Mathematica},
    VOLUME = {167},
      YEAR = {1991},
    NUMBER = {3-4},
     PAGES = {153--206},
      ISSN = {0001-5962,1871-2509},
   MRCLASS = {58D15},
  MRNUMBER = {1120602},
MRREVIEWER = {Martin\ Fuchs},
       DOI = {10.1007/BF02392449},
       URL = {https://doi-org.proxy.library.cornell.edu/10.1007/BF02392449},
}

@article {SacksUhlenbeck,
    AUTHOR = {Sacks, J. and Uhlenbeck, K.},
     TITLE = {The existence of minimal immersions of {$2$}-spheres},
   JOURNAL = {Ann. of Math. (2)},
  FJOURNAL = {Annals of Mathematics. Second Series},
    VOLUME = {113},
      YEAR = {1981},
    NUMBER = {1},
     PAGES = {1--24},
      ISSN = {0003-486X},
   MRCLASS = {58E12 (53C42 58E20)},
  MRNUMBER = {604040},
MRREVIEWER = {John\ C.\ Wood},
       DOI = {10.2307/1971131},
       URL = {https://doi-org.proxy.library.cornell.edu/10.2307/1971131},
}

@article {DingLiLi,
    AUTHOR = {Ding, Weiyue and Li, Jiayu and Li, Wei},
     TITLE = {Nonstationary weak limit of a stationary harmonic map
              sequence},
   JOURNAL = {Comm. Pure Appl. Math.},
  FJOURNAL = {Communications on Pure and Applied Mathematics},
    VOLUME = {56},
      YEAR = {2003},
    NUMBER = {2},
     PAGES = {270--277},
      ISSN = {0010-3640,1097-0312},
   MRCLASS = {58E20},
  MRNUMBER = {1934622},
MRREVIEWER = {Andreas\ Gastel},
       DOI = {10.1002/cpa.10058},
       URL = {https://doi-org.proxy.library.cornell.edu/10.1002/cpa.10058},
}

@book {EvansPDE,
    AUTHOR = {Evans, Lawrence C.},
     TITLE = {Partial differential equations},
    SERIES = {Graduate Studies in Mathematics},
    VOLUME = {19},
   EDITION = {Second},
 PUBLISHER = {American Mathematical Society, Providence, RI},
      YEAR = {2010},
     PAGES = {xxii+749},
      ISBN = {978-0-8218-4974-3},
   MRCLASS = {35-01},
  MRNUMBER = {2597943},
MRREVIEWER = {Diego\ M.\ Maldonado},
       DOI = {10.1090/gsm/019},
       URL = {https://doi-org.proxy.library.cornell.edu/10.1090/gsm/019},
}

@book {LinHanBook,
    AUTHOR = {Han, Qing and Lin, Fanghua},
     TITLE = {Elliptic partial differential equations},
    SERIES = {Courant Lecture Notes in Mathematics},
    VOLUME = {1},
   EDITION = {Second},
 PUBLISHER = {Courant Institute of Mathematical Sciences, New York; American
              Mathematical Society, Providence, RI},
      YEAR = {2011},
     PAGES = {x+147},
      ISBN = {978-0-8218-5313-9},
   MRCLASS = {35Jxx (35-01 35B50)},
  MRNUMBER = {2777537},
}

@book {TheoSalamon,
    AUTHOR = {B\"uhler, Theo and Salamon, Dietmar A.},
     TITLE = {Functional analysis},
    SERIES = {Graduate Studies in Mathematics},
    VOLUME = {191},
 PUBLISHER = {American Mathematical Society, Providence, RI},
      YEAR = {2018},
     PAGES = {xiv+466},
      ISBN = {978-1-4704-4190-6},
   MRCLASS = {46-01},
  MRNUMBER = {3823238},
MRREVIEWER = {Richard\ Becker},
       DOI = {10.1090/gsm/191},
       URL = {https://doi-org.proxy.library.cornell.edu/10.1090/gsm/191},
}

@article {JagerKaulMinimizing,
    AUTHOR = {J\"ager, Willi and Kaul, Helmut},
     TITLE = {Rotationally symmetric harmonic maps from a ball into a sphere
              and the regularity problem for weak solutions of elliptic
              systems},
   JOURNAL = {J. Reine Angew. Math.},
  FJOURNAL = {Journal f\"ur die Reine und Angewandte Mathematik. [Crelle's
              Journal]},
    VOLUME = {343},
      YEAR = {1983},
     PAGES = {146--161},
      ISSN = {0075-4102,1435-5345},
   MRCLASS = {58E20 (35J55)},
  MRNUMBER = {705882},
MRREVIEWER = {John\ C.\ Wood},
       DOI = {10.1515/crll.1983.343.146},
       URL = {https://doi-org.proxy.library.cornell.edu/10.1515/crll.1983.343.146},
}

@article {LinMinimizing,
    AUTHOR = {Lin, Fang-Hua},
     TITLE = {A remark on the map {$x/|x|$}},
   JOURNAL = {C. R. Acad. Sci. Paris S\'er. I Math.},
  FJOURNAL = {Comptes Rendus des S\'eances de l'Acad\'emie des Sciences.
              S\'erie I. Math\'ematique},
    VOLUME = {305},
      YEAR = {1987},
    NUMBER = {12},
     PAGES = {529--531},
      ISSN = {0249-6291},
   MRCLASS = {58E20 (53C42 81E99)},
  MRNUMBER = {916327},
MRREVIEWER = {John\ C.\ Wood},
}

@article {CoronGulliver,
    AUTHOR = {Coron, Jean-Michel and Gulliver, Robert},
     TITLE = {Minimizing {$p$}-harmonic maps into spheres},
   JOURNAL = {J. Reine Angew. Math.},
  FJOURNAL = {Journal f\"ur die Reine und Angewandte Mathematik. [Crelle's
              Journal]},
    VOLUME = {401},
      YEAR = {1989},
     PAGES = {82--100},
      ISSN = {0075-4102,1435-5345},
   MRCLASS = {58E20},
  MRNUMBER = {1018054},
MRREVIEWER = {Helmut\ Kaul},
       DOI = {10.1515/crll.1989.401.82},
       URL = {https://doi-org.proxy.library.cornell.edu/10.1515/crll.1989.401.82},
}

@article {BrezisCoronLieb,
    AUTHOR = {Brezis, Ha\"im and Coron, Jean-Michel and Lieb, Elliott H.},
     TITLE = {Harmonic maps with defects},
   JOURNAL = {Comm. Math. Phys.},
  FJOURNAL = {Communications in Mathematical Physics},
    VOLUME = {107},
      YEAR = {1986},
    NUMBER = {4},
     PAGES = {649--705},
      ISSN = {0010-3616,1432-0916},
   MRCLASS = {58E20 (81E13 82A05)},
  MRNUMBER = {868739},
MRREVIEWER = {John\ C.\ Wood},
       URL = {http://projecteuclid.org.proxy.library.cornell.edu/euclid.cmp/1104116234},
}

@article {HardtLinStability,
    AUTHOR = {Hardt, Robert and Lin, Fang-Hau},
     TITLE = {Stability of singularities of minimizing harmonic maps},
   JOURNAL = {J. Differential Geom.},
  FJOURNAL = {Journal of Differential Geometry},
    VOLUME = {29},
      YEAR = {1989},
    NUMBER = {1},
     PAGES = {113--123},
      ISSN = {0022-040X,1945-743X},
   MRCLASS = {58E20},
  MRNUMBER = {978080},
MRREVIEWER = {Thomas\ Bartsch},
       URL = {http://projecteuclid.org.proxy.library.cornell.edu/euclid.jdg/1214442637},
}

@article {McIntoshSimon,
    AUTHOR = {McIntosh, Robert and Simon, Leon},
     TITLE = {Perturbing away singularities of harmonic maps},
   JOURNAL = {Manuscripta Math.},
  FJOURNAL = {Manuscripta Mathematica},
    VOLUME = {67},
      YEAR = {1990},
    NUMBER = {2},
     PAGES = {113--124},
      ISSN = {0025-2611,1432-1785},
   MRCLASS = {58E20},
  MRNUMBER = {1042233},
MRREVIEWER = {Martin\ Fuchs},
       DOI = {10.1007/BF02568425},
       URL = {https://doi-org.proxy.library.cornell.edu/10.1007/BF02568425},
}

@article {Wangp-harmonicmaps,
    AUTHOR = {Wang, Changyou},
     TITLE = {Minimality and perturbation of singularities for certain
              {$p$}-harmonic maps},
   JOURNAL = {Indiana Univ. Math. J.},
  FJOURNAL = {Indiana University Mathematics Journal},
    VOLUME = {47},
      YEAR = {1998},
    NUMBER = {2},
     PAGES = {725--740},
      ISSN = {0022-2518,1943-5258},
   MRCLASS = {58E20},
  MRNUMBER = {1647881},
MRREVIEWER = {Martin\ Fuchs},
       DOI = {10.1512/iumj.1998.47.1434},
       URL = {https://doi-org.proxy.library.cornell.edu/10.1512/iumj.1998.47.1434},
}

@article {SimonSolomon,
    AUTHOR = {Simon, Leon and Solomon, Bruce},
     TITLE = {Minimal hypersurfaces asymptotic to quadratic cones in {${\bf
              R}^{n+1}$}},
   JOURNAL = {Invent. Math.},
  FJOURNAL = {Inventiones Mathematicae},
    VOLUME = {86},
      YEAR = {1986},
    NUMBER = {3},
     PAGES = {535--551},
      ISSN = {0020-9910,1432-1297},
   MRCLASS = {49F10 (53A10)},
  MRNUMBER = {860681},
MRREVIEWER = {Harold\ Parks},
       DOI = {10.1007/BF01389267},
       URL = {https://doi-org.proxy.library.cornell.edu/10.1007/BF01389267},
}

@article {EdelenSpolaorQuadraticcone,
    AUTHOR = {Edelen, Nick and Spolaor, Luca},
     TITLE = {Regularity of minimal surfaces near quadratic cones},
   JOURNAL = {Ann. of Math. (2)},
  FJOURNAL = {Annals of Mathematics. Second Series},
    VOLUME = {198},
      YEAR = {2023},
    NUMBER = {3},
     PAGES = {1013--1046},
      ISSN = {0003-486X,1939-8980},
   MRCLASS = {49Q05 (35J93 49Q20 53A10)},
  MRNUMBER = {4660135},
MRREVIEWER = {Annalisa\ Cesaroni},
       DOI = {10.4007/annals.2023.198.3.2},
       URL = {https://doi-org.proxy.library.cornell.edu/10.4007/annals.2023.198.3.2},
}

@article {MazetSimonsCones,
    AUTHOR = {Mazet, Laurent},
     TITLE = {Minimal hypersurfaces asymptotic to {S}imons cones},
   JOURNAL = {J. Inst. Math. Jussieu},
  FJOURNAL = {Journal of the Institute of Mathematics of Jussieu. JIMJ.
              Journal de l'Institut de Math\'ematiques de Jussieu},
    VOLUME = {16},
      YEAR = {2017},
    NUMBER = {1},
     PAGES = {39--58},
      ISSN = {1474-7480,1475-3030},
   MRCLASS = {53A10},
  MRNUMBER = {3591961},
MRREVIEWER = {Jianquan\ Ge},
       DOI = {10.1017/S1474748015000110},
       URL = {https://doi-org.proxy.library.cornell.edu/10.1017/S1474748015000110},
}

@article{EngelsteinRestrepoZhao,
  title={On the asymptotic properties of solutions to one-phase free boundary problems},
  author={Engelstein, Max and Restrepo, Daniel and Zhao, Zihui},
  journal={arXiv preprint arXiv:2511.08393},
  year={2025}
}

@article {MillotPisante,
    AUTHOR = {Millot, Vincent and Pisante, Adriano},
     TITLE = {Symmetry of local minimizers for the three-dimensional
              {G}inzburg-{L}andau functional},
   JOURNAL = {J. Eur. Math. Soc. (JEMS)},
  FJOURNAL = {Journal of the European Mathematical Society (JEMS)},
    VOLUME = {12},
      YEAR = {2010},
    NUMBER = {5},
     PAGES = {1069--1096},
      ISSN = {1435-9855,1435-9863},
   MRCLASS = {58E20 (31C45 35Q56 49K10 58J70)},
  MRNUMBER = {2677610},
MRREVIEWER = {Andreas\ Gastel},
       DOI = {10.4171/JEMS/223},
       URL = {https://doi-org.proxy.library.cornell.edu/10.4171/JEMS/223},
}

@article {HardtLinRemark,
    AUTHOR = {Hardt, Robert and Lin, Fang-Hua},
    TITLE = {A remark on {$H^1$} mappings},
   JOURNAL = {Manuscripta Math.},
  FJOURNAL = {Manuscripta Mathematica},
    VOLUME = {56},
      YEAR = {1986},
    NUMBER = {1},
     PAGES = {1--10},
      ISSN = {0025-2611,1432-1785},
   MRCLASS = {58E20},
  MRNUMBER = {846982},
MRREVIEWER = {J.\ Eells},
       DOI = {10.1007/BF01171029},
       URL = {https://doi-org.proxy.library.cornell.edu/10.1007/BF01171029},
}

@article {Mazowieckaetc,
    AUTHOR = {Mazowiecka, Katarzyna and Mi\'skiewicz, Micha\l{} and
              Schikorra, Armin},
     TITLE = {On the size of the singular set of minimizing harmonic maps},
   JOURNAL = {Mem. Amer. Math. Soc.},
  FJOURNAL = {Memoirs of the American Mathematical Society},
    VOLUME = {302},
      YEAR = {2024},
    NUMBER = {1519},
     PAGES = {v+82},
      ISSN = {0065-9266,1947-6221},
      ISBN = {978-1-4704-7162-0; 978-1-4704-7970-1},
   MRCLASS = {58E20 (35B30 35J50 35J57)},
  MRNUMBER = {4835768},
MRREVIEWER = {Christopher\ Steven\ Goodrich},
       DOI = {10.1090/memo/1519},
       URL = {https://doi-org.proxy.library.cornell.edu/10.1090/memo/1519},
}

@article {LinWangSphere,
    AUTHOR = {Lin, Fang Hua and Wang, Chang You},
     TITLE = {Stable stationary harmonic maps to spheres},
   JOURNAL = {Acta Math. Sin. (Engl. Ser.)},
  FJOURNAL = {Acta Mathematica Sinica (English Series)},
    VOLUME = {22},
      YEAR = {2006},
    NUMBER = {2},
     PAGES = {319--330},
      ISSN = {1439-8516,1439-7617},
   MRCLASS = {58E20 (35J20 35J60)},
  MRNUMBER = {2214353},
MRREVIEWER = {Futoshi\ Takahashi},
       DOI = {10.1007/s10114-005-0673-7},
       URL = {https://doi-org.proxy.library.cornell.edu/10.1007/s10114-005-0673-7},
}

@article {HsuLiregularity,
    AUTHOR = {Hsu, De Liang and Li, Jia Yu},
     TITLE = {On the regularity for stationary harmonic maps},
   JOURNAL = {Acta Math. Sin. (Engl. Ser.)},
  FJOURNAL = {Acta Mathematica Sinica (English Series)},
    VOLUME = {24},
      YEAR = {2008},
    NUMBER = {2},
     PAGES = {223--226},
      ISSN = {1439-8516,1439-7617},
   MRCLASS = {58E20},
  MRNUMBER = {2383350},
MRREVIEWER = {Andreas\ Gastel},
       DOI = {10.1007/s10114-007-1007-8},
       URL = {https://doi-org.proxy.library.cornell.edu/10.1007/s10114-007-1007-8},
}

@article {Federer,
    AUTHOR = {Federer, Herbert},
     TITLE = {The singular sets of area minimizing rectifiable currents with
              codimension one and of area minimizing flat chains modulo two
              with arbitrary codimension},
   JOURNAL = {Bull. Amer. Math. Soc.},
  FJOURNAL = {Bulletin of the American Mathematical Society},
    VOLUME = {76},
      YEAR = {1970},
     PAGES = {767--771},
      ISSN = {0002-9904},
   MRCLASS = {28.80 (26.00)},
  MRNUMBER = {260981},
MRREVIEWER = {J.\ E.\ Brothers},
       DOI = {10.1090/S0002-9904-1970-12542-3},
       URL = {https://doi-org.proxy.library.cornell.edu/10.1090/S0002-9904-1970-12542-3},
}

@article {SchoenUhlenbeckRegularity,
    AUTHOR = {Schoen, Richard and Uhlenbeck, Karen},
     TITLE = {A regularity theory for harmonic maps},
   JOURNAL = {J. Differential Geometry},
  FJOURNAL = {Journal of Differential Geometry},
    VOLUME = {17},
      YEAR = {1982},
    NUMBER = {2},
     PAGES = {307--335},
      ISSN = {0022-040X,1945-743X},
   MRCLASS = {58E20 (35J20)},
  MRNUMBER = {664498},
MRREVIEWER = {J.\ Eells},
       URL = {http://projecteuclid.org.proxy.library.cornell.edu/euclid.jdg/1214436923},
}

@article {EellsSampson,
    AUTHOR = {Eells, Jr., James and Sampson, J. H.},
     TITLE = {Harmonic mappings of {R}iemannian manifolds},
   JOURNAL = {Amer. J. Math.},
  FJOURNAL = {American Journal of Mathematics},
    VOLUME = {86},
      YEAR = {1964},
     PAGES = {109--160},
      ISSN = {0002-9327,1080-6377},
   MRCLASS = {53.72 (57.50)},
  MRNUMBER = {164306},
MRREVIEWER = {J.\ A.\ Wolf},
       DOI = {10.2307/2373037},
       URL = {https://doi-org.proxy.library.cornell.edu/10.2307/2373037},
}

@article {HildebrandtKaulWidman,
    AUTHOR = {Hildebrandt, St\'efan and Kaul, Helmut and Widman, Kjell-Ove},
     TITLE = {An existence theorem for harmonic mappings of {R}iemannian
              manifolds},
   JOURNAL = {Acta Math.},
  FJOURNAL = {Acta Mathematica},
    VOLUME = {138},
      YEAR = {1977},
    NUMBER = {1-2},
     PAGES = {1--16},
      ISSN = {0001-5962,1871-2509},
   MRCLASS = {58E15},
  MRNUMBER = {433502},
MRREVIEWER = {Jean-Claude\ Mitteau},
       DOI = {10.1007/BF02392311},
       URL = {https://doi-org.proxy.library.cornell.edu/10.1007/BF02392311},
}

@article {MikhailStern,
    AUTHOR = {Karpukhin, Mikhail and Stern, Daniel},
     TITLE = {Existence of harmonic maps and eigenvalue optimization in
              higher dimensions},
   JOURNAL = {Invent. Math.},
  FJOURNAL = {Inventiones Mathematicae},
    VOLUME = {236},
      YEAR = {2024},
    NUMBER = {2},
     PAGES = {713--778},
      ISSN = {0020-9910,1432-1297},
   MRCLASS = {53C43},
  MRNUMBER = {4728241},
MRREVIEWER = {Hezi\ Lin},
       DOI = {10.1007/s00222-024-01247-3},
       URL = {https://doi-org.proxy.library.cornell.edu/10.1007/s00222-024-01247-3},
}

@book {Morrey,
    AUTHOR = {Morrey, Jr., Charles B.},
     TITLE = {Multiple integrals in the calculus of variations},
    SERIES = {Die Grundlehren der mathematischen Wissenschaften},
    VOLUME = {Band 130},
 PUBLISHER = {Springer-Verlag New York, Inc., New York},
      YEAR = {1966},
     PAGES = {ix+506},
   MRCLASS = {49.00 (00.00)},
  MRNUMBER = {202511},
MRREVIEWER = {M.\ Schechter},
}

@incollection {SchoenAnalyticAspects,
    AUTHOR = {Schoen, Richard M.},
     TITLE = {Analytic aspects of the harmonic map problem},
 BOOKTITLE = {Seminar on nonlinear partial differential equations
              ({B}erkeley, {C}alif., 1983)},
    SERIES = {Math. Sci. Res. Inst. Publ.},
    VOLUME = {2},
     PAGES = {321--358},
 PUBLISHER = {Springer, New York},
      YEAR = {1984},
      ISBN = {0-387-96079-1},
   MRCLASS = {58E20},
  MRNUMBER = {765241},
MRREVIEWER = {Helmut\ Kaul},
       DOI = {10.1007/978-1-4612-1110-5\_17},
       URL = {https://doi-org.proxy.library.cornell.edu/10.1007/978-1-4612-1110-5_17},
}

@article {Helein,
    AUTHOR = {H\'elein, Fr\'ed\'eric},
     TITLE = {R\'egularit\'e{} des applications faiblement harmoniques entre
              une surface et une vari\'et\'e{} riemannienne},
   JOURNAL = {C. R. Acad. Sci. Paris S\'er. I Math.},
  FJOURNAL = {Comptes Rendus de l'Acad\'emie des Sciences. S\'erie I.
              Math\'ematique},
    VOLUME = {312},
      YEAR = {1991},
    NUMBER = {8},
     PAGES = {591--596},
      ISSN = {0764-4442},
   MRCLASS = {58E20},
  MRNUMBER = {1101039},
MRREVIEWER = {John\ C.\ Wood},
}

@article {EvansSphere,
    AUTHOR = {Evans, Lawrence C.},
     TITLE = {Partial regularity for stationary harmonic maps into spheres},
   JOURNAL = {Arch. Rational Mech. Anal.},
  FJOURNAL = {Archive for Rational Mechanics and Analysis},
    VOLUME = {116},
      YEAR = {1991},
    NUMBER = {2},
     PAGES = {101--113},
      ISSN = {0003-9527},
   MRCLASS = {58E20},
  MRNUMBER = {1143435},
       DOI = {10.1007/BF00375587},
       URL = {https://doi-org.proxy.library.cornell.edu/10.1007/BF00375587},
}

@article {Bethuel,
    AUTHOR = {Bethuel, Fabrice},
     TITLE = {On the singular set of stationary harmonic maps},
   JOURNAL = {Manuscripta Math.},
  FJOURNAL = {Manuscripta Mathematica},
    VOLUME = {78},
      YEAR = {1993},
    NUMBER = {4},
     PAGES = {417--443},
      ISSN = {0025-2611,1432-1785},
   MRCLASS = {58E20},
  MRNUMBER = {1208652},
MRREVIEWER = {Caio\ J. C. Negreiros},
       DOI = {10.1007/BF02599324},
       URL = {https://doi-org.proxy.library.cornell.edu/10.1007/BF02599324},
}

@article {RiviereStruwe,
    AUTHOR = {Rivi\`ere, Tristan and Struwe, Michael},
     TITLE = {Partial regularity for harmonic maps and related problems},
   JOURNAL = {Comm. Pure Appl. Math.},
  FJOURNAL = {Communications on Pure and Applied Mathematics},
    VOLUME = {61},
      YEAR = {2008},
    NUMBER = {4},
     PAGES = {451--463},
      ISSN = {0010-3640,1097-0312},
   MRCLASS = {58E20 (35B65 35J60)},
  MRNUMBER = {2383929},
MRREVIEWER = {Andreas\ Gastel},
       DOI = {10.1002/cpa.20205},
       URL = {https://doi-org.proxy.library.cornell.edu/10.1002/cpa.20205},
}

@article {ParkerBubbleTree,
    AUTHOR = {Parker, Thomas H.},
     TITLE = {Bubble tree convergence for harmonic maps},
   JOURNAL = {J. Differential Geom.},
  FJOURNAL = {Journal of Differential Geometry},
    VOLUME = {44},
      YEAR = {1996},
    NUMBER = {3},
     PAGES = {595--633},
      ISSN = {0022-040X,1945-743X},
   MRCLASS = {58E20},
  MRNUMBER = {1431008},
MRREVIEWER = {Daniel\ Pollack},
       URL = {http://projecteuclid.org.proxy.library.cornell.edu/euclid.jdg/1214459224},
}

@inproceedings {Jost,
    AUTHOR = {Jost, J\"urgen},
     TITLE = {Two-dimensional geometric variational problems},
 BOOKTITLE = {Proceedings of the {I}nternational {C}ongress of
              {M}athematicians, {V}ol. 1, 2 ({B}erkeley, {C}alif., 1986)},
     PAGES = {1094--1100},
 PUBLISHER = {Amer. Math. Soc., Providence, RI},
      YEAR = {1987},
      ISBN = {0-8218-0110-4},
   MRCLASS = {58E12 (49F10 53A10 58E20)},
  MRNUMBER = {934312},
MRREVIEWER = {John\ C.\ Wood},
       DOI = {10.1073/pnas.84.21.7453},
       URL = {https://doi-org.proxy.library.cornell.edu/10.1073/pnas.84.21.7453},
}

@article {LinRiviere,
    AUTHOR = {Lin, Fang-Hua and Rivi\`ere, Tristan},
     TITLE = {Energy quantization for harmonic maps},
   JOURNAL = {Duke Math. J.},
  FJOURNAL = {Duke Mathematical Journal},
    VOLUME = {111},
      YEAR = {2002},
    NUMBER = {1},
     PAGES = {177--193},
      ISSN = {0012-7094,1547-7398},
   MRCLASS = {58E20},
  MRNUMBER = {1876445},
MRREVIEWER = {Ernst\ C.\ Kuwert},
       DOI = {10.1215/S0012-7094-02-11116-8},
       URL = {https://doi-org.proxy.library.cornell.edu/10.1215/S0012-7094-02-11116-8},
}

@article {LinGradientEstimate,
    AUTHOR = {Lin, Fang-Hua},
     TITLE = {Gradient estimates and blow-up analysis for stationary
              harmonic maps},
   JOURNAL = {Ann. of Math. (2)},
  FJOURNAL = {Annals of Mathematics. Second Series},
    VOLUME = {149},
      YEAR = {1999},
    NUMBER = {3},
     PAGES = {785--829},
      ISSN = {0003-486X,1939-8980},
   MRCLASS = {58E20 (49Q20)},
  MRNUMBER = {1709303},
MRREVIEWER = {Harold\ Parks},
       DOI = {10.2307/121073},
       URL = {https://doi-org.proxy.library.cornell.edu/10.2307/121073},
}

@misc{naber2024energyidentitystationaryharmonic,
      title={Energy Identity for Stationary Harmonic Maps}, 
      author={Naber, Aaron and Valtorta, Daniele},
      year={2024},
      eprint={2401.02242},
      archivePrefix={arXiv},
      primaryClass={math.AP},
      url={https://arxiv.org/abs/2401.02242}, 
      howpublished = {arXiv 2401.02242},
}

@misc{Krantz,
      title={Partial Regularity of Stable Stationary Harmonic Maps into Certain Lie Groups}, 
      author={Krantz, Jacob},
      year={2026},
      eprint={2605.03809},
      archivePrefix={arXiv},
      primaryClass={math.DG},
      url={https://arxiv.org/abs/2605.03809}, 
}

@article{li2025energy,
  title={Energy identity for Ginzburg-Landau approximation of harmonic maps},
  author={Li, Xuanyu},
  journal={arXiv preprint arXiv:2503.23675},
  year={2025}
}

@article {HardtSimon,
    AUTHOR = {Hardt, Robert and Simon, Leon},
     TITLE = {Area minimizing hypersurfaces with isolated singularities},
   JOURNAL = {J. Reine Angew. Math.},
  FJOURNAL = {Journal f\"ur die Reine und Angewandte Mathematik. [Crelle's
              Journal]},
    VOLUME = {362},
      YEAR = {1985},
     PAGES = {102--129},
      ISSN = {0075-4102,1435-5345},
   MRCLASS = {49F10 (58E12)},
  MRNUMBER = {809969},
MRREVIEWER = {A.\ Averna},
       DOI = {10.1515/crll.1985.362.102},
       URL = {https://doi-org.proxy.library.cornell.edu/10.1515/crll.1985.362.102},
}

@article {DingExpander,
    AUTHOR = {Ding, Qi},
     TITLE = {Minimal cones and self-expanding solutions for mean curvature
              flows},
   JOURNAL = {Math. Ann.},
  FJOURNAL = {Mathematische Annalen},
    VOLUME = {376},
      YEAR = {2020},
    NUMBER = {1-2},
     PAGES = {359--405},
      ISSN = {0025-5831,1432-1807},
   MRCLASS = {53C42 (53E10)},
  MRNUMBER = {4055164},
MRREVIEWER = {Xu\ Cheng},
       DOI = {10.1007/s00208-019-01941-1},
       URL = {https://doi-org.proxy.library.cornell.edu/10.1007/s00208-019-01941-1},
}

@article {WangSmoothing,
    AUTHOR = {Wang, Zhihan},
     TITLE = {Mean convex smoothing of mean convex cones},
   JOURNAL = {Geom. Funct. Anal.},
  FJOURNAL = {Geometric and Functional Analysis},
    VOLUME = {34},
      YEAR = {2024},
    NUMBER = {1},
     PAGES = {263--301},
      ISSN = {1016-443X,1420-8970},
   MRCLASS = {53C42 (49Q15 49Q20 58E12)},
  MRNUMBER = {4706448},
MRREVIEWER = {Annalisa\ Cesaroni},
       DOI = {10.1007/s00039-024-00666-x},
       URL = {https://doi-org.proxy.library.cornell.edu/10.1007/s00039-024-00666-x},
}

@article {ChodoshLiokumovichSpolaor,
    AUTHOR = {Chodosh, Otis and Liokumovich, Yevgeny and Spolaor, Luca},
     TITLE = {Singular behavior and generic regularity of min-max minimal
              hypersurfaces},
   JOURNAL = {Ars Inven. Anal.},
  FJOURNAL = {Ars Inveniendi Analytica},
      YEAR = {2022},
     PAGES = {Paper No. 2, 27},
      ISSN = {2769-8505},
   MRCLASS = {49Q05 (53C42)},
  MRNUMBER = {4462477},
MRREVIEWER = {Alexis\ Michelat},
       DOI = {10.15781/j4aj-kd66},
       URL = {https://doi-org.proxy.library.cornell.edu/10.15781/j4aj-kd66},
}

@article {SmaleGenericHomological,
    AUTHOR = {Smale, Nathan},
     TITLE = {Generic regularity of homologically area minimizing
              hypersurfaces in eight-dimensional manifolds},
   JOURNAL = {Comm. Anal. Geom.},
  FJOURNAL = {Communications in Analysis and Geometry},
    VOLUME = {1},
      YEAR = {1993},
    NUMBER = {2},
     PAGES = {217--228},
      ISSN = {1019-8385,1944-9992},
   MRCLASS = {49Q05 (58E15)},
  MRNUMBER = {1243523},
MRREVIEWER = {Martin\ Fuchs},
       DOI = {10.4310/CAG.1993.v1.n2.a2},
       URL = {https://doi-org.proxy.library.cornell.edu/10.4310/CAG.1993.v1.n2.a2},
}

@article {ChodoshMantoulidisSchulze,
    AUTHOR = {Chodosh, Otis and Mantoulidis, Christos and Schulze, Felix},
     TITLE = {Generic regularity for minimizing hypersurfaces in dimensions
              9 and 10},
   JOURNAL = {Publ. Math. Inst. Hautes \'Etudes Sci.},
  FJOURNAL = {Publications Math\'ematiques. Institut de Hautes \'Etudes
              Scientifiques},
    VOLUME = {143},
      YEAR = {2026},
     PAGES = {143--188},
      ISSN = {0073-8301,1618-1913},
   MRCLASS = {53A10 (49Q05 49Q20 58A25)},
  MRNUMBER = {5082540},
}

@article{chodosh2025generic,
  title={Generic regularity for minimizing hypersurfaces in dimension 11},
  author={Chodosh, Otis and Mantoulidis, Christos and Schulze, Felix and Wang, Zhihan},
  journal={arXiv preprint arXiv:2506.12852},
  year={2025}
}

@article {ChodoshMantoulidisSchulzeImprovedBound,
    AUTHOR = {Chodosh, Otis and Mantoulidis, Christos and Schulze, Felix},
     TITLE = {Improved generic regularity of codimension-1 minimizing
              integral currents},
   JOURNAL = {Ars Inven. Anal.},
  FJOURNAL = {Ars Inveniendi Analytica},
      YEAR = {2024},
     PAGES = {Paper No. 3, 16},
      ISSN = {2769-8505},
   MRCLASS = {53A10 (49Q05 49Q20 58A25)},
  MRNUMBER = {4753940},
MRREVIEWER = {Hongbin\ Cui},
       DOI = {10.15781/z70n-ed29},
       URL = {https://doi-org.proxy.library.cornell.edu/10.15781/z70n-ed29},
}

@article {LiMinkowskiBound,
    AUTHOR = {Li, Xuanyu},
     TITLE = {Minkowski content estimates for generic area minimizing
              hypersurfaces},
   JOURNAL = {Calc. Var. Partial Differential Equations},
  FJOURNAL = {Calculus of Variations and Partial Differential Equations},
    VOLUME = {63},
      YEAR = {2024},
    NUMBER = {7},
     PAGES = {Paper No. 176, 15},
      ISSN = {0944-2669,1432-0835},
   MRCLASS = {53C42 (28A78)},
  MRNUMBER = {4772593},
MRREVIEWER = {Annalisa\ Cesaroni},
       DOI = {10.1007/s00526-024-02791-9},
       URL = {https://doi-org.proxy.library.cornell.edu/10.1007/s00526-024-02791-9},
}

@article {FigalliRosotonSerra,
    AUTHOR = {Figalli, Alessio and Ros-Oton, Xavier and Serra, Joaquim},
     TITLE = {Generic regularity of free boundaries for the obstacle
              problem},
   JOURNAL = {Publ. Math. Inst. Hautes \'Etudes Sci.},
  FJOURNAL = {Publications Math\'ematiques. Institut de Hautes \'Etudes
              Scientifiques},
    VOLUME = {132},
      YEAR = {2020},
     PAGES = {181--292},
      ISSN = {0073-8301,1618-1913},
   MRCLASS = {35R35},
  MRNUMBER = {4179834},
       DOI = {10.1007/s10240-020-00119-9},
       URL = {https://doi-org.proxy.library.cornell.edu/10.1007/s10240-020-00119-9},
}

@article {XavierTorres,
    AUTHOR = {Fern\'andez-Real, Xavier and Torres-Latorre, Clara},
     TITLE = {Generic regularity of free boundaries for the thin obstacle
              problem},
   JOURNAL = {Adv. Math.},
  FJOURNAL = {Advances in Mathematics},
    VOLUME = {433},
      YEAR = {2023},
     PAGES = {Paper No. 109323, 29},
      ISSN = {0001-8708,1090-2082},
   MRCLASS = {35R35},
  MRNUMBER = {4649881},
       DOI = {10.1016/j.aim.2023.109323},
       URL = {https://doi-org.proxy.library.cornell.edu/10.1016/j.aim.2023.109323},
}

@article{fernandez2023generic,
  title={Generic properties in free boundary problems},
  author={Fern{\'a}ndez-Real, Xavier and Yu, Hui},
  journal={arXiv preprint arXiv:2308.13209},
  year={2023}
}

@article{fernandez2025linearized,
  title={Linearized equation and generic regularity in the Alt-Caffarelli problem},
  author={Fern{\'a}ndez-Real, Xavier and Yu, Hui},
  journal={arXiv preprint arXiv:2510.18330},
  year={2025}
}

@article {WhiteBumpyMetric,
    AUTHOR = {White, Brian},
     TITLE = {The space of minimal submanifolds for varying {R}iemannian
              metrics},
   JOURNAL = {Indiana Univ. Math. J.},
  FJOURNAL = {Indiana University Mathematics Journal},
    VOLUME = {40},
      YEAR = {1991},
    NUMBER = {1},
     PAGES = {161--200},
      ISSN = {0022-2518,1943-5258},
   MRCLASS = {58D10 (53C42)},
  MRNUMBER = {1101226},
MRREVIEWER = {Jo\~ao\ Lucas Marques Barbosa},
       DOI = {10.1512/iumj.1991.40.40008},
       URL = {https://doi-org.proxy.library.cornell.edu/10.1512/iumj.1991.40.40008},
}

@article{gutwein2026deformations,
  title={Deformations of harmonic maps with conical singularities},
  author={Gutwein, Dominik and Langlais, Thibault},
  journal={arXiv preprint arXiv:2609.20588},
  year={2026}
}

@article {AlmgrenLieb,
    AUTHOR = {Almgren, Jr., Frederick J. and Lieb, Elliott H.},
     TITLE = {Singularities of energy minimizing maps from the ball to the
              sphere: examples, counterexamples, and bounds},
   JOURNAL = {Ann. of Math. (2)},
  FJOURNAL = {Annals of Mathematics. Second Series},
    VOLUME = {128},
      YEAR = {1988},
    NUMBER = {3},
     PAGES = {483--530},
      ISSN = {0003-486X,1939-8980},
   MRCLASS = {58E20},
  MRNUMBER = {970609},
MRREVIEWER = {John\ C.\ Wood},
       DOI = {10.2307/1971434},
       URL = {https://doi-org.proxy.library.cornell.edu/10.2307/1971434},
}

@article{JiangLiYu,
  title={Prescribed singular sets for stationary harmonic maps},
  author={Jiang, Wenshuai and Li, Chang and Yu, Wenyou},
  journal={arXiv preprint arXiv:2609.21443},
  year={2026}
}
\bibliographystyle{plain}
\end{document}